\documentclass[a4paper,11pt,reqno]{amsart}
\usepackage[utf8]{inputenc}
\usepackage{amsmath}
\usepackage{amssymb}
\usepackage{amsthm}
\usepackage{mathtools}
\usepackage{hyperref}
\usepackage{color}
\usepackage{enumerate}
\usepackage{a4wide}
\usepackage[all]{xy}
\usepackage{tikz-cd}

\usepackage{amsmath}

\makeatletter
\newcommand{\doublewidetilde}[1]{{%
  \mathpalette\double@widetilde{#1}%
}}
\newcommand{\double@widetilde}[2]{%
  \sbox\z@{$\m@th#1\widetilde{#2}$}%
  \ht\z@=.9\ht\z@
  \widetilde{\box\z@}%
}
\makeatother

\DeclareMathOperator{\PGL}{PGL}

\DeclareMathOperator{\upO}{O}

\newcommand{\CC}{\mathbb{C}}

\newcommand{\EE}{\mathbb{E}}
\newcommand{\FF}{\mathbb{F}}

\newcommand{\RR}{\mathbb{R}}
\newcommand{\ZZ}{\mathbb{Z}}

\newcommand{\calD}{\mathcal{D}}

\newcommand{\calF}{\mathcal{F}}

\newcommand{\calO}{\mathcal{O}}

\newcommand{\calV}{\mathcal{V}}
\newcommand{\calW}{\mathcal{W}}

\newcommand{\0}{\textbf{0}}
\renewcommand{\1}{{\rm\bf 1}}

\DeclareMathOperator{\Ind}{Ind}

\DeclareMathOperator{\Hom}{Hom}

\DeclareMathOperator{\const}{const}
\DeclareMathOperator{\diag}{diag}

\DeclareMathOperator{\supp}{supp}

\DeclareMathOperator{\ev}{ev}

\renewcommand\Re{\operatorname{Re}}
\renewcommand\Im{\operatorname{Im}}

\DeclareMathOperator{\Fsymp}{\calF_{\textup{symp}}}

\newcommand{\SplusT}{\;\backslash\kern-0.8em{\backslash} \;}
\newcommand{\SminusT}{\; /\kern-0.8em{/} \;}

\renewcommand{\setminus}{-}

\theoremstyle{plain}
\newtheorem{theorem}{Theorem}[section]
\newtheorem*{theorem*}{Theorem}
\newtheorem{proposition}[theorem]{Proposition}
\newtheorem{lemma}[theorem]{Lemma}
\newtheorem{corollary}[theorem]{Corollary}

\newtheorem*{fact*}{Fact}

\newtheorem{thmalph}{Theorem}

\theoremstyle{definition}
\newtheorem{definition}[theorem]{Definition}

\newtheorem{remark}[theorem]{Remark}

\numberwithin{equation}{section}

\definecolor{back}{RGB}{255,255,255}
\definecolor{fore}{RGB}{0,0,0}
\definecolor{title}{RGB}{255,0,90}

\definecolor{green}{rgb}{0.0, 0.5, 0.0}
\definecolor{purple}{rgb}{0.5, 0.0, 0.5}
\definecolor{bluegreen}{rgb}{0.0,0.5, 0.5}
\definecolor{orange}{rgb}{1,0.5, 0.1}
\definecolor{redgreen}{rgb}{0.5, 0.5, 0.0}

\def\green{\color{green}}

\def\green{\color{green}}

\def\g2{{\green 2}}

\title{Symmetry breaking for $\PGL(2)$ over non-archimedean local fields}

\author{Corina Ciobotaru}
\address{Aarhus Institute of Advanced Studies and Department of Mathematics, Aarhus University, Ny Munkegade 118, 8000 Aarhus C, Denmark}
\email{cociobotaru@aias.au.dk}

\author{Jan Frahm}
\address{Department of Mathematics, Aarhus University, Ny Munkegade 118, 8000 Aarhus C, Denmark}
\email{frahm@math.au.dk}

\begin{document}



\begin{abstract}
For a quadratic extension $\EE/\FF$ of non-archimedean local fields we construct explicit holomorphic families of intertwining operators between principal series representations of $\PGL(2,\EE)$ and $\PGL(2,\FF)$, also referred to as \emph{symmetry breaking operators}. These families are given in terms of their distribution kernels which can be viewed as distributions on $\EE$ depending holomorphically on the principal series parameters. For all such parameters we determine the support of these distributions, and we study their mapping properties. This leads to a classification of all intertwining operators between principal series representations, not necessarily irreducible. As an application, we show that every Steinberg representation of $\PGL(2,\EE)$ contains a Steinberg representation of $\PGL(2,\FF)$ as a direct summand of Hilbert spaces.
\end{abstract}

\maketitle


\section*{Introduction}

For a quadratic extension $\EE/\FF$ of local fields, the pair $(G,G')=(\PGL(2,\EE),\PGL(2,\FF))$ of reductive groups has the strong Gelfand property, i.e.
$$ \dim\Hom_{G'}(\pi|_{G'},\tau) \leq 1 $$
for all irreducible smooth admissible representations $\pi$ of $G$ and $\tau$ of $G'$, additionally assumed to be of moderate growth in the archimedean case. For non-archimedean fields, this result is due to D.~Prasad~\cite{Pra92} where he also determines for which representations the multiplicities are equal to $1$ in terms of epsilon factors.

More recently, for the archimedean case, T.~Kobayashi~\cite{Kob15} advocated to not only study the abstract multiplicities $\dim\Hom_{G'}(\pi|_{G'},\tau)$, but to explicitly construct and classify intertwining operators in $\Hom_{G'}(\pi|_{G'},\tau)$, which he also refers to as \emph{symmetry breaking operators}. Together with B.~Speh he studied the pair of real reductive groups $(G,G')=(\upO(1,n+1),\upO(1,n))$ in great detail, see \cite{KS15,KS18} and also \cite{Cle16,Cle17,FW20} for similar results for other pairs of groups. The key players in their construction are meromorphic families of symmetry breaking operators between principal series representations. They construct and study these families in terms of their distribution kernels which they identify with distributions on $\RR^n$, determine the optimal holomorphic renormalization of these families, their supports and the mapping properties of the corresponding symmetry breaking operators using mostly techniques from analysis. Such results have applications to, for instance, branching laws for unitary representations~\cite{Wei21}, conformal geometry~\cite{JO20}, partial differential equations~\cite{FOZ20,MOZ16} and automorphic forms~\cite{BR10,FS18}.

The aim of this paper is to carry out a similar analysis for the pair of reductive groups $(G,G')=(\PGL(2,\EE),\PGL(2,\FF))$ where $\EE/\FF$ is a quadratic extension of non-archimedean local fields. For this, we first establish a similar correspondence between intertwining operators between principal series of $G$ and $G'$ and distributions/generalized functions on the field $\EE$ which have certain invariance properties. Next, we construct explicit holomorphic families of such invariant distributions for all principal series representations of $G$ and $G'$. We further determine for all principal series parameters the support of these distributions and also study mapping properties of the corresponding symmetry breaking operators. This leads to a full classification of symmetry breaking operators between all principal series of $G$ and $G'$, not necessarily irreducible. Generically, the space of symmetry breaking operators is one-dimensional, as predicted by the results of D.~Prasad~\cite{Pra92}. However, for some discrete set of parameters where at least one of the principal series representations is reducible, the dimension is two, and we find an explicit basis in these cases.

Let us describe our results in more detail.

\subsection*{Symmetry breaking operators between principal series}

Let $\EE$ be a non-archimedean local field, $\calO_\EE$ its ring of integers and $\calO_\EE^\times$ the group of units. Write $q_\EE$ for the cardinality of the residue field of $\EE$, $\nu_\EE:\EE\to\ZZ\cup\{\infty\}$ for the valuation of $\EE$, and choose a uniformizer $\varpi_\EE$ for $\calO_\EE$. For a multiplicative character $\chi\in\widehat{\calO_\EE^\times}$ of $\calO_\EE^\times$ and $s\in\CC/\frac{2\pi i}{\ln q_\EE}\ZZ$ we denote by $\chi_s$ the multiplicative character of $\EE^\times$ given by
$$ \chi_s(x) = \chi(\varpi_\EE^{-\nu_\EE(x)}x)|x|_\EE^s \qquad (x\in\EE^\times). $$

The character $\chi_s$ gives rise to a character of the Borel subgroup $B$ of $G=\PGL(2,\EE)$ consisting of upper triangular matrices by
$$ \chi_s\begin{pmatrix}a&b\\0&1\end{pmatrix} = \chi_s(a) \qquad (a\in\EE^\times,b\in\EE). $$
Denote by $\pi_{\chi,s}$ the corresponding principal series representations induced from $\chi_s$ (smooth normalized parabolic induction):
$$ \pi_{\chi,s} = \Ind_B^G(\chi_s). $$

For a quadratic extension $\EE/\FF$ of non-archimedean local fields we consider the pair of reductive groups $(G,G')=(\PGL(2,\EE),\PGL(2,\FF))$. We write $\tau_{\eta,t}$ for the principal series representation of $G'$ with parameters $\eta\in\widehat{\calO_\FF^\times}$ and $t\in\CC/\frac{2\pi i}{\ln q_\FF}\ZZ$. The main object of study in this paper is the space
$$ \Hom_{G'}(\pi_{\chi,s}|_{G'},\tau_{\eta,t}) $$
of intertwining operators between principal series of $G$ and $G'$, also referred to as \emph{symmetry breaking operators} in the archimedean setting.

To construct and classify such operators explicitly, we give an equivalent description in terms of their distribution kernels (see Propositions~\ref{prop:SchwartzKernel} and \ref{prop:GeneralCharSBOKernels}). More precisely, a symmetry breaking operator $A\in\Hom_{G'}(\pi_{\chi,s}|_{G'},\tau_{\eta,t})$ corresponds to a distribution $K\in\calD'(\EE)$ on $\EE$ such that (in the distribution sense)
$$ Af(\overline{n}_y) = \int_\EE K(x)f(\overline{n}_{x+y})\,dx \qquad (f\in\Ind_B^G(\chi_s),y\in\FF), $$
where $\EE\to\overline{N},\,x\mapsto\overline{n}_x$ parameterizes the unipotent radical $\overline{N}$ of a Borel subgroup of $G$ opposite to $B$. This establishes a bijection
$$ \Hom_{G'}(\pi_{\chi,s}|_{G'},\tau_{\eta,t}) \to \calD'(\EE)_{s,t}^{\chi,\eta}\subseteq\calD'(\EE), \quad A\mapsto K $$
onto a subspace $\calD'(\EE)_{s,t}^{\chi,\eta}$ of $\calD'(\EE)$ consisting of distributions satisfying certain equivariance properties (see Proposition~\ref{prop:CharSBOKernels}). We call these distributions \emph{symmetry breaking kernels}.

\subsection*{Construction and classification of symmetry breaking kernels}

The main ingredient in the description of $\calD'(\EE)_{s,t}^{\chi,\eta}$ is a family of distributions $K_{s,t}^{\chi,\eta}$ that depends meromorphically on $(s,t)\in\CC$ and is contained in $\calD'(\EE)_{s,t}^{\chi,\eta}$ whenever it is defined. More precisely, we put
\begin{equation*}
	K_{s,t}^{\chi,\eta}(x) :=  \chi_{s-\frac{1}{2}}(x^2)\cdot(\chi_{s-\frac{1}{2}}\eta_{t+\frac{1}{2}})\left(\frac{x_2}{x_1^2-x_2^2\alpha^2}\right), \qquad \mbox{where } x=x_1+x_2\alpha \in \EE\setminus\FF.
\end{equation*}
This function is locally integrable on $\EE$ whenever $\Re(2s\pm t+\frac{1}{2})>0$ (see Lemma~\ref{lem:KernelL1loc}) and therefore defines a family of distributions on $\EE$. The main result of Section~\ref{sec::hol_sym_fam} is its meromorphic extension to all $(s,t)\in\CC^2$. To state this result, let
\begin{align*}
	\SplusT &= \{(\chi,s,\eta,t)\in\widehat{\calO_\EE^\times}\times\CC\times\widehat{\calO_\FF^\times}\times\CC:\chi|_{\calO_\FF^\times}\cdot\eta=1,2s+t+\tfrac{1}{2}\in\tfrac{2\pi i}{\ln q_\FF}\ZZ\},\\
	\SminusT &= \{(\chi,s,\eta,t)\in\widehat{\calO_\EE^\times}\times\CC\times\widehat{\calO_\FF^\times}\times\CC:\chi|_{\calO_\FF^\times}\cdot\eta^{-1}=1,2s-t+\tfrac{1}{2}\in\tfrac{2\pi i}{\ln q_\FF}\ZZ\}.
\end{align*}
We call a family of distributions meromorphic/holomorphic if its pairing with every test function is meromorphic/holomorphic.

\begin{thmalph}[see Theorem \ref{thm:HolomorphicRenormalization}]\label{thm:IntroA}
	The family of distributions $K_{s,t}^{\chi,\eta}\in\calD'(\EE)$ has a meromorphic extension to $(s,t)\in\CC^2$ for all $\chi$ and $\eta$ with simple poles at $(\chi,s,\eta,t)\in\SplusT\cup\SminusT$. For $(\chi,s,\eta,t)\not\in\SplusT\cup\SminusT$ we have $K_{s,t}^{\chi,\eta}\in\calD'(\EE)_{s,t}^{\chi,\eta}$.
\end{thmalph}

Since $\SplusT$ resp. $\SminusT$ is precisely the zero set of the local non-archimedean $L$-factor $L(\frac{1}{2},\chi_s|_{\FF^\times}\cdot\eta_t)$ resp. $L(\frac{1}{2},\chi_s|_{\FF^\times}\cdot\eta_t^{-1})$ (see \eqref{eq:DefLocalLFactor} for the definition), the renormalized family
$$ \widetilde{K}_{s,t}^{\chi,\eta} := L(\tfrac{1}{2},\chi_s|_{\FF^\times}\cdot\eta_t)^{-1}L(\tfrac{1}{2},\chi_s|_{\FF^\times}\cdot\eta_t^{-1})^{-1}\cdot K_{s,t}^{\chi,\eta} $$
is holomorphic in $(s,t)\in\CC^2$ and satisfies $\widetilde{K}_{s,t}^{\chi,\eta}\in\calD'(\EE)_{s,t}^{\chi,\eta}$ for all parameters $(\chi,s,\eta,t)$. Its values at $(\chi,s,\eta,t)\in\SplusT\cup\SminusT$ are obtained explicitly (see the residue identities in Theorem \ref{thm:ResidueIdentities}). This allows us to determine the set of parameters $(\chi,s,\eta,t)$ for which $\widetilde{K}_{s,t}^{\chi,\eta}=0$ and also the support of the distribution $\widetilde{K}_{s,t}^{\chi,\eta}\in\calD'(\EE)$ for all parameters.

\begin{thmalph}[see Theorem \ref{thm:ZerosAndSupportOfKernel}]\label{thm:IntroB}
$\widetilde{K}_{s,t}^{\chi,\eta}=0$ if and only if $(\chi,s,\eta,t)\in L$ with the discrete set $L\subseteq\SminusT$ given by \eqref{equ::L}. Moreover, we have
$$ \supp\widetilde{K}_{s,t}^{\chi,\eta} = \begin{cases}\emptyset&\mbox{for $(\chi,s,\eta,t)\in L$,}\\\{0\}&\mbox{for $(\chi,s,\eta,t)\in\SminusT\setminus L$,}\\\FF&\mbox{for $(\chi,s,\eta,t)\in\SplusT\setminus\SminusT$,}\\\EE&\mbox{else.}\end{cases} $$
\end{thmalph}

Theorems~\ref{thm:IntroA} and \ref{thm:IntroB} are proven by explicitly evaluating the distribution $K_{s,t}^{\chi,\eta}$ on a test function using elementary computations of $p$-adic integrals (see in particular Section~\ref{sec::mero_family}).

Using the family $\widetilde{K}_{s,t}^{\chi,\eta}$, we determine in Section \ref{sec::class_sym} the space $\calD'(\EE)_{s,t}^{\chi,\eta}$ of symmetry breaking kernels, and hence the space $\Hom_{G'}(\pi_{\chi,s}|_{G'},\tau_{\eta,t})$ of symmetry breaking operators, for all parameters. For $(\chi,s,\eta,t)\not\in L$ the kernel $K_{s,t}^{\chi,\eta}$ is a non-zero element of $\calD'(\EE)_{s,t}^{\chi,\eta}$. For $(\chi,s,\eta,t)\in L$ we can regularize the family $(s',t')\mapsto\widetilde{K}_{s',t'}^{\chi,\eta}$ along two different hyperplanes in $\CC^2$ to obtain two linearly independent kernels. Regularizing along $\SminusT$, i.e. along $\{(s',t'):2s'-t'+\frac{1}{2}\in\frac{2\pi i}{\ln q_\FF}\ZZ\}\subseteq\CC^2$, yields the Dirac delta distribution $\delta\in\calD'(\EE)$ at $0$. Regularizing along the hyperplane $\CC\times\{t\}\subseteq\CC^2$ yields a symmetry breaking kernel
$$ \doublewidetilde{K}_{s,t}^{\chi,\eta} := \left.(1-q_\FF^{-2z})^{-1}\widetilde{K}_{s+z,t}^{\chi,\eta}\right|_{z=0}. $$
whose support equals $\EE$. In particular, $\delta$ and $\doublewidetilde{K}_{s,t}^{\chi,\eta}$ are linearly independent.

\begin{thmalph}[see Corollary \ref{cor::dim_of D_E_chi_eta}]\label{thm:IntroC}
$$ \Hom_{G'}(\pi_{\chi,s}|_{G'},\tau_{\eta,t}) \simeq\calD'(\EE)_{s,t}^{\chi,\eta} = \begin{cases}\CC\widetilde{K}_{s,t}^{\chi,\eta}&\mbox{for $(\chi,s,\eta,t)\notin L$,}\\\CC\delta\oplus\CC\doublewidetilde{K}_{s,t}^{\chi,\eta}&\mbox{for $(\chi,s,\eta,t)\in L$.}\end{cases} $$
\end{thmalph}

Note that this is not a contradiction to the multiplicity-one property of the pair $(G,G')$ since for all $(\chi,s,\eta,t)\in L$ at least one of the representations $\pi_{\chi,s}$ and $\tau_{\eta,t}$ is reducible.


\subsection*{Mapping properties of symmetry breaking operators}

For $\chi^2=1$, the representation $\pi_{\chi,s}$ contains a unique one-dimensional subspace invariant under the maximal compact subgroup $K=\PGL(2,\calO_\EE)$. Let $\varphi_{\chi,s}$ be the unique element of this subspace with $\varphi_{\chi,s}(1)=1$. 

If additionally $s\in-\frac{1}{2}+\frac{\pi i}{\ln q_\EE}\ZZ$, then the subspace spanned by $\varphi_{\chi,s}$ is also $G$-invariant. The quotient of $\pi_{\chi,s}$ by this one-dimensional subrepresentation is irreducible and isomorphic to the unique irreducible subrepresentation $\pi_{\chi,-s}^0$ of $\pi_{\chi,-s}$. It is also referred to as \emph{Steinberg representation}. All other representations $\pi_{\chi,s}$ are irreducible (see Theorem~\ref{thm:CompositionSeries} for details).

We use the notation $\psi_{\eta,t}$ for the corresponding vector in $\tau_{\eta,t}$ if $\eta^2=1$ and $\tau_{\eta,t}^0$ for the Steinberg representation which occurs as subrepresentation of $\tau_{\eta,t}$ for $\eta^2=1$ and $t\in\frac{1}{2}+\frac{\pi i}{\ln q_\FF}\ZZ$.

In Section \ref{sec::mapp_prop} we compute the action of the symmetry breaking operator $\widetilde{A}_{s,t}^{\chi,\eta}$ with kernel $\widetilde{K}_{s,t}^{\chi,\eta}$ on the vectors $\varphi_{\chi,s}$ and also find the image of this operator in some cases.

\begin{thmalph}[see Corollary \ref{cor::sph_vectors} and Theorem \ref{thm::long_image}]\label{thm:IntroD}
\begin{enumerate}
	\item Assume $\chi^2=1$, then $\widetilde{A}_{s,t}^{\chi,\eta}\varphi_{\chi,s}=0$ unless $\eta=\chi|_{\calO_\FF^\times}$ in which case we have
	$$ \widetilde{A}_{s,t}^{\chi,\eta}\varphi_{\chi,s} = (1-q_\FF^{-1})\psi_{\eta,t}\times\begin{cases}(1-q_\EE^{-(2s+1)})&\mbox{for $\EE/\FF$ unramified,}\\(1-q_\EE^{-(2s+1)})(1+q_\FF^{t-\frac{1}{2}})&\mbox{for $\EE/\FF$ ramified.}\end{cases} $$
	\item For all parameters $(\chi,s,\eta,t)$ with $\chi|_{\calO_\FF^\times}\cdot\eta=\chi|_{\calO_\FF^\times}\cdot\eta^{-1}=1$, the image of the operator $\widetilde{A}_{s,t}^{\chi,\eta}:\pi_{\chi,s}|_{G'}\to\tau_{\eta,t}$ is given in Theorem \ref{thm::long_image}.
\end{enumerate}
\end{thmalph}

Similar statements for the symmetry breaking operators with distribution kernel $\delta$ and $\doublewidetilde{K}_{s,t}^{\chi,\eta}$ in the case $(\chi,s,\eta,t)\in L$ can be found in Remark~\ref{rem:ActionOnSphVectorForCandDoubleTildeA} as well as Propositions~\ref{prop::im_C} and \ref{prop::im_A}.

\subsection*{Application: Discrete components in the restriction of unitary representations}

Finally, in Section \ref{sec::app} we study the restriction of some unitary representations of $G$ to $G'$. In particular, we are interested in unitary representations of $G'$ that occur as direct summands in the restriction.

For $\chi^2=1$ and $s\in(0,\frac{1}{2})$, the representations $\pi_{\chi,s}$ are irreducible and unitarizable, and their completions (also denoted by $\pi_{\chi,s}$) make up the \emph{complementary series}. At the endpoint $s=\frac{1}{2}$, the Steinberg representation $\pi_{\chi,\frac{1}{2}}^0$, which is a proper subrepresentation of $\pi_{\chi,\frac{1}{2}}$, is also irreducible and unitarizable. $\pi_{\chi,\frac{1}{2}}^0$ is in fact a discrete series representation for $G$, i.e. it occurs discretely in the Plancherel decomposition of $L^2(G)$.

\begin{thmalph}[see Theorem \ref{thm:DiscreteComponents}]\label{thm:IntroE}
	Assume that $\chi^2=1$ and $\eta=\chi_{\calO_\FF^\times}$.
	\begin{enumerate}
		\item For $s\in(\frac{1}{4},\frac{1}{2})$ and $t=2s-1$, the restriction of the irreducible unitary representation $\pi_{\chi,s}$ of $G$ to $G'$ contains $\tau_{\eta,t}$ as a direct summand.
		\item For $s=t=\frac{1}{2}$, the restriction of the irreducible unitary representation $\pi_{\chi,\frac{1}{2}}^0$ of $G$ to $G'$ contains $\tau_{\eta,\frac{1}{2}}^0$ as a direct summand.
	\end{enumerate}
\end{thmalph}

For the archimedean case $\EE/\FF=\CC/\RR$, the corresponding statement is due to Speh--Venkataramana~\cite{SV11}. Our proof follows essentially the same idea.

\subsection*{Relation to previous work and outlook}

Most of this paper is inspired by the monograph \cite{KS15} by Kobayashi--Speh which treats in detail the same problem for spherical principal series representations of the pair of real reductive groups $(G,G')=(\upO(1,n+1),\upO(1,n))$. For instance, our Theorems~\ref{thm:IntroA} and \ref{thm:IntroB} are the analog of \cite[Theorem 1.5]{KS15}, Theorem~\ref{thm:IntroC} corresponds to \cite[Theorem 1.9]{KS15}, and Theorem~\ref{thm:IntroD} is the counterpart of \cite[Theorems 1.10 and 1.11]{KS15}. Note that the case $n=2$ essentially corresponds to $(\PGL(2,\CC),\PGL(2,\RR))$.

In their subsequent book \cite{KS18} they also study the case of unramified principal series representations, but it seems that there is no complete classification of symmetry breaking operators between not necessarily unramified principal series of $(G,G')=(\PGL(2,\CC),\PGL(2,\RR))$, which would be the archimedean counterpart of our results. It is reasonable that such a classification could be achieved using the same methods.

We also remark that there are more sophisticated methods to study multiplicities between irreducible representations of $p$-adic groups. However, the point of this paper is to present an elementary and rather explicit approach to the problem of constructing and classifying symmetry breaking operators between principal series representations, using similar methods as those that have been applied in the archimedean situation in the literature (see e.g. \cite{Cle16, Cle17,FW20,KS15,KS18}).

A possible application of the results of this paper is the analytic continuation of branching laws for unitary representations of the pair $(G,G')=(\PGL(2,\EE),\PGL(2,\FF))$. In the archimedean case, this was done e.g. in \cite{Wei21}. We hope to return to this in a future work.

\subsection*{Acknowledgements}

We thank Paul Nelson for his help with the evaluation of the integral in Appendix~\ref{app:Integral}. The first named author was partially supported by The Mathematisches Forschungsinstitut Oberwolfach (MFO, Oberwolfach Research Institute for Mathematics). She would like to thank MFO for the perfect working conditions they provide. As well, she was supported by the European Union’s Horizon 2020 research and innovation program under the Marie Sklodowska-Curie grant agreement No 754513, The Aarhus University Research Foundation, and a research grant from the Villum Foundation (Grant No. VIL53023). The second named author was partially supported by a research grant from the Villum Foundation (Grant No. 00025373).

\subsection*{Notation}

For a complex vector space $V$ we denote by $V^\vee=\Hom_\CC(V,\CC)$ its dual space. If $\pi$ is a representation of a group $G$ on $V$, then $\pi^\vee$ denotes the contragredient representation of $G$ on $V^\vee$. For two sets $X$ and $Y$ we denote by $X\setminus Y$ their set-theoretic difference.

\section{Symmetry breaking operators and distribution kernels}\label{sec:SBOs}

We generalize the characterization in \cite[Chapter 3]{KS15} of symmetry breaking operators between induced representations in terms of their distribution kernels from the case of real reductive groups to the case of reductive groups over arbitrary local fields.

\subsection{Distributional sections}

For a locally compact group $G$ let $d^\ell_Gg$ resp. $d^r_Gg$ denote a left resp. right Haar measure on $G$. In the case where $G$ is unimodular, we usually drop the superscript and write $d_Gg$ or simply $dg$.

Now let $G$ be an algebraic group defined over a local field $\FF$ and let $H$ be an algebraic subgroup. For a finite-dimensional continuous complex representation $(\rho_V, V)$ of $H$ let 
$$\calV:=G\times_H V:= (G\times V)/_{(g,v) \sim (gh,\rho_V(h^{-1})v)}$$
denote the corresponding complex homogeneous vector bundle over $G/H$. Note that a section of $\calV$ can be identified with a function $f:G\to V$ such that $f(gh)=\rho_V(h)^{-1}f(g)$ for all $g\in G$, $h\in H$. We will write $C(G/H,\calV)$ for the space of continuous sections of $\calV$ and $C_c(G/H,\calV)$ for the subspace of $C(G/H,\calV)$ of sections with compact support.

Moreover, we write $C^\infty(G/H,\calV)$ for the space of smooth sections in the case where $\FF$ is archimedean and locally constant sections in the case where $\FF$ is non-archimedean. Finally, put $C_c^\infty(G/H,\calV):=C^\infty(G/H,\calV)\cap C_c(G/H,\calV)$. If $\FF$ is archimedean, we endow $C^\infty(G/H,\calV)$ with the usual Fr\'{e}chet topology of compact convergence of all derivatives, and $C_c^\infty(G/H,\calV)$ with the LF topology as an inductive limit of the subspaces of sections with support contained in a fixed compact set. 

The group $G$ acts on both $C(G/H,\calV)$, $C_c(G/H,\calV)$, $C^\infty(G/H,\calV)$ and $C_c^\infty(G/H,\calV)$ by left translation in the following way: for each $g \in G$ and any $f \in C(G/H,\calV)$ we denote $(g\cdot f)(xH):= f(g^{-1}xH)$, for any $xH \in G/H$.

We want to extend the above $G$-action on smooth sections to distributional sections. To do that, let
$$ \Delta_{G/H}:=\frac{(\Delta_G)|_H}{\Delta_H}: H \to \RR_{+}, $$
where $\Delta_G$ and $\Delta_H$ are the modular functions of $G$ and $H$. Consider the line bundle 
$$\Omega_{G/H}:=G\times_H\CC_{\Delta_{G/H}}:=(G\times \CC)/_{(g,z) \sim (gh,\Delta_{G/H}(h^{-1})z)}\text{ over } G/H.$$
The following theorem provides a $G$-invariant integral on $C_c(G/H,\Omega_{G/H})$.



\begin{theorem}
\label{thm::invaraint integral}
Let $G$ be a locally compact group and $H$ a closed subgroup. Then the mapping
$$ C_c(G) \to C_c(G/H,\Omega_{G/H}), \quad f\mapsto \tilde{f}, \quad \tilde{f}(x) = \int_{H}f(xh)\Delta_{G/H}(h)d^{\ell}_Hh \quad (x\in G),$$
is surjective. Moreover, the linear map
$$ C_c(G/H,\Omega_{G/H}) \to \CC, \quad \tilde{f} \mapsto \int_{G/H}\tilde{f}(x)\,d_{G/H}(xH) := \int_G f(g) \,d_G^{\ell}g $$
is well-defined, giving a unique (up to scalar multiples) $G$-invariant integral on $C_c(G/H,\Omega_{G/H})$. In particular, $d_{G/H}$ satisfies the iterated integral formula
$$ \int_G f(g)\,d^{\ell}_Gg = \int_{G/H} \left( \int_H f(xh)\Delta_{G/H}(h)\,d^{\ell}_H h \right)d_{G/H}(xH) \qquad (f\in C_c(G)). $$
\end{theorem}

The proof of Theorem \ref{thm::invaraint integral} is essentially identical to the proof for the case $\Delta_{G/H}\equiv1$  which can for instance be found in \cite[Theorem 2.49]{Fol16}.

In view of the integral $d_{G/H}$, we define the \emph{dualizing bundle of $\calV$} to be the homogeneous vector bundle $\calV^*:=G\times_HV^*$ associated with the representation $(\rho_{V^*},V^*):=(\rho_V^{\vee}\otimes \Delta_{G/H},V^\vee\otimes\CC)$, where $(\rho_V^\vee,V^\vee)$ denotes the contragredient representation of $(\rho_V,V)$ on $V^\vee:=\Hom(V,\CC)$. In this way, we obtain a canonical non-degenerate $G$-invariant pairing
$$ C^\infty(G/H,\calV)\otimes C_c^\infty(G/H,\calV^{*}) \to \CC, \quad   f_1\otimes f_2\mapsto \langle f_1,f_2\rangle:= \int_{G/H}\langle f_1(g),f_2(g)\rangle\,d_{G/H}(gH), $$
so we can view $C^\infty(G/H,\calV)$ as a subspace of the dual space $C_c^\infty(G/H,\calV^*)^\vee$.

This motivates to define the distributional sections of $\calV$ as the dual object to compactly supported smooth sections of $\calV^*$
$$ \calD'(G/H,\calV) := C_c^\infty(G/H,\calV^*)^\vee. $$
The $G$-action on $\calD'(G/H,\calV)$ given by
$$ \langle g\cdot u,\varphi\rangle = \langle u,g^{-1}\cdot\varphi\rangle \qquad (g\in G,u\in\calD'(G/H,\calV),\varphi\in C_c^\infty(G/H,\calV^*)) $$
extends the $G$-action on $C^\infty(G/H,\calV)$. Note that $(\calV^*)^*=\calV$, so $\calD'(G/H,\calV^*)=C_c^\infty(G/H,\calV)^\vee$.

\subsection{Distribution kernels of symmetry breaking operators}

Now let $G'$ be another algebraic subgroup of $G$ and $H'$ an algebraic subgroup of $G'$. For a finite-dimensional continuous complex representation $(\rho_W, W)$ of $H'$ we write $\calW=G'\times_{H'}W$ for the corresponding vector bundle over $G'/H'$. Denote by $\Hom_{G'}(C_c^\infty(G/H,\calV),C^\infty(G'/H',\calW))$ the space of $G'$-equivariant linear maps $C_c^\infty(G/H,\calV)\to C^\infty(G'/H',\calW)$ which are additionally continuous in the case where $\FF$ is archimedean.

\begin{definition}
	A element $T\in\Hom_{G'}(C_c^\infty(G/H,\calV),C^\infty(G'/H',\calW))$ is called a \emph{symmetry breaking operator}.
\end{definition}

In order to construct and classify symmetry breaking operators, we relate them to their distribution kernels. These turn out to be elements of the space $\calD'(G/H,\calV^*)\otimes W$ which are invariant under the action of $H'$. Here, $H'$ acts on $\calD'(G/H,\calV^*)$ by restricting the action of $G$. We identify $K\in\calD'(G/H,\calV^*)\otimes W$ with a linear map $C_c^\infty(G/H,\calV)\to W$, so that for $f\in C_c^\infty(G/H,\calV)$ we have $\langle K,f\rangle\in W$. Then the diagonal action of $h'\in H'$ on $K\in\calD'(G/H,\calV^*)\otimes W$ can be written as
\begin{equation}
	\langle(h',h')\cdot K,f\rangle = \rho_W(h')\langle K,(h')^{-1} \cdot f\rangle \qquad (f\in C_c^\infty(G/H,\calV)).\label{eq:H'ActionOnDistributionKernels}
\end{equation}
Denote by $(\calD'(G/H,\calV^*)\otimes W)^{\Delta(H')}$ the subspace of all $H'$-invariant elements.

The following statement readily generalizes \cite[Proposition 3.2]{KS15}:

\begin{proposition}\label{prop:SchwartzKernel}
	The following map is a linear isomorphism:
	\begin{equation*}
		\Hom_{G'}(C_c^\infty(G/H,\calV),C^\infty(G'/H',\calW)) \to (\calD'(G/H,\calV^*)\otimes W)^{\Delta(H')}, \quad T \mapsto K_T := \ev_{eH'}\circ T,
	\end{equation*}
	where $\ev_{eH'}:C^\infty(G'/H',\calW)\to W,\,f\mapsto f(eH')$ is evaluation at the base point $eH'\in G'/H'$.
\end{proposition}

\begin{proof}
	With \eqref{eq:H'ActionOnDistributionKernels} it is easy to see that $K_T\in(\calD'(G/H,\calV^*)\otimes W)^{\Delta(H')}$ whenever $T\in\Hom_{G'}(C_c^\infty(G/H,\calV),C^\infty(G'/H',\calW))$, so the map under consideration is defined. Now, for $T \in \Hom_{G'}(C_c^\infty(G/H,\calV),C^\infty(G'/H',\calW))$ and $f\in C_c^\infty(G/H,\calV)$ note that
	\begin{equation}
         \label{eq:TfromKT}
		(Tf)(g') = ((g')^{-1}\cdot (Tf))(1) =(T((g')^{-1}\cdot f))(1) =  \langle K_T,(g')^{-1}\cdot f\rangle \qquad (g'\in G').
	\end{equation}
	This implies that the map $T\mapsto K_T$ is injective. To show surjectivity, we define for a given $K_T\in(\calD'(G/H,\calV^*)\otimes W)^{\Delta(H')}$ and $f\in C_c^\infty(G/H,\calV)$ a function $Tf:G'\to W$ by the right hand side of \eqref{eq:TfromKT}, i.e. $(Tf)(g'):=\langle K_T,(g')^{-1}\cdot f\rangle$. Since $f$ is smooth/locally constant with compact support, the map $g'\mapsto(g')^{-1}\cdot f$ is smooth/locally constant and hence $Tf$ is also smooth/locally constant. Moreover, for all $g'\in G'$, $h'\in H'$:
	\begin{equation*}
	\begin{split}
		(Tf)(g'h') &=  \langle K_T,(g'h')^{-1}\cdot f\rangle= \langle K_T,(h')^{-1} (g')^{-1}\cdot f\rangle\\
		&= \rho_{W}(h')^{-1}\langle(h',h')\cdot K_T,(g')^{-1}\cdot f\rangle = \rho_{W}(h')^{-1}\langle K_T,(g')^{-1}\cdot f\rangle\\
		&= \rho_{W}(h')^{-1}Tf(g'),
	\end{split}
	\end{equation*}
	so $Tf\in C^\infty(G'/H',\calW)$. It remains to show that $f\mapsto Tf$ is $G'$-intertwining, and in the archimedean case additionally continuous. The intertwining property is shown as follows:
	\begin{equation*}
		\begin{split}
			(T(g'\cdot f))(x)&= \langle K_T,x^{-1}g'\cdot f\rangle = \langle K_T,((g')^{-1}x)^{-1}\cdot f\rangle = ((g')^{-1}\cdot Tf)(x)
		\end{split}
	\end{equation*}
	for all $f\in C_c^\infty(G/H,\calV)$ and $g',x\in G'$. Finally, continuity of $T$ is a consequence of the continuity of $K_T$ as a linear map $C_c^\infty(G/H,\calV)\to W$ if $\FF$ is archimedean.
\end{proof}

\subsection{Symmetry breaking operators between principal series}

Now assume that $G/H=G/P$ is a flag variety of a reductive algebraic group $G$ over the local field $\FF$, i.e. $P\subseteq G$ is a parabolic subgroup defined over $\FF$. We assume additionally that $G'\subseteq G$ is a reductive subgroup and that $G'/H'=G'/P'$ is a flag variety for $G'$. In particular, both $G/P$ and $G'/P'$ are compact.

Fix an opposite parabolic subgroup $\overline{P}$ of $P$ and denote by $\overline{N}$ its unipotent radical. Note that if we fix a Levi subgroup of $P$, then there is a unique opposite parabolic subgroup $\overline{P}$ with the same Levi. Note that $\overline{N} \cap P = \{e\}$ and $\overline{N}P$ is open and dense in $G$.

While a distribution on $G/P$ is in general not uniquely determined by its restriction to the open dense subset $\overline{N}P/P\subseteq G/P$, we show that this is indeed the case for the kernel $K_T$ given by Proposition \ref{prop:SchwartzKernel} of a symmetry breaking operator $T\in\Hom_{G'}(C^\infty(G/P,\calV),C^\infty(G'/P',\calW))$ under the following condition:
\begin{equation}
	P'\overline{N}P = G.\label{eq:B'NB=G}
\end{equation}

To make this precise, consider for open subsets $\Omega'\subseteq\Omega\subseteq G/P$ the restriction map
$$ r_{\Omega'}^\Omega:\calD'(\Omega,\calV^*) \to \calD'(\Omega',\calV^*) $$
which is dual to the extension map
\begin{equation}
\label{equ::def_i_omega}
 i_{\Omega'}^\Omega:C_c^\infty(\Omega',\calV)\to C_c^\infty(\Omega,\calV), \quad (i_{\Omega'}^\Omega\varphi)(x) = \begin{cases}\varphi(x)&\mbox{for }x\in \Omega',\\0&\mbox{else.}\end{cases} 
\end{equation}

Of particular importance will be the restriction to subsets of the form $\Omega=UP/P$ with $U\subseteq\overline{N}$ open. On $UP/P\simeq U$ the bundle $\calV$ is trivial, so we obtain an isomorphism
\begin{equation} \label{equ::def_gamma_U}
\gamma_U:C_c^\infty(U) \otimes V \to C_c^\infty(UP/P,\calV) \quad (\gamma_U\varphi)(x): = \rho_V(p^{-1})(\varphi(\overline{n})) \qquad (x=\overline{n}p \in UP), 
\end{equation}
where we view $C_c^\infty(U)\otimes V$ as the space of compactly supported $V$-valued smooth functions on $U$. The corresponding dual map into the space of $V^*$-valued distributions on $U$ is the isomorphism given by
\begin{equation}
	\gamma_U^*:\calD'(UP/P,\calV^*)\to\calD'(U)  \otimes V^{*},  \quad  \langle\gamma_U^*v, \varphi \rangle := \langle v, \gamma_U\varphi \rangle, \label{eq:DefIota*}
\end{equation}
where $v\in\calD'(UP/P,\calV^*)$ and $\varphi \in C_c^\infty(U) \otimes V$, and
$$ \calD'(U)=C_c^\infty(U)^\vee $$
denotes the space of distributions on $U$.

Now, for every $g\in G$, the action of $g^{-1}$ on $C^\infty(G/P,\calV)$ maps $C_c^\infty(g\overline{N}P/P,\calV)$ to $C_c^\infty(\overline{N}P/P,\calV)$, so the dual map $g\cdot:\calD'(\overline{N}P/P,\calV^*)\to\calD'(g\overline{N}P/P,\calV^*)$ makes the following diagram commute:
\begin{equation}\label{eq:RestrictionDiagram} \xymatrix{
	\calD'(G/P,\calV^*) \ar[r]^{g\cdot} \ar[d]_{r^{G/P}_{\overline{N}P/P}} & \calD'(G/P,\calV^*) \ar[d]^{r^{G/P}_{g\overline{N}P/P}} \\
	\calD'(\overline{N}P/P,\calV^*) \ar[r]^{g\cdot} & \calD'(g\overline{N}P/P,\calV^*).
	} \end{equation}
Write $\overline{N}_g=\overline{N}\cap g\overline{N}P$ and denote by $\pi(g)$ the following composition:
$$ \calD'(\overline{N})\otimes V^* \stackrel{(\gamma_{\overline{N}}^*)^{-1}}{\longrightarrow} \calD'(\overline{N}P/P,\calV^*) \stackrel{g\cdot}{\longrightarrow} \calD'(g\overline{N}P/P,\calV^*) \stackrel{r^{g\overline{N}P/P}_{\overline{N}_gP/P}}{\longrightarrow} \calD'(\overline{N}_gP/P,\calV^*) \stackrel{\gamma_{\overline{N}_g}^*}{\longrightarrow} \calD'(\overline{N}_g)\otimes V^*. $$

\begin{proposition}\label{prop:GeneralCharSBOKernels}
Under the assumption \eqref{eq:B'NB=G}, the map 
	$$ \gamma_{\overline{N}}^*\circ r_{\overline{N}P/P}^{G/P}:(\calD'(G/P,\calV^*)\otimes W)^{\Delta(P')} \to \calD'(\overline{N})\otimes V^{*}\otimes W $$
	induces a bijection onto the subspace 
	\begin{equation}
		(\calD'(\overline{N})\otimes V^{*}\otimes W)^{\Delta(P')}:=\{ K \in \calD'(\overline{N})\otimes V^{*}\otimes W \; \vert\; \big[\pi(p)\otimes\rho_W(p)\big]K=K|_{\overline{N}_{p}}\mbox{ for all }p\in P'\}. \label{eq:GeneralEquivariancePropertyKernels}
	\end{equation}
\end{proposition}

\begin{proof}
Let us prove that given $A \in  (\calD'(G/P,\calV^*)\otimes W)^{\Delta(P')}$ the distribution
$$ K=\gamma_{\overline{N}}^*\circ r_{\overline{N}P/P}^{G/P}(A) $$
is indeed contained in $(\calD'(\overline{N})\otimes V^{*}\otimes W)^{\Delta(P')}$. In fact, by the definition of $\pi(p)$ we get for every $p\in P'$:
$$ \pi(p)K = \gamma_{\overline{N}_{p}}^*\circ r_{\overline{N}_{p}P/P}^{p'\overline{N}P/P}((p,1)\cdot r^{G/P}_{\overline{N}P/P}(A)). $$
By the commutativity of the diagram \eqref{eq:RestrictionDiagram} this can be written as
$$ \gamma_{\overline{N}_{p}}^*\circ r_{\overline{N}_{p}P/P}^{p'\overline{N}P/P}\circ r^{G/P}_{p\overline{N}P/P}((p,1)\cdot A), $$
which, by the functoriality of restriction, equals
$$ \gamma_{\overline{N}_{p}}^*\circ r_{\overline{N}_{p}P/P}^{G/P}((p,1)\cdot A). $$
Now, $A \in  (\calD'(G/P,\calV^*)\otimes W)^{\Delta(P')}$ implies $ (p,p)\cdot A=  A$, thus $(p,1)\cdot A=\rho_W(p)^{-1}A$, so the above expression equals
$$ \rho_W(p)^{-1}\big[\gamma_{\overline{N}_{p}}^*\circ r_{\overline{N}_{p}P/P}^{G/P}(A)\big]. $$
Finally, $\gamma_U^*$ is functorial in $U$, so this can be written as
$$ \rho_W(p)^{-1}\big[\gamma_{\overline{N}}^*\circ r_{\overline{N}P/P}^{G/P}(A)\big]|_{\overline{N}_{p}} = \rho_W(p)^{-1}K|_{\overline{N}_{p}} $$
and hence $K\in(\calD'(\overline{N})\otimes V^{*}\otimes W)^{\Delta(P')}$.

To see that the above map is surjective we use the sheaf property of distributions. For a fixed $K \in (\calD'(\overline{N})\otimes V^{*}\otimes W)^{\Delta(P')}$ and every $p\in P'$ define
$$ K_{p} := \rho_W(p)\big[(p,1)\cdot(\gamma_{\overline{N}}^*)^{-1}(K)\big] \in\calD'(p\overline{N}P/P,\calV^*)\otimes W. $$
We claim that $K_{p}$ and $K_{p'}$ agree on $p\overline{N}P/P\cap p'\overline{N}P/P$ for all $p,p'\in P'$. In fact, acting on both $K_{p}$ and $K_{p'}$ by the inverse of $p'$, we may and do assume that $p'=e$. To show that $K_{p}=K_e$ on $p\overline{N}P/P\cap\overline{N}P/P=\overline{N}_{p}P/P$ we note that by the definition of $\pi(p)$ we have
$$ K_{p}|_{\overline{N}_{p}P/P} = \rho_W(p)\big[(p,1)\cdot(\gamma_{\overline{N}}^*)^{-1}(K)\big]|_{\overline{N}_{p}P/P} = (\gamma_{\overline{N}_{p}}^*)^{-1}\big[\pi(p)\otimes\rho_W(p)\big](K). $$
By the assumption on $K$ we have $\big[\pi(p)\otimes\rho_W(p)\big](K)=K|_{\overline{N}_{p}}$, so the expression above equals
$$ (\gamma_{\overline{N}_{p}}^*)^{-1}(K|_{\overline{N}_{p}}) = K_e|_{\overline{N}_{p}P/P} $$
by the functoriality of $\gamma_U^*$ in $U$. This shows that $\{K_{p}\}_{p\in P'}$ is a family of distributions that pairwise agree on intersections. By the assumption \eqref{eq:B'NB=G}, the domains $p\overline{N}P/P$ ($p\in P$) cover $G/P$, so by the sheaf property of distributions there exists a (unique) $A\in\calD'(G/P,\calV^*)\otimes W$ such that $A|_{p\overline{N}P/P}=K_{p}$. That $A$ is in fact contained in $(\calD'(G/P,\calV^*)\otimes W)^{\Delta(P')}$ follows by restricting it to the cover $p\overline{N}P/P$, $p\in P'$, and using the equivariance property of $K$.

Finally, we show that the above map is injective. Assume $A\in(\calD'(G/P,\calV^*)\otimes W)^{\Delta(P')}$ such that $\gamma_{\overline{N}}^*\circ r_{\overline{N}P/P}^{G/P}(A)=0$. Since $\gamma_{\overline{N}}^*$ is an isomorphism, this is equivalent to $r_{\overline{N}P/P}^{G/P}(A)=0$. But $(p,p)\cdot A=A$ for every $p\in P'$, so the diagram \eqref{eq:RestrictionDiagram} implies
$$ r_{p\overline{N}P/P}^{G/P}(A) = r_{p\overline{N}P/P}^{G/P}((p,p)\cdot A) = (p,p)\cdot r_{\overline{N}P/P}^{G/P}(A) = 0. $$
Since the open sets $p\overline{N}P/P$ ($p\in P'$) cover $G/P$, this shows $A=0$.
\end{proof}

We interpret $\calD'(\overline{N})$ as a space of generalized functions on $\overline{N}$ by identifying a locally integrable function $f: \overline{N} \to\CC$ with a distribution $f\in\calD'(\overline{N})$ by
$$ \langle f,\varphi\rangle = \int_{\overline{N}} f(\overline{n})\varphi(\overline{n})\,d\overline{n} \qquad (\varphi\in C_c^\infty(\overline{N})), $$
where $d\overline{n}$ denotes a fixed Haar measure on $\overline{N}$. It is natural to use the Haar measure on $\overline{N}$, because after suitable normalization of $d\overline{n}$ we have
\begin{equation}
	\int_{G/P}h(xP)\,d_{G/P}(xP) = \int_{\overline{N}} h(\overline{n})\,d\overline{n} \qquad (h\in C(G/P,\Omega_{G/P})).\label{eq:IntegralG/PvsNbar}
\end{equation}
This follows from Theorem \ref{thm::invaraint integral} and the integral formula
$$ \int_G h(g)\,dg = \int_{\overline{N}}\int_P h(\overline{n}p)\,d^r_Pp\,d_{\overline{N}}^\ell\overline{n} = \int_{\overline{N}}\int_P h(\overline{n}p)\Delta_{G/P}(p)\,d^\ell_Pp\,d_{\overline{N}}^\ell\overline{n}. $$

Using the language of generalized functions on $\overline{N}$, the condition in \eqref{eq:GeneralEquivariancePropertyKernels} can be rewritten as follows:

\begin{lemma}
\label{lem::gen_fct_distribution}
Fix $g\in G$ and for $\overline{n}\in\overline{N}_g$ write $g^{-1}\overline{n}=\overline{n}(g^{-1}\overline{n})p(g^{-1}\overline{n})\in\overline{N}P$. Then the map $\pi(g):\calD'(\overline{N})\otimes V^*\to\calD'(\overline{N}_g)\otimes V^*$ is given by
\begin{equation}\label{eq:PartialActionNbar_1}
	(\pi(g)u)(\overline{n}) = \rho_{V^*}(p(g^{-1}\overline{n}))^{-1}u(\overline{n}(g^{-1}\overline{n})) \qquad (\overline{n}\in\overline{N}_g)
\end{equation}
in the sense of generalized functions.
\end{lemma}

\begin{proof}
Unwrapping the definition of $\pi(g)$, we have for any $\phi\in C_c^\infty(\overline{N}_g)\otimes V$:
$$ \langle\pi(g)u,\phi\rangle = \langle u,\gamma_{\overline{N}}^{-1}\circ g^{-1}\circ\iota_{\overline{N}_gP/P}^{g\overline{N}P/P}\circ\gamma_{\overline{N}_g}\phi\rangle. $$
If now $u$ is given by a locally integrable $V$-valued function on $\overline{N}$, this equals
$$ \int_{\overline{N}} u(\overline{n})(\gamma_{\overline{N}}^{-1}\circ g^{-1}\circ\iota_{\overline{N}_gP/P}^{g\overline{N}P/P}\circ\gamma_{\overline{N}_g}\phi)(\overline{n})\,d\overline{n}. $$
By \eqref{eq:IntegralG/PvsNbar} the integral over $\overline{N}$ can also be written as an integral over the open subset $\overline{N}P/P$ of $G/P$:
$$ \int_{\overline{N}P/P} (\gamma_{\overline{N}}u)(xP)(g^{-1}\circ\iota_{\overline{N}_gP/P}^{g\overline{N}P/P}\circ\gamma_{\overline{N}_g}\phi)(xP)\,d_{G/P}(xP), $$
where we abuse notation and write $\gamma_{\overline{N}}$ also for its extension to $L^1_{\mathrm{loc}}(\overline{N})\otimes V$, the locally integrable $V$-valued functions on $\overline{N}$, by the same formula. Since the integral over $G/P$ is $G$-invariant, this equals
$$ \int_{g\overline{N}P/P} (\gamma_{\overline{N}}u)(g^{-1}xP)(\iota_{\overline{N}_gP/P}^{g\overline{N}P/P}\circ\gamma_{\overline{N}_g}\phi)(xP)\,d_{G/P}(xP). $$
Note that it suffices to integrate over $\overline{N}_gP/P$ since this set contains the support of $\gamma_{\overline{N}_g}\phi$. Finally, rewriting the integral over $\overline{N}_gP/P$ as an integral over $\overline{N}_g$ using \eqref{eq:IntegralG/PvsNbar} we obtain
$$ \int_{\overline{N}_g} (\gamma_{\overline{N}}u)(g^{-1}\overline{n})\phi(\overline{n})\,d\overline{n}. $$
Since
$$ (\gamma_{\overline{N}}u)(g^{-1}\overline{n}) = \rho_{V^*}(p(g^{-1}\overline{n}))^{-1}u(\overline{n}(g^{-1}\overline{n})) $$
by definition, the proof is complete.
\end{proof}

\section{Symmetry breaking operators for principal series of $\PGL(2)$}
\label{sec::pgl_2_sym}

We apply the theory developed in the previous section to characterize symmetry breaking operators between principal series of $\PGL(2)$ in terms of their distribution kernels.

\subsection{Principal series of $\PGL(2)$}\label{sec:PSforPGL2}

Let $\EE$ be a non-archimedean local field with absolute value $|\cdot|_\EE$. We denote by $\calO_\EE=\{u\in \EE:|u|_\EE \leq1\}$ the ring of integers, by $\calO_\EE^{\times}=\{x\in \EE:|x|_\EE=1\}$ the group of units, and we choose a uniformizer $\varpi_\EE$ for $\calO_\EE$. Write $q_\EE$ for the cardinality of the residue field $\calO_\EE/\varpi_\EE\calO_\EE$. In what follows, we normalize Haar measure $dx$ on $\EE$ so that $\calO_\EE$ has measure $1$. Further, we normalize the absolute value $|\cdot|_\EE$ so that $d(ax)=|a|_\EE\,dx$ for all $a\in\EE^\times$.

We consider the group $G=\PGL(2,\EE)$ and abuse notation by writing an element of $G$ as a $2\times2$ matrix with entries in $\EE$ which represents the corresponding equivalence class in $\PGL(2,\EE)$. Let $B$ be the standard Borel subgroup of $G$ consisting of upper triangular matrices.

For a fixed character $\chi$ of $\calO_\EE^\times$ (i.e. $\chi \in \widehat{\calO_\EE^\times}$) and $s\in\CC$ we denote by $\chi_s$ the character of $\EE^\times$ given by
$$ \chi_s(x): = \chi(\varpi_\EE^{-\nu_\EE(x)}x)|x|_\EE^s \qquad (x\in\EE^\times). $$
Clearly, every character of $\EE^\times$ is of this form since $\EE^\times=\calO_\EE^\times\cdot\{\varpi_\EE^k:k\in\ZZ\}$. The character $\chi_s$ defines a character of $B$, also denoted by $\chi_s$, by
$$ \chi_s\begin{pmatrix}a&b\\0&1\end{pmatrix} = \chi_s(a). $$
Note that $\chi_s$ only depends on $s+\frac{2\pi i}{\ln q_\EE}\ZZ$.

Let $(\pi_{\chi,s},\Ind_B^G(\chi_s))$ denote the smooth representation of $G$ induced from the character $\chi_s$ of $B$, i.e. $\pi_{\chi,s}$ acts by left translation on the space
$$ \Ind_B^G(\chi_s):= \{f\in C^\infty(G):f(gb)=\chi_{s+\frac{1}{2}}(b)^{-1}f(g), \; \forall g\in G,b\in B\}. $$
The representations $\pi_{\chi,s}$ are called the \emph{principal series of $G$}. If $\chi=1$, they are called \emph{unramified}, otherwise they are called \emph{ramified}. Note that $\pi_{\chi,s}$ can also be realized on the space $C^\infty(G/B,\calV)$, where $\calV$ is the homogeneous vector bundle induced from the character $\chi_{s+\frac{1}{2}}=\chi_s\otimes\Delta_{G/B}^{1/2}$ of $B$. Since
$$ \Delta_{B}\begin{pmatrix}a&b\\0&1\end{pmatrix} = |a|_\EE^{-1} \qquad \Delta_{G/B} \begin{pmatrix}a&b\\0&1\end{pmatrix}=(\frac{\Delta_G|_B}{\Delta_B})\begin{pmatrix}a&b\\0&1\end{pmatrix}= |a|_\EE  \qquad (a\in \EE^\times,b\in \EE), $$
the dualizing bundle $\calV^*$ is induced from
$$(\chi_{s+\frac{1}{2}})^{\vee}\otimes \Delta_{G/B}= \overline{\chi}_{-s-\frac{1}{2}} \cdot \Delta_{G/B} = \overline{\chi}_{-s+\frac{1}{2}},$$
 where $\chi^\vee=\overline{\chi}$ is the complex conjugate character of $\chi$ and $\overline{\chi}_{-s+\frac{1}{2}}= \overline{\chi}_{-s}\otimes\Delta_{G/B}^{1/2}$.
 
 From \cite[Theorem 3.3]{JL70} we have the following result:
 
\begin{theorem}\label{thm:CompositionSeries}
	The representation $\pi_{\chi,s}$ is irreducible unless $\chi^2=1$ and $s\in\pm\frac{1}{2}+\frac{i\pi}{\ln q_\EE}\ZZ$. In this case, it has a unique irreducible subrepresentation and the quotient of $\pi_{\chi,s}$ by this subrepresentation is irreducible. More precisely:
	\begin{enumerate}
		\item If $\chi^2=1$ and $s\in-\frac{1}{2}+\frac{i\pi}{\ln q_\EE}\ZZ$, the representation $\pi_{\chi,s}$ has a unique one-dimensional subrepresentation given by the character $\chi_{s+\frac{1}{2}}\circ\det:G\to\CC^\times$.
		\item If $\chi^2=1$ and $s\in\frac{1}{2}+\frac{i\pi}{\ln q_\EE}\ZZ$, the representation $\pi_{\chi,s}$ has a unique one-dimensional quotient given by the character $\chi_{s-\frac{1}{2}}\circ\det:G\to\CC^\times$.
	\end{enumerate}
\end{theorem}

The proper infinite-dimensional quotients/subrepresentations of $\pi_{\chi,s}$ are called \emph{Steinberg representations}. They can most easily be described using the standard intertwining operator. Recall that for sufficiently large $\Re(s)$ the standard intertwining operator $T_{\chi,s}:\pi_{\chi,s}\to\pi_{\overline{\chi},-s}$ is defined by the absolutely convergent integral
$$ T_{\chi,s}f(g) = \int_{\overline{N}}f(gw\overline{n})\,d\overline{n} \qquad (f\in\Ind_B^G(\chi_s)), $$
where
\begin{equation}
	w = \begin{pmatrix}0&1\\1&0\end{pmatrix}\in G\label{eq:LongestWeylGrpElt}
\end{equation}
denotes a representative of the longest Weyl group element. This family of operators can be extended meromorphically to all $s\in\CC$. Since for $x,y\in\FF$, $y\neq0$, we can decompose
$$ \overline{n}_xw\overline{n}_y = \overline{n}_{x+1/y}\begin{pmatrix}y&1\\0&y^{-1}\end{pmatrix}, $$
it follows that
\begin{equation}
	T_{\chi,s}f(\overline{n}_x) = \int_\EE\chi_{s+\frac{1}{2}}(y^2)^{-1}f(\overline{n}_{x+\frac{1}{y}})\,dy = \int_\EE\chi_{s-\frac{1}{2}}(y^2)f(\overline{n}_{x-y})\,dy,\label{eq:IntertwinerConvolutionFormula}
\end{equation}
i.e. on $\overline{N}\simeq\EE$ the operator $T_{\chi,s}$ is given by convolution with the function $\chi_{s-\frac{1}{2}}(y^2)$. In particular, for $\chi^2=1$ and $s\in\frac{1}{2}+\frac{i\pi}{\ln q_\EE}\ZZ$ we obtain an intertwining operator
$$ T_{\chi,s}:\pi_{\chi,s}\to\pi_{\chi,-s}, \quad T_{\chi,s}f(\overline{n}_x) = \int_\EE f(\overline{n}_y)\,dy. $$
The image of this operator are the functions which are constant on $\overline{N}$, i.e the one-dimensional irreducible subrepresentation of $\pi_{\chi,-s}$, and its kernel is the Steinberg representation.

\subsection{The branching problem}

Let $\EE/\FF$ be a quadratic field extension, i.e. $\EE=\FF(\alpha)$ with $\alpha\in \EE\setminus \FF$, $\alpha^2\in \FF^\times\setminus\FF^{\times2}$. After possibly multiplying $\alpha$ by a power of $\varpi_\FF$, we may assume that either $\alpha^2\in \calO_\FF^{\times}$ (in the case where $\EE/\FF$ is unramified, i.e. $q_\EE=q_\FF^2$) or $\alpha^2\in\varpi_\FF \calO_\FF^{\times}$ (in the case where $\EE/\FF$ is ramified, i.e. $q_\EE=q_\FF$). Then the absolute values $|\cdot|_\EE$ and $|\cdot|_\FF$ are related by
$$ |x_1+x_2\alpha|_\EE = |x_1^2-x_2^2\alpha^2|_\FF \qquad (x_1,x_2\in \FF). $$

We consider the pair of groups $(G,G')=(\PGL(2,\EE),\PGL(2,\FF))$ and study the restriction of representations of $G$ to $G'$. With the same notation as for $\EE$, we consider for a character $\eta$ of $\calO_\FF^{\times}$ and $t\in\CC$ the principal series $(\tau_{\eta,t},\Ind_{B'}^{G'}(\eta_t))$. An intertwining operator $A \in \Hom_{G'}(\pi_{\chi,s}|_{G'},\tau_{\eta,t})$ is called \emph{symmetry breaking operator}.

\subsection{Symmetry breaking kernels for $\PGL(2)$}

We now use the results from Section~\ref{sec:SBOs} to identify symmetry breaking operators for $(G,G')=(\PGL(2,\EE),\PGL(2,\FF))$ with their distribution kernels. Putting $(\rho_V,V):=(\chi_{s+\frac{1}{2}},\CC)$ and $(\rho_W,W):=(\eta_{t+\frac{1}{2}},\CC)$, we have
\begin{equation*} 
\begin{split}
\Hom_{G'}(\pi_{\chi,s}|_{G'},\tau_{\eta,t}) :=& \Hom_{G'}(C^\infty(G/B,\calV),C^\infty(G'/B',\calW))\\
=& \Hom_{G'}(C_c^\infty(G/B,\calV),C^\infty(G'/B',\calW)),\\
\end{split} 
\end{equation*}
so Proposition~\ref{prop:SchwartzKernel} applies. Moreover,  the assumption \eqref{eq:B'NB=G} is satisfied in this situation:

\begin{lemma}
For $(G,G')=(\PGL(2,\EE),\PGL(2,\FF))$ with $B\subseteq G$ the Borel subgroup of upper triangular matrices, $\overline{N}$ the subgroup of lower triangular unipotent matrices and $B'=B\cap G'$, we have $B'\overline{N}B=G$.
\end{lemma}

\begin{proof}
Recall that one can realize $G/B$ as the projective line $P^1(\EE)$. The action of $G$ on $P^1(\EE)$ is given by
$$ \begin{pmatrix}a&b\\c&d\end{pmatrix}\cdot \big[ \begin{smallmatrix}  x \\y \end{smallmatrix} \big] = \big[ \begin{smallmatrix} a x +by \\cx+dy \end{smallmatrix} \big], \qquad \big[ \begin{smallmatrix}  x \\y \end{smallmatrix} \big] \in P^1(\EE), \begin{pmatrix}a&b\\c&d\end{pmatrix} \in G, $$
and $B$ is the stabilizer of the point $\big[ \begin{smallmatrix}  1 \\ 0 \end{smallmatrix} \big] \in P^1(\EE) $. Hence, the set $\overline{N}B/B$ consists of all $\big[ \begin{smallmatrix}  1 \\ x \end{smallmatrix} \big]\in P^1(\EE)$, $x\in\EE$. Consequently, the complement of $\overline{N}B/B$ in $P^1(\EE)$ only consists of the single point $\big[ \begin{smallmatrix} 0 \\ 1 \end{smallmatrix} \big]$, and $B'$ clearly contains matrices that do not stabilize this point.
\end{proof}

It follows that Proposition~\ref{prop:GeneralCharSBOKernels} also applies, so we obtain
$$ \Hom_{G'}(\pi_{\chi,s}|_{G'},\tau_{\eta,t}) \simeq (\calD'(\overline{N})\otimes V^*\otimes W)^{\Delta(B')}. $$
Since both $V$ and $W$ are one-dimensional, the right hand side is a space of scalar-valued distributions on $\overline{N}$. To make the condition in \eqref{eq:GeneralEquivariancePropertyKernels} more precise, we use Lemma~\ref{lem::gen_fct_distribution}. Note that $\rho_{V^*}=\overline{\chi}_{-s+\frac{1}{2}}$, so the condition for $K\in\calD'(\overline{N})$ to be contained in $(\calD'(\overline{N})\otimes V^*\otimes W)^{\Delta(B')}$ becomes
\begin{equation}\label{eq:MorePreciseEquivarianceConditionForKernels}
	K(\overline{n}(b^{-1}\overline{n})) = \overline{\chi}_{-s+\frac{1}{2}}(p(b^{-1}\overline{n}))\eta_{t+\frac{1}{2}}(b)^{-1}K(\overline{n}) \qquad \mbox{for all }b\in B',\overline{n} \in \overline{N}_b.
\end{equation}

Using the isomorphism
$$ \EE\to \overline{N}, \quad x\mapsto \overline{n}_x:=\begin{pmatrix}1&0\\x&1\end{pmatrix}, $$
we can identify $(\calD'(\overline{N})\otimes V^*\otimes W)^{\Delta(B')}$ with a space of distributions on $\EE$ which we denote $ \calD'(\EE)_{s,t}^{\chi,\eta}$. Then we can compute both $\overline{n}(b^{-1}\overline{n}_x)$ and $p(b^{-1}\overline{n}_x)$ explicitly for $b\in B'$ and $x\in\EE$ such that $b^{-1}\overline{n}_x\in\overline{N}B$ and in this way make \eqref{eq:MorePreciseEquivarianceConditionForKernels} more explicit. We first consider $b=\diag(a,1)$ with $a\in\FF^\times$, then
$$ b^{-1}\overline{n}_x =\begin{pmatrix}a&0\\0&1\end{pmatrix}^{-1}\overline{n}_x = \overline{n}_{ax}\begin{pmatrix}a^{-1}&0\\0&1\end{pmatrix} \qquad \mbox{for all }x\in\EE, $$
so $\overline{N}_b=\overline{N}$ and
$$ \overline{n}(b^{-1}\overline{n}_x) = \overline{n}_{ax} \qquad \mbox{and} \qquad p(b^{-1}\overline{n}_x) = \begin{pmatrix}a^{-1}&0\\0&1\end{pmatrix}, $$
and \eqref{eq:MorePreciseEquivarianceConditionForKernels} becomes
$$ K(ax) = \overline{\chi}_{-s+\frac{1}{2}}(a^{-1})\eta_{t+\frac{1}{2}}(a)^{-1}K(x) \qquad (x\in\EE,a\in\FF^\times). $$
Next, we consider $b=\begin{pmatrix}1&y\\0&1\end{pmatrix}\in N'$ with $y\in\FF$ and decompose
$$ b^{-1}\overline{n}_x = \begin{pmatrix}1&y\\0&1\end{pmatrix}^{-1}\overline{n}_x = \overline{n}_{(1-xy)^{-1}x}\begin{pmatrix}1-xy&-y\\0&(1-xy)^{-1}\end{pmatrix} \qquad \mbox{for }x\neq y^{-1}, $$
so $\overline{N}_b=\overline{N}\setminus\{\overline{n}_{y^{-1}}\}$ and
$$ \overline{n}(b^{-1}\overline{n}_x) = \overline{n}_{(1-xy)^{-1}x} \qquad \mbox{and} \qquad p(b^{-1}\overline{n}_x) = \begin{pmatrix}1-xy&-y\\0&(1-xy)^{-1}\end{pmatrix}, $$
and \eqref{eq:MorePreciseEquivarianceConditionForKernels} becomes
$$ K(\tfrac{x}{1-xy}) = \overline{\chi}_{-s+\frac{1}{2}}((1-xy)^2)K(x) \qquad (y\in\FF,x\in\EE\setminus\{y^{-1}\}). $$
Since $\overline{\chi}_{-s+\frac{1}{2}}=(\chi_{s -\frac{1}{2}})^{-1}$, we obtain the following characterization of distribution kernels of symmetry breaking operators between $\pi_{\chi,s}$ and $\tau_{\eta,t}$:

\begin{proposition}\label{prop:CharSBOKernels}
	$\calD'(\EE)_{s,t}^{\chi,\eta}$ consists of the distributions $u\in\calD'(\EE)$ satisfying the following two conditions (in the sense of distributions):
	\begin{enumerate}[(a)]
		\item\label{prop:CharSBOKernels1} For $a\in \FF^\times$: $u(ax)=\chi_{s-\frac{1}{2}}(a)\eta_{t+\frac{1}{2}}(a)^{-1}u(x)$ for all $x\in \EE$,
		\item\label{prop:CharSBOKernels2} For $b\in \FF^\times$: $u(\frac{x}{1-bx})=\chi_{s -\frac{1}{2}}((1-bx)^2)^{-1}u(x)$ for all $x\in \EE\setminus\{b^{-1}\}$.
	\end{enumerate}
\end{proposition}

\begin{remark}\label{rem:CharSBOKernelsDistributionLanguage}

The two conditions in Proposition~\ref{prop:CharSBOKernels} are meant in the sense of generalized functions. We rephrase them in the language of distributions by applying both sides to a test function $\varphi\in C_c^\infty(\EE)$.
\begin{enumerate}[(a)]
	\item Since $d(ax)=|a|_\EE\,dx$, for all $u\in L^1_{\mathrm{loc}}(\EE)$ and $\varphi\in C_c^\infty(\EE)$ we have
	$$ \langle u(a\cdot),\varphi\rangle = \int_\EE u(ax)\varphi(x)\,dx = |a|_\EE^{-1}\int_\EE u(x)\varphi(a^{-1}x)\,dx = |a|_\EE^{-1}\langle u,\varphi(a^{-1}\cdot)\rangle. $$
	Therefore, condition \eqref{prop:CharSBOKernels1} can be rewritten as
	\begin{equation}
		\langle u,\varphi(a^{-1}\cdot)\rangle = |a|_\EE\chi_{s-\frac{1}{2}}(a)\eta_{t+\frac{1}{2}}(a)^{-1}\langle u,\varphi\rangle = (\chi_{s+\frac{1}{2}}\eta_{t+\frac{1}{2}}^{-1})(a)\langle u,\varphi\rangle \qquad (\varphi\in C_c^\infty(\EE)).\label{eq:AltCharSBOKernels1}
	\end{equation}
	\item For the second condition, note that, since $\frac{dx}{|x|_\EE}$ is a Haar measure on the multiplicative group $\EE^\times$ and $\EE\setminus\EE^\times=\{0\}$ is of measure zero, we have the following integral formula:
	\begin{align*}
		& \int_\EE \psi\left(\frac{x}{1-bx}\right)|1-bx|_\EE^{-2}\,dx = \int_\EE \psi\left(\frac{1}{x^{-1}-b}\right)|x^{-1}-b|_\EE^{-2}|x|_\EE^{-2}\,dx\\
		&\stackrel{(y=x^{-1})}{=} \int_\EE \psi\left(\frac{1}{y-b}\right)|y-b|_\EE^{-2}\,dy
		\stackrel{(z=y-b)}{=} \int_\EE\psi(z^{-1})|z|_\EE^{-2}\,dz \stackrel{(w=z^{-1})}{=} \int_\EE\psi(w)\,dw
	\end{align*}
	for $\psi\in L^1(\EE)$, so
	$$ \langle u(\tfrac{\cdot}{1-b\cdot}),\varphi\rangle = \langle u,|1+b\cdot|_\EE^{-2}\varphi(\tfrac{\cdot}{1+b\cdot})\rangle. $$
	With this formula, condition \eqref{prop:CharSBOKernels2} applied to a test function $\varphi\in C_c^\infty(\EE\setminus\{b^{-1}\})$ reads
	\begin{equation}
		\langle u,|1+b\cdot|_\EE^2\varphi(\tfrac{\cdot}{1+b\cdot})\rangle = \langle u,\chi_{s -\frac{1}{2}}((1-b\cdot)^2)^{-1}\varphi\rangle.\label{eq:AltCharSBOKernels2}
	\end{equation}
\end{enumerate}
\end{remark}

We conclude that the construction and classification of symmetry breaking operators $A\in\Hom_H(\pi_{\chi,s}|_H,\tau_{\eta,t})$ can be carried out in terms of their distribution kernels $K\in\calD'(\EE)_{s,t}^{\chi,\eta}$.

\section{A holomorphic family of symmetry breaking kernels}
\label{sec::hol_sym_fam}

We find a family of symmetry breaking kernels $K_{s,t}^{\chi,\eta}\in\calD'(\EE)_{s,t}^{\chi,\eta}$ that depends meromorphically on $s,t\in\CC$ and normalize it to a holomorphic family $\widetilde{K}_{s,t}^{\chi,\eta}$.

\subsection{Classification of symmetry breaking kernels on $\EE^\times$}\label{sec:ClassSBOKernelsEtimes}

We first classify the possible restrictions of a symmetry breaking kernel $u\in\calD'(\EE)_{s,t}^{\chi,\eta}$ to $\EE^\times=\EE\setminus\{0\}$. For this consider the representative $w$ of the longest Weyl group element defined in \eqref{eq:LongestWeylGrpElt}. Since $w\overline{N}=Nw$, the conditions on a distribution on $G/B$ to be $N'$-invariant are easier to study on the open dense subset $w\overline{N}B/B=NwB/B$ than on $\overline{N}B/B$. We therefore first apply the action of $w$ to a possible symmetry breaking kernel in $(\calD'(G/B,\calV^*)\otimes W)^{B'}$.

In the coordinates $\overline{n}_x\in\overline{N}$, $x\in\EE$, this amounts to decomposing $w\overline{n}_x$ into $\overline{N}B$, if possible. For $x\in\EE^\times$ we find
\begin{equation}
\label{equ::some_calc}
\begin{split}
w\overline{n}_x = \begin{pmatrix}0&1\\1&0\end{pmatrix}  \begin{pmatrix}1&0\\x&1\end{pmatrix} &= \begin{pmatrix}x&1\\1&0\end{pmatrix}= \begin{pmatrix}1&0\\1/x&1\end{pmatrix}  \begin{pmatrix}x&0\\0&-x^{-1}\end{pmatrix} \begin{pmatrix}1&1/x \\0&1\end{pmatrix}\\
& \equiv  \begin{pmatrix}1&0\\1/x&1\end{pmatrix}  \begin{pmatrix}-x^{2}&0\\0&1\end{pmatrix} \begin{pmatrix}1&1/x \\0&1\end{pmatrix},
\end{split}
\end{equation}
where the last identity is modulo multiples of the identity matrix. Hence the action of $w$ on $u\in\calD'(G/B,\calV^*)$ is given by
$$ \pi_{\overline{\chi},-s}(w)u(\overline{n}_x) = \overline{\chi}_{-s+\frac{1}{2}}(-x^2)^{-1}u(\overline{n}_{1/x}) = \chi_{s-\frac{1}{2}}(-1)\cdot\chi_{s-\frac{1}{2}}(x^2)u(\overline{n}_{1/x}) \qquad (x\in\EE^\times). $$
This motivates the definition of the following involutive linear isomorphism:
$$ \kappa:\calD'(\EE^\times)\to\calD'(\EE^\times), \quad (\kappa u)(x):=\chi_{s-\frac{1}{2}}(x^2)u(\tfrac{1}{x}). $$

\begin{lemma}
\label{lem:CharSBOKernelsKappa}
	If $u\in\calD'(\EE)_{s,t}^{\chi,\eta}$, then $v=\kappa(u|_{\EE^\times})$ satisfies
	\begin{enumerate}[(a)]
		\item\label{lem:CharSBOKernelsKappa1} For $a\in \FF^\times$: $v(ax)=(\chi_{s-\frac{1}{2}}\eta_{t+\frac{1}{2}})(a)v(x)$ for all $x\in \EE^\times$,
		\item\label{lem:CharSBOKernelsKappa2} For $y\in \FF$: $v(x+y)=v(x)$ for all $x\in \EE^\times\setminus\{-y\}$.
	\end{enumerate}
\end{lemma}

\begin{proof}
	If $u\in\calD'(\EE)_{s,t}^{\chi,\eta}$, then for $a\in \FF^\times$:
	\begin{align*}
		(\kappa u)(ax) &= \chi_{s-\frac{1}{2}}((ax)^2)u(\tfrac{1}{ax})\\
		&= \chi_{s-\frac{1}{2}}(a^2)\chi_{s-\frac{1}{2}}(x^2)\chi_{s-\frac{1}{2}}(\tfrac{1}{a})\eta_{t+\frac{1}{2}}(\tfrac{1}{a})^{-1}u(\tfrac{1}{x})\\
		&= \chi_{s-\frac{1}{2}}(a)\eta_{t+\frac{1}{2}}(a)(\kappa u)(x),
	\end{align*}
	and for $y\in \FF$:
	\begin{align*}
		(\kappa u)(x+y) &= \chi_{s-\frac{1}{2}}((x+y)^2)u(\tfrac{1}{x+y}) = \chi_{s-\frac{1}{2}}((x+y)^2)u(\tfrac{1/x}{1+y/x})\\
		&= \chi_{s-\frac{1}{2}}((x+y)^2)\chi_{s-\frac{1}{2}}((1+\tfrac{y}{x})^2)^{-1}u(\tfrac{1}{x})\\
		&= \chi_{s-\frac{1}{2}}(x^2)u(\tfrac{1}{x}) = (\kappa u)(x).\qedhere
	\end{align*}
\end{proof}

It follows from Lemma~\ref{lem:CharSBOKernelsKappa}~\eqref{lem:CharSBOKernelsKappa2} that if $u\in\calD'(\EE)_{s,t}^{\chi,\eta}$, then $v:=\kappa(u|_{\EE^\times})$ is of the form $v(x)=w(x_2)$  for some $v_1\in\calD'(\FF)$,  where $x=x_1+x_2\alpha \in \EE^{\times}$. 
Moreover, Lemma~\ref{lem:CharSBOKernelsKappa}~\eqref{lem:CharSBOKernelsKappa1} implies that $w$ is homogeneous of degree $\chi_{s-\frac{1}{2}}|_{\FF^\times}\cdot\eta_{t+\frac{1}{2}}$. Homogeneous distribution are classified in \cite[\S2.3]{GG63}, and we recall the result in Appendix~\ref{app:HomogeneousDistributions}. From Theorem~\ref{thm:HomogeneousDistributions} we conclude that unless $\chi|_{\calO^\times_\FF}\cdot\eta=1$ and $2s+t+\frac{1}{2}\in\frac{2\pi i}{\ln q_\FF}\ZZ$, i.e. $(\chi,s,\eta,t)\in\SplusT$, we have
$$ v(x_1+\alpha x_2) = \const\times(\chi_{s-\frac{1}{2}}\eta_{t+\frac{1}{2}})(x_2) $$
and hence $u|_{\EE^\times}=\kappa(v)$ is of the form
$$ u(x) = \const\times\chi_{s-\frac{1}{2}}(x^2)(\chi_{s-\frac{1}{2}}\eta_{t+\frac{1}{2}})((\tfrac{1}{x})_2), $$
where $\frac{1}{x}=(\frac{1}{x})_1+(\frac{1}{x})_2\alpha$ with $(\frac{1}{x})_1,(\frac{1}{x})_2\in\FF$. Note that
$$ \frac{1}{x} = \frac{1}{x_1+x_2\alpha} = \frac{x_1-x_2\alpha}{x_1^2-x_2^2\alpha^2}, $$
so $(\frac{1}{x})_2=(x_1^2-x_2^2\alpha^2)^{-1}x_2$ and hence
\begin{equation*}
	u(x) = \const\times\chi_{s-\frac{1}{2}}(x^2)(\chi_{s-\frac{1}{2}}\eta_{t+\frac{1}{2}})(\tfrac{x_2}{x_1^2-x_2^2\alpha^2}).
\end{equation*}
In the case where where $(\chi,s,\eta,t)\in\SplusT$ we have by the same arguments that
$$ v(x_1+\alpha x_2) = \const\times\delta(x_2) $$
and hence $u|_{\EE^\times}=\kappa(v)$ is of the form
\begin{align*}
	u(x) &= \const\times\chi_{s-\frac{1}{2}}(x^2)\delta((\tfrac{1}{x})_2) = \const\times\chi_{s-\frac{1}{2}}(x^2)\delta(\tfrac{x_2}{x_1^2-x_2^2\alpha^2})\\
	&= \const\times|x_1^2-x_2^2\alpha^2|_\FF\chi_{s-\frac{1}{2}}(x^2)\delta(x_2) = \const\times\chi_{s}(x^2)\delta(x_2)= \const\times\chi_{s}(x_1^2)\delta(x_2).
\end{align*}

If we denote by $\calD'(\EE^\times)_{s,t}^{\chi,\eta}$ the space of all distributions $u\in\calD'(\EE^\times)$ satisfying the conditions of Proposition~\ref{prop:CharSBOKernels} wherever it makes sense, this shows:

\begin{proposition}\label{prop:CharacterizationSBOKernelsPGL2}
$\dim\calD'(\EE^\times)_{s,t}^{\chi,\eta}=1$ for all $(\chi,s,\eta,t)$. More precisely, $\calD'(\EE^\times)_{s,t}^{\chi,\eta}$ is spanned by the distribution
\begin{align*}
	u(x) = \begin{cases}\eta_{t+\frac{1}{2}}(x_1^2)^{-1}\delta(x_2)&\mbox{if $(\chi,s,\eta,t)\in\SplusT$,}\\\chi_{s-\frac{1}{2}}(x^2)(\chi_{s-\frac{1}{2}}\eta_{t+\frac{1}{2}})(\tfrac{x_2}{x_1^2-x_2^2\alpha^2})&\mbox{else.}\end{cases}
\end{align*}
\end{proposition}

\subsection{The meromorphic family}
\label{sec::mero_family}

We define a function $K_{s,t}^{\chi,\eta}$ on $\EE\setminus\FF$ by
$$ K_{s,t}^{\chi,\eta}(x) = \chi_{s-\frac{1}{2}}(x^2)(\chi_{s-\frac{1}{2}}\eta_{t+\frac{1}{2}})(\tfrac{x_2}{x_1^2-x_2^2\alpha^2}),\qquad \mbox{where } x=x_1+x_2\alpha \in \EE^{\times}. $$
Note that since $x^2=(x_1^2-x_2^2\alpha^2)\frac{x_1+x_2\alpha}{x_1-x_2\alpha}$, this can also be written as
\begin{align*}
	K_{s,t}^{\chi,\eta}(x) &= \chi\left(\frac{x_1+x_2\alpha}{x_1-x_2\alpha}\right)\chi_{s-\frac{1}{2}}(x_2)\eta_{t+\frac{1}{2}}(\tfrac{x_2}{x_1^2-x_2^2\alpha^2})\\
	&= 	\chi\left(\frac{x_1+x_2\alpha}{x_1-x_2\alpha}\right)\chi_0(x_2)\eta_0(\tfrac{x_2}{x_1^2-x_2^2\alpha^2})|x|_\EE^{-t-\frac{1}{2}}|x_2|_\FF^{2s+t-\frac{1}{2}}.
\end{align*}
Note that $K_{s,t}^{\chi,\eta}$ only depends on $s$ modulo $\frac{2\pi i}{\ln q_\EE}\ZZ$ and on $t$ modulo $\frac{2\pi i}{\ln q_\FF}\ZZ$.

Since
$$ |K_{s,t}^{\chi,\eta}(x)|=|x|_\EE^{-\Re(t+\frac{1}{2})}|x_2|_\FF^{\Re(2s+t-\frac{1}{2})}, $$
it is clear that $K_{s,t}^{\chi,\eta}$ is locally integrable on $\EE$ for $-\Re(t+\frac{1}{2}),\Re(2s+t-\frac{1}{2})\geq0$. In fact, it is locally integrable for a slightly larger range of parameters:

\begin{lemma}\label{lem:KernelL1loc}
$K_{s,t}^{\chi,\eta}$ is locally integrable on $\EE$ if and only if $\Re(2s\pm t+\frac{1}{2})>0$.
\end{lemma}

\begin{proof}
This will follow from the computations in Sections~\ref{sec:SecondTerm} and \ref{sec:ThirdTerm}.
\end{proof}

In particular, $K_{s,t}^{\chi,\eta}$ is a distribution in $\calD'(\EE)$ for $\Re(2s\pm t+\frac{1}{2})>0$. To meromorphically extend $K_{s,t}^{\chi,\eta}\in\calD'(\EE)$ in $s,t\in\CC$ and determine its residues, we let $\varphi\in C_c^\infty(\EE)$ and consider $\langle K_{s,t}^{\chi,\eta},\varphi\rangle$. Since $\varphi$ is locally constant with compact support, there is a neighborhood $V=\varpi_\FF^n\calO_\FF$ of $0$ in $\FF$ such that $\varphi$ is constant on each coset $x+(V\times V)$. Here, we identify $\EE\simeq\FF\times\FF$ by $x_1+x_2\alpha\mapsto(x_1,x_2)$. We write
\begin{multline}
	\langle K_{s,t}^{\chi,\eta},\varphi\rangle = \int_{\FF\times(\FF\setminus V)}K_{s,t}^{\chi,\eta}(x)\varphi(x)\,dx + \int_{(\FF\setminus V)\times V} K_{s,t}^{\chi,\eta}(x)\varphi(x)\,dx + \int_{V \times V}K_{s,t}^{\chi,\eta}(x)\varphi(x)\,dx\label{eq:MeroContThreeTerms}
\end{multline}
and analyze the three terms separately.

\subsubsection{The first term}

The integral over $\FF\times(\FF\setminus V)$ turns out to be holomorphic in $s,t\in\CC$.

\begin{proposition}\label{prop:FirstTerm}
	The map
	$$ (s,t)\mapsto\int_{\FF\times(\FF\setminus V)}K_{s,t}^{\chi,\eta}(x)\varphi(x)\,dx $$
	defines a holomorphic function in $(s,t)\in\CC^2$.
\end{proposition}

\begin{proof}
	For $(x_1,x_2)\in\FF\times(\FF\setminus V)$ we have
	\begin{equation*}
		|x|_\EE = |x_1^2-x_2^2\alpha^2|_\FF > |\varpi_\FF^{2n}\alpha^2|_\FF > 0 \qquad \mbox{and} \qquad |x_2|_\FF > |\varpi_\FF^n|_\FF > 0,
	\end{equation*}
	so the integral converges absolutely for all $s,t\in\CC$.
\end{proof}

\subsubsection{The second term}\label{sec:SecondTerm}

In this subsection we prove the following statement for the integral over $(\FF\setminus V)\times V$ in \eqref{eq:MeroContThreeTerms}:

\begin{proposition}\label{prop:SecondTerm}
	The map
	$$ (s,t)\mapsto\int_{(\FF\setminus V)\times V} K_{s,t}^{\chi,\eta}(x)\varphi(x)\,dx $$
	extends to a meromorphic function in $(s,t)\in\CC^2$.
	\begin{itemize}
		\item If $\chi|_{\calO_\FF^{\times}}\cdot\eta\neq1$ it is holomorphic in $(s,t)\in\CC^2$.
		\item If $\chi|_{\calO_\FF^{\times}}\cdot\eta=1$ it is holomorphic in $\{(s,t)\in\CC^2:2s+t+\frac{1}{2}\not\in\frac{2\pi i}{\ln q_\FF}\ZZ\}$ with residue at $2s+t+\frac{1}{2}\in\frac{2\pi i}{\ln q_\FF}\ZZ$ given by
		$$ \lim_{2s+t+\frac{1}{2}\to0}(1-q_\FF^{-(2s+t+\frac{1}{2})})\int_{(\FF\setminus V)\times V} K_{s,t}^{\chi,\eta}(x)\varphi(x)\,dx = (1-q_\FF^{-1})\int_{\FF\setminus V}\eta_{t+\frac{1}{2}}(x_1^2)^{-1}\varphi(x_1)\,dx_1. $$
	\end{itemize}
\end{proposition}

We remark that the last expression can be viewed as the distribution $(1-q_\FF^{-1})\eta_{t+\frac{1}{2}}(x_1^2)^{-1}\delta(x_2)$ applied to $\varphi$.

First note that $x\in(\FF\setminus V)\times V$ implies
$$ |x|_\EE = |x_1^2-x_2^2\alpha^2|_\FF = |x_1^2|_\FF = |x_1|_\FF^2, $$
so the integral under consideration becomes
$$ \int_{(\FF\setminus V)\times V}\chi\left(\frac{x_1+x_2\alpha}{x_1-x_2\alpha}\right)\chi_0(x_2)\eta_0\left(\frac{x_2}{x_1^2-x_2^2\alpha^2}\right)|x_1|_\FF^{-2t-1}|x_2|_\FF^{2s+t-\frac{1}{2}}\varphi(x)\,dx. $$
Since $\varphi$ has compact support and is constant on all cosets $x+(V\times V)$, the restriction of $\varphi$ to $(\FF\setminus V)\times V$ is a finite linear combination of characteristic functions of translates of $V\times V$: $\varphi|_{(\FF\setminus V)\times V}=\sum_\ell c_\ell\chi_{(y_\ell+V)\times V}$ with $c_\ell=\varphi(y_\ell)\in\CC$ and $y_\ell\in\FF\setminus V$. The contribution of the $\ell$'th term to the integral is $\varphi(y_\ell)$ times
\begin{equation}
\label{equ::third_term2}
\begin{split}
& \int\limits_{y_\ell+\varpi_\FF^n\calO_\FF}\int\limits_{\varpi_\FF^n\mathcal{O}_\FF} \chi\left(\frac{x_1  +\alpha x_2}{x_1 -\alpha x_2}\right)\chi_0(x_2)\eta_0(x_2)\eta_0(x_1^2-x_2^2\alpha^2)^{-1}|x_1|_\FF^{-2t-1}|x_2|_\FF^{2s+t-\frac{1}{2}}\,dx_2\,dx_1\\
=& \int\limits_{y_\ell+\varpi_\FF^n\calO_\FF}|x_1|_\FF^{-2t-1}\Bigg(\sum\limits_{j=n}^{\infty} \int\limits_{\varpi_\FF^j \calO_\FF^{\times}} \chi\left(\frac{x_1  +\alpha x_2}{x_1 -\alpha x_2}\right)(\chi_0\eta_0)(x_2)\eta_0(x_1^2-x_2^2\alpha^2)^{-1}|x_2|_\FF^{2s+t-\frac{1}{2}}\,dx_2\Bigg)dx_1\\
=& \int\limits_{y_\ell+\varpi_\FF^n\calO_\FF}|x_1|_\FF^{-2t-1}\Bigg(\sum\limits_{j=n}^{\infty}q_\FF^{-j(2s+t-\frac{1}{2})} \int\limits_{\varpi_\FF^j \calO_\FF^{\times}} \chi\left(\frac{x_1  +\alpha x_2}{x_1 -\alpha x_2}\right)(\chi\eta)(\varpi_\FF^{-j}x_2)\eta_0(x_1^2-x_2^2\alpha^2)^{-1}\,dx_2\Bigg)dx_1.
\end{split}
\end{equation}

The inner integral is computed in the following lemma:

\begin{lemma}\label{lem::zero_third_term}
	For every $n\in\ZZ$ there is $N\in\ZZ$ such that:
	\begin{enumerate}
		\item For any $x_1 \in \FF\setminus\varpi_\FF^n\calO_\FF$ and any $x_2 \in \varpi_\FF^N \mathcal{O}_\FF$ we have:
		$$ \chi\left(\frac{x_1  +\alpha x_2}{x_1 -\alpha x_2}\right)=1 \qquad \mbox{and} \qquad \eta_0(x_1^{2}-\alpha^{2} x_2^{2})=\eta_0(x_1^2). $$
		\item For any $x_1\in\FF\setminus\varpi_\FF^n\calO_\FF$ and any $j \geq N$ we have
		$$ \int\limits_{\varpi_\FF^j \calO_\FF^{\times}} \chi\left(\frac{x_1  +\alpha x_2}{x_1 -\alpha x_2}\right)(\chi\eta)(\varpi_\FF^{-j}x_2)\eta_0(x_1^2-x_2^2\alpha^2)^{-1}\,dx_2=\begin{cases}q_\FF^{-j}(1-q_\FF^{-1})\eta_0(x_1^2)^{-1}&\mbox{if $\chi|_{\calO_\FF^{\times}}\cdot\eta=1$,}\\0&\mbox{if $\chi|_{\calO_\FF^{\times}}\cdot\eta\neq1$.}\end{cases}. $$
	\end{enumerate}
\end{lemma}

\begin{proof}
\begin{enumerate}
	\item We first consider the term
	$$ \chi\left(\frac{x_1  +\alpha x_2}{x_1 -\alpha x_2}\right) = \chi\left(1+\frac{2\alpha x_2}{x_1 -\alpha x_2}\right). $$
	Since $\chi$ is locally constant, it is constant in a neighborhood of $1$. But
	$$ \left|\frac{2\alpha x_2}{x_1-\alpha x_2}\right|_\EE = |2\alpha|_\EE\cdot |x_2|_\FF^2\cdot|x|_\EE^{-1}=|2\alpha|_\EE\cdot |x_2|_\FF^2\cdot|x_1|_\FF^{-2}, $$
	and if $x_1\in\FF\setminus\varpi_\FF^n\calO_\FF$ we have $|x_1|_\FF>|\varpi_\FF^n|_\FF>0$. It follows that for any given neighborhood of $1$ we can find $N>0$ such that $1+\frac{2\alpha x_2}{x_1 -\alpha x_2}$ is contained in this neighborhood for all $x_1\in\FF\setminus\varpi_\FF^n\calO_\FF$ and $x_2\in\varpi_\FF^N\calO_\FF$.\\
	Now consider the term $\eta_0(x_1^2-x_2^2\alpha^2)$. Since $\eta_0$ is locally constant, we can for every fixed $y\neq0$ find $N\in\ZZ$ such that $\eta_0(x_1^2-x_2^2\alpha^2)=\eta_0(y^2)=\eta_0(x_1^2)$ for all $x_1\in y+\varpi_\FF^N\calO_\FF$ and $x_2\in\varpi_\FF^N\calO_\FF$. Covering the compact set $\varpi_\FF^n \calO_\FF^{\times}$ by finitely many of the sets $y+\varpi_\FF^N\calO_\FF$ we obtain $N\in\ZZ$ with the property that $\eta_0(x_1^2-x_2^2\alpha^2)=\eta_0(x_1^2)$ for all $x_1\in\varpi_\FF^n \calO_\FF^{\times}$ and $x_2\in\varpi_\FF^N\calO_\FF$. Finally, using the fact that $\eta_0$ is multiplicative, a rescaling argument shows that this in fact holds for all $x_1\in\varpi_\FF^j \calO_\FF^{\times}$ with $j\leq n$, hence for all $x_1\in\FF\setminus\varpi_\FF^n\calO_\FF$.
	\item By (1) and (2) we find for $x_1\in\FF\setminus\varpi_\FF^n\calO_\FF$ and $j\geq N$:
	\begin{multline*}
		\int\limits_{\varpi_\FF^j \calO_\FF^{\times}} \chi\left(\frac{x_1  +\alpha x_2}{x_1 -\alpha x_2}\right)(\chi\eta)(\varpi_\FF^{-j}x_2)\eta_0(x_1^2-x_2^2\alpha^2)^{-1}\,dx_2\\
		= \eta_0(x_1^2)^{-1}\int\limits_{\varpi_\FF^j \calO_\FF^{\times}} (\chi\eta)(\varpi_\FF^{-j}x_2)\,dx_2 = |\varpi_\FF^j|_\FF\cdot\eta_0(x_1^2)^{-1}\int\limits_{\calO_\FF^{\times}} (\chi\eta)(x_2)\,dx_2.
	\end{multline*}
	This is the integral of the character $\chi|_{\calO_\FF^{\times}}\cdot\eta$ of the compact abelian group $\calO_\FF^{\times}$, hence it vanishes if $\chi|_{\calO_\FF^{\times}}\cdot\eta$ is non-trivial and it equals the volume of $\calO_\FF^{\times}$ if $\chi|_{\calO_\FF^{\times}}\cdot\eta$ is the trivial character (see \cite[Lemma  4.1]{Sally98}). The volume of $\calO_\FF^{\times}$ is $(1-q_\FF^{-1})$, so the claim follows.\qedhere
\end{enumerate}
\end{proof}

\begin{proof}[Proof of Proposition~\ref{prop:SecondTerm}]
Assume first that $\chi|_{\calO_\FF^{\times}}\cdot\eta\neq1$, then \eqref{equ::third_term2} becomes
\begin{equation}
	\int\limits_{y_\ell+\varpi_\FF^n\calO_\FF}|x_1|_\FF^{-2t-1}\Bigg(\sum\limits_{j=n}^{N-1}q_\FF^{-j(2s+t-\frac{1}{2})} \int\limits_{\varpi_\FF^j \calO_\FF^{\times}} \chi\left(\frac{x_1  +\alpha x_2}{x_1 -\alpha x_2}\right)(\chi\eta)(\varpi_\FF^{-j}x_2)\eta_0(x_1^2-x_2^2\alpha^2)^{-1}\,dx_2\Bigg)dx_1.\label{equ::third_term2a}
\end{equation}
This expression is obviously holomorphic in $s,t\in\CC$. If $\chi|_{\calO_\FF^{\times}}\cdot\eta=1$, then in addition to \eqref{equ::third_term2a} we get the term
\begin{equation}\label{equ::third_term2b}
\begin{split}
	& \int\limits_{y_\ell+\varpi_\FF^n\calO_\FF}|x_1|_\FF^{-2t-1}\Bigg(\sum\limits_{j=N}^{\infty}q_\FF^{-j(2s+t-\frac{1}{2})} q_\FF^{-j}(1-q_\FF^{-1})\eta_0(x_1^2)^{-1}\Bigg)dx_1\\
	={}& \Bigg((1-q_\FF^{-1})\sum\limits_{j=N}^{\infty}q_\FF^{-j(2s+t+\frac{1}{2})}\Bigg)\int\limits_{y_\ell+\varpi_\FF^n\calO_\FF}|x_1|_\FF^{-2t-1}\eta_0(x_1^2)^{-1}\,dx_1\\
	={}& (1-q_\FF^{-1})\frac{q_\FF^{-N(2s+t+\frac{1}{2})}}{1-q_\FF^{-(2s+t+\frac{1}{2})}}\int\limits_{y_\ell+\varpi_\FF^n\calO_\FF}|x_1|_\FF^{-2t-1}\eta_0(x_1^2)^{-1}\,dx_1.
\end{split}
\end{equation}
Putting together \eqref{equ::third_term2}, \eqref{equ::third_term2a} and \eqref{equ::third_term2b} shows that in this case
\begin{multline*}
	\lim_{2s+t+\frac{1}{2}\to0}(1-q_\FF^{-(2s+t+\frac{1}{2})})\int_{(\FF\setminus V)\times V}K_{s,t}^{\chi,\eta}(x)\varphi(x)\,dx\\
	= (1-q_\FF^{-1})\sum_\ell\varphi(y_\ell)\int_{y_\ell+\varpi_\FF^n}|x_1|_\FF^{-2t-1}\eta_0(x_1^2)^{-1}\,dx_1\\
	= (1-q_\FF^{-1})\int_{\FF\setminus\varpi_\FF^n\calO_\FF}\eta_{t+\frac{1}{2}}(x_1^2)^{-1}\varphi(x_1)\,dx_1.\qedhere
\end{multline*}
\end{proof}

\subsubsection{The third term}\label{sec:ThirdTerm}

The integral over $V\times V$ in \eqref{eq:MeroContThreeTerms} is the most involved term. To state formulas for its residues, we use the constant $c_\FF>0$ and the gamma factors introduced in Appendix~\ref{app:Integral}.

\begin{proposition}\label{prop:ThirdTerm}
	The map
	$$ (s,t)\mapsto \int_{V\times V}K_{s,t}^{\chi,\eta}(x)\varphi(x)\, dx $$
	extends to a meromorphic function in $(s,t)\in\CC^2$.
	\begin{itemize}
		\item If $\chi|_{\calO_\FF^{\times}}\cdot\eta^{-1}\neq1$ it is holomorphic in $(s,t)\in\CC^2$.
		\item If $\chi|_{\calO_\FF^{\times}}\cdot\eta^{-1}=1$ and $\chi|_{\calO_\FF^{\times}}\cdot\eta\neq1$ it is holomorphic in $\{(s,t)\in\CC^2:2s-t+\frac{1}{2}\not\in\frac{2\pi i}{\ln q_\FF}\ZZ\}$ with residue at $2s-t+\frac{1}{2}\in\frac{2\pi i}{\ln q_\FF}\ZZ$ given by
		\begin{multline}\label{eq:ThirdTermResidue-}
			\qquad\quad\lim_{2s-t+\frac{1}{2}\to0}(1-q_\FF^{-(2s-t+\frac{1}{2})})\int_{V\times V} K_{s,t}^{\chi,\eta}(x)\varphi(x)\,dx\\
			=c_\FF(1-q_\FF^{-1})\chi_0(-\alpha)^{-2}|\alpha^2|_\FF^{-t+\frac{1}{2}}\Gamma((\chi^2|_{\calO_\FF^{\times}})_{2t})\Gamma((\chi^{-2})_{-t+\frac{1}{2}})\cdot\varphi(0).
		\end{multline}
		\item If $\chi|_{\calO_\FF^{\times}}\cdot\eta^{-1}=1$ and $\chi|_{\calO_\FF^{\times}}\cdot\eta=1$ it is holomorphic in $\{(s,t)\in\CC^2:2s\pm t+\frac{1}{2}\not\in\frac{2\pi i}{\ln q_\FF}\ZZ\}$ with residue at $2s-t+\frac{1}{2}\in\frac{2\pi i}{\ln q_\FF}\ZZ$ given by \eqref{eq:ThirdTermResidue-}, and  with residue at $2s+t+\frac{1}{2}\in\frac{2\pi i}{\ln q_\FF}\ZZ$ given by
		\begin{equation}\label{eq:ThirdTermResidue+}
			\lim_{2s+t+\frac{1}{2}\to0}(1-q_\FF^{-(2s+t+\frac{1}{2})})\int_{V\times V} K_{s,t}^{\chi,\eta}(x)\varphi(x)\,dx = \frac{(1-q_\FF^{-1})^2q_\FF^{2tn}}{1-q_\FF^{2t}}.
		\end{equation}
	\end{itemize}
\end{proposition}

Since $\varphi$ is constant on $V\times V$, the third term in \eqref{eq:MeroContThreeTerms} equals $\varphi(0)$ times
$$ \int\limits_{V\times V}\chi\left(\frac{x_1+x_2\alpha}{x_1-x_2\alpha}\right)\chi_0(x_2)\eta_0\left(\frac{x_2}{x_1^2-x_2^2\alpha^2}\right)|x|_\EE^{-t-\frac{1}{2}}|x_2|_\FF^{2s+t-\frac{1}{2}}\,dx. $$
Decomposing $V=\varpi_\FF^n\calO_\FF=\{0\}\sqcup\bigsqcup_{i=n}^\infty\varpi_\FF^i \calO_\FF^{\times}$ and neglecting the set $\{0\}$ which is of measure zero, we rewrite this integral as
\begin{multline}
\label{equ::fourth_term}
	\sum\limits_{i=n}^\infty\sum\limits_{j=n}^\infty|\varpi_\FF^{2j}-\varpi_\FF^{2i}\alpha^2|_\FF^{-t-\frac{1}{2}}|\varpi_\FF^i|_\FF^{2s+t-\frac{1}{2}}\\
	\times\int_{\varpi_\FF^i \calO_\FF^{\times}}(\chi_0\eta_0)(x_2)\int_{\varpi_\FF^j \calO_\FF^{\times}}\chi\left(\frac{x_1+x_2\alpha}{x_1-x_2\alpha}\right)\eta_0(x_1^2-x_2^2\alpha^2)^{-1}\,dx_1\,dx_2.
\end{multline}

The double integral simplifies as follows:

\begin{lemma}\label{lem:ZeroIntegralx1x2}
For $i,j\in\ZZ$ we have
\begin{multline*}
	\int\limits_{\varpi_\FF^i \calO_\FF^{\times}} (\chi_0\eta_0)(x_2)   \int\limits_{\varpi_\FF^j \calO_\FF^{\times}} \chi\left(\frac{x_1+\alpha x_2}{x_1-\alpha x_2}\right)\eta_0(x_1^2-x_2^2\alpha^2)^{-1}\,dx_1\,dx_2\\
	= \begin{cases}q_\FF^{-2i}(1-q_\FF^{-1})\int_{\varpi_\FF^{j-i} \calO_\FF^{\times}}\chi_0(x_1-\alpha)^{-2}\,dx_1&\mbox{for $\chi|_{\calO_\FF^{\times}}\cdot\eta^{-1}=1$,}\\0&\mbox{for $\chi|_{\calO_\FF^{\times}}\cdot\eta^{-1}\neq1$.}\end{cases}
\end{multline*}
\end{lemma}

\begin{proof}
Writing
$$ \frac{x_1+x_2\alpha}{x_1-x_2\alpha} = \frac{\frac{x_1}{x_2}+\alpha}{\frac{x_1}{x_2}-\alpha} \qquad \mbox{and} \qquad x_1^2-x_2^2\alpha^2 = x_2^2\Big(\Big(\frac{x_1}{x_2}\Big)^2-\alpha^2\Big) $$
and substituting $a=\frac{x_1}{x_2}$ with $dx_1=|x_2|_\FF\,da$ yields
\begin{align}
	& \int\limits_{\varpi_F^i \calO_\FF^{\times}} (\chi_0\eta_0)(x_2)   \int\limits_{\varpi_F^j \calO_\FF^{\times}} \chi\left(\frac{x_1+\alpha x_2}{x_1-\alpha x_2}\right)\eta_0(x_1^2-x_2^2\alpha^2)^{-1}\,dx_1\,dx_2\notag\\
	={}& q_\FF^{-i}\int\limits_{\varpi_F^i \calO_\FF^{\times}} (\chi_0\eta_0^{-1})(x_2)\int\limits_{\varpi_F^{j-i}\calO_\FF^{\times}} \chi\left(\frac{x_1+\alpha}{x_1-\alpha}\right)\eta_0(x_1^2-\alpha^2)^{-1}\,dx_1\,dx_2\notag\\
	={}& q_\FF^{-i}\int\limits_{\varpi_F^{j-i} \calO_\FF^{\times}} \chi\left(\frac{x_1+\alpha}{x_1-\alpha}\right)\eta_0(x_1^2-\alpha^2)^{-1}\int\limits_{\varpi_F^i \calO_\FF^{\times}} (\chi_0\eta_0^{-1})(x_2)\,dx_2\,dx_1.\label{eq:ZeroIntegralx1x2A}
\end{align}
The inner integral can be evaluated as in the proof of Lemma~\ref{lem::zero_third_term}:
$$ \int\limits_{\varpi_\FF^i \calO_\FF^{\times}} (\chi_0\eta_0^{-1})(x_2)\,dx_2 = q_\FF^{-i}\int_{\calO_\FF^{\times}}(\chi\eta^{-1})(x_2)\,dx_2 = \begin{cases}q_\FF^{-i}(1-q_\FF^{-1})&\mbox{for $\chi|_{\calO_\FF^{\times}}\cdot\eta^{-1}=1$,}\\0&\mbox{for $\chi|_{\calO_\FF^{\times}}\cdot\eta^{-1}\neq1$.}\end{cases} $$
Assume that $\chi|_{\calO_\FF^{\times}}\cdot\eta^{-1}=1$, i.e. $\eta=\chi|_{\calO_\FF^{\times}}$, then \eqref{eq:ZeroIntegralx1x2A} becomes
\begin{multline*}
	q_\FF^{-2i}(1-q_\FF^{-1})\int_{\varpi_\FF^{j-i}\calO_\FF^{\times}}\chi\left(\frac{x_1+\alpha}{x_1-\alpha}\right)\eta_0(x_1^2-\alpha^2)^{-1}\,dx_1\\
	= q_\FF^{-2i}(1-q_\FF^{-1})\int_{\varpi_\FF^{j-i}\calO_\FF^{\times}}\chi_0(x_1-\alpha)^{-2}\,dx_1.\qedhere
\end{multline*}
\end{proof}

This shows Proposition~\ref{prop:ThirdTerm} for $\chi|_{\calO_\FF^{\times}}\cdot\eta^{-1}\neq1$ since in this case the integral vanishes. Now assume that $\chi|_{\calO_\FF^{\times}}\cdot\eta^{-1}=1$, then the previous lemma allows to rewrite \eqref{equ::fourth_term} as
\begin{align}
	& (1-q_\FF^{-1})\sum\limits_{i=n}^\infty\sum\limits_{j=n}^\infty|\varpi_\FF^{2j}-\varpi_\FF^{2i}\alpha^2|_\FF^{-t-\frac{1}{2}}|\varpi_\FF^i|_\FF^{2s+t+\frac{3}{2}}\int_{\varpi_\FF^{j-i}\calO_\FF^{\times}}\chi_0(x_1-\alpha)^{-2}\,dx_1\notag\\
	={}& (1-q_\FF^{-1})\sum\limits_{i=n}^\infty|\varpi_\FF^i|_\FF^{2s-t+\frac{1}{2}}\sum\limits_{j=n}^\infty\int_{\varpi_\FF^{j-i} \calO_\FF^{\times}}\chi_0(x_1-\alpha)^{-2}|x_1^2-\alpha^2|_\FF^{-t-\frac{1}{2}}\,dx_1\notag\\
	={}& (1-q_\FF^{-1})q_\FF^{-n(2s-t+\frac{1}{2})}\sum\limits_{i=0}^\infty q_\FF^{-i(2s-t+\frac{1}{2})}\sum_{m=-\infty}^i\int_{\varpi_\FF^{-m} \calO_\FF^{\times}}\chi_0(x_1-\alpha)^{-2}|x_1^2-\alpha^2|_\FF^{-t-\frac{1}{2}}\,dx_1.\label{equ::fourth_termB}
\end{align}

\begin{lemma}\label{lem:AnotherIntegral}
There exists $M>0$ such that for $k\geq M$:
\begin{align*}
	\int_{\varpi_\FF^{-k} \calO_\FF^{\times}}\chi_0(x_1-\alpha)^{-2}\,dx_1 &= \begin{cases}(1-q_\FF^{-1})q_\FF^{k}&\mbox{for $\chi^2|_{\calO_\FF^{\times}}=1$,}\\0&\mbox{for $\chi^2|_{\calO_\FF^{\times}}\neq1$.}\end{cases}
\end{align*}
\end{lemma}

\begin{proof}
Since $\chi_0$ is locally constant on $\calO_\EE^{\times}$ (see e.g. \cite[\S1.4]{GG63}), there exists $M>0$ such that $\chi_0(1+y)=\chi_0(1)=1$ for all $y\in\varpi_{\EE}^M\calO_\EE$. After possibly replacing $M$ by $2M$, we have $\chi_0(1+y)=1$ for all $y\in\varpi_\FF^M\calO_\EE$. Recall that we assume $\alpha\in\calO_\EE$. Then $k\geq M$ implies $\varpi_\FF^k\alpha\in\varpi_\FF^M\calO_\EE$, so for all $x_1\in\varpi_\FF^{-k} \calO_\FF^{\times}$ we obtain
$$ \chi_0(x_1-\alpha) =\chi_0(x_1)\chi_0(1-x_1^{-1}\alpha) = \chi_0(x_1). $$
This allows us to compute the integral for $k\geq M$:
\begin{equation*}
	\int_{\varpi_\FF^{-k} \calO_\FF^{\times}}\chi_0(x_1-\alpha)^{-2}\,dx_1 = \int_{\varpi_\FF^{-k} \calO_\FF^{\times}}\chi_0(x_1)^{-2}\,dx_1 = q_\FF^{k}\int_{\calO_\FF^{\times}}\chi(u)^{-2}\,du,
\end{equation*}
which is the integral of the character $\chi^{-2}|_{ \calO_\FF^{\times}}$ over the compact group $\calO_\FF^{\times}$ (see \cite[Lemma 4.1]{Sally98}).
\end{proof}

Note that under the assumption $\chi|_{\calO_\FF^{\times}}\cdot\eta^{-1}=1$, the condition $\chi^2|_{\calO_\FF^{\times}}=1$ is equivalent to $\chi|_{\calO_\FF^{\times}}\cdot\eta=1$. In view of the previous lemma, we split the summation in \eqref{equ::fourth_termB} into
\begin{equation*}
\begin{split}
 \sum_{i=0}^\infty\sum_{m=-\infty}^i &= \Bigg(\sum_{i=0}^{M-1} + \sum_{i=M}^\infty \Bigg)\sum_{m=-\infty}^i = \sum_{i=0}^{M-1}\Bigg( \sum_{m=-\infty}^{M-1} - \sum_{m=i+1}^{M-1}  \Bigg) + \sum_{i=M}^\infty \Bigg(  \sum_{m=-\infty}^{M-1} + \sum_{m=M}^{i}  \Bigg) \\
&=\sum_{i=0}^\infty\sum_{m=-\infty}^{M-1}-\sum_{i=0}^{M-2}\sum_{m=i+1}^{M-1}+\sum_{i=M}^\infty\sum_{m=M}^i.
\end{split}
\end{equation*}
(Note that in the second double sum, the term for $i=M-1$ vanishes since in this case $\sum_{m=i+1}^{M-1}=\sum_{m=M}^{M-1}=0$.) Then \eqref{equ::fourth_termB} becomes
\begin{align}
	&(1-q_\FF^{-1})q_\FF^{-n(2s-t+\frac{1}{2})}\Bigg(\sum_{i=0}^\infty q_\FF^{-i(2s-t+\frac{1}{2})}\int_{\varpi_\FF^{-(M-1)}\calO_\FF}\chi_0(x_1-\alpha)^{-2}|x_1^2-\alpha^2|_{\FF}^{-t-\frac{1}{2}}\,dx_1\notag\\
	&\hspace{3cm}-\sum_{i=0}^{M-2} q_\FF^{-i(2s-t+\frac{1}{2})}\sum_{m=i+1}^{M-1}\int_{\varpi_\FF^{-m} \calO_\FF^{\times}}\chi_0(x_1-\alpha)^{-2}|x_1^2-\alpha^2|_\FF^{-t-\frac{1}{2}}\,dx_1\notag\\
	&\hspace{3cm}+\sum_{i=M}^\infty q_\FF^{-i(2s-t+\frac{1}{2})}\sum_{m=M}^i\int_{\varpi_\FF^{-m} \calO_\FF^{\times}}\chi_0(x_1-\alpha)^{-2}|x_1^2-\alpha^2|_\FF^{-t-\frac{1}{2}}\,dx_1\Bigg).\label{equ::fourth_termC}
\end{align}

The second term in \eqref{equ::fourth_termC} is a finite sum and hence holomorphic in $s,t\in\CC$. The first term becomes
$$ \frac{(1-q_\FF^{-1})q_\FF^{-n(2s-t+\frac{1}{2})}}{1-q_\FF^{-(2s-t+\frac{1}{2})}}\int_{\varpi_\FF^{-(M-1)}\calO_\FF}\chi_0(x_1-\alpha)^{-2}|x_1^2-\alpha^2|_{\FF}^{-t-\frac{1}{2}}\,dx_1 $$
which is holomorphic in $\{(s,t)\in\CC^2:2s-t+\frac{1}{2}\not\in\frac{2\pi i}{\ln q_\FF}\ZZ\}$ with the obvious residue at $1-q_\FF^{-(2s-t+\frac{1}{2})}=0$. Finally, the third term vanishes for $\chi|_{\calO_\FF^{\times}}\cdot\eta\neq1$, and for $\chi|_{\calO_\FF^{\times}}\cdot\eta=1$ it becomes
\begin{multline*}
	(1-q_\FF^{-1})^2q_\FF^{-n(2s-t+\frac{1}{2})}\sum_{i=M}^\infty q_\FF^{-i(2s-t+\frac{1}{2})}\sum_{m=M}^iq_\FF^{-2tm}\\
	= (1-q_\FF^{-1})^2q_\FF^{-n(2s-t+\frac{1}{2})}q_\FF^{-M(2s+t+\frac{1}{2})}\frac{1}{(1-q_\FF^{-(2s-t+\frac{1}{2})})(1-q_\FF^{-(2s+t+\frac{1}{2})})}.
\end{multline*}
This shows the holomorphicity statements in Proposition~\ref{prop:ThirdTerm}. It remains to prove the claimed residue formulas.

Following the previous discussion, multiplying \eqref{equ::fourth_termC} with $(1-q_\FF^{-(2s-t+\frac{1}{2})})$ and setting $2s-t+\frac{1}{2}=0$ yields
$$ (1-q_\FF^{-1})\int_{\varpi_\FF^{-(M-1)}\calO_\FF}\chi_0(x_1-\alpha)^{-2}|x_1^2-\alpha^2|_{\FF}^{-t-\frac{1}{2}}\,dx_1 $$
if $\chi|_{\calO_\FF^{\times}}\cdot\eta\neq1$, and in the case $\chi|_{\calO_\FF^{\times}}\cdot\eta=1$ additionally
$$ (1-q_\FF^{-1})\frac{(1-q_\FF^{-1})q_\FF^{-M(2s+t+\frac{1}{2})}}{1-q_\FF^{-(2s+t+\frac{1}{2})}}. $$
In view of the choice of $M$ in Lemma~\ref{lem:AnotherIntegral}, the fact that $2s+t+\frac{1}{2}=2t$, the previous computations, and the change of variables $x_1= -y\alpha^2$, this expression equals (in each of the cases $\chi|_{\calO_\FF^{\times}}\cdot\eta\neq1$ or $\chi|_{\calO_\FF^{\times}}\cdot\eta=1$)
\begin{multline*}
	(1-q_\FF^{-1})\int_\FF \chi_0(x_1-\alpha)^{-2}|x_1^2-\alpha^2|_\FF^{-t-\frac{1}{2}}\,dx_1\\
	= (1-q_\FF^{-1})\chi_0(-\alpha)^{-2}|\alpha^2|_\FF^{-t+\frac{1}{2}}\int_\FF \chi_0(1+y\alpha)^{-2}|1+y\alpha|_\EE^{-t-\frac{1}{2}}\,dy.
\end{multline*}
Evaluating the remaining integral with \eqref{prop:GammaIntegral} shows \eqref{eq:ThirdTermResidue-}.

On the other hand, multiplying \eqref{equ::fourth_termC} with $(1-q_\FF^{-(2s+t+\frac{1}{2})})$ and setting $2s+t+\frac{1}{2}=0$ yields
$$ \frac{(1-q_\FF^{-1})^2q_\FF^{-n(2s-t+\frac{1}{2})}}{1-q_\FF^{-(2s-t+\frac{1}{2})}}. $$
Since in this case $2s-t+\frac{1}{2}=-2t$, this shows \eqref{eq:ThirdTermResidue+}.

\subsubsection{Holomorphic renormalization}
\label{subsec::holom_renor}

For a character $\eta_t$ of $\FF^\times$, $\eta\in\widehat{\calO_\FF^\times}$, $t\in\CC$, recall the local non-archimedean $L$-factor
\begin{equation}
	L(s,\eta_t) = \begin{cases}(1-q_\FF^{-s-t})^{-1}&\mbox{if $\eta=1$,}\\1&\mbox{if $\eta\neq1$.}\end{cases} \qquad\qquad (s\in\CC).\label{eq:DefLocalLFactor}
\end{equation}

We define the following renormalization of $K_{s,t}^{\chi,\eta}$:
$$ \widetilde{K}_{s,t}^{\chi,\eta} := L(\tfrac{1}{2},\chi_s|_{\FF^\times}\cdot\eta_t)^{-1}L(\tfrac{1}{2},\chi_s|_{\FF^\times}\cdot\eta_t^{-1})^{-1}\cdot K_{s,t}^{\chi,\eta}, $$
or more explicitly:
$$ \widetilde{K}_{s,t}^{\chi,\eta} = K_{s,t}^{\chi,\eta}\times\begin{cases}(1-q_\FF^{-(2s+t+\frac{1}{2})})(1-q_\FF^{-(2s-t+\frac{1}{2})})&\mbox{if $\chi|_{\calO_\FF^{\times}}\cdot\eta=1$ and $\chi|_{\calO_\FF^{\times}}\cdot\eta^{-1}=1$,}\\(1-q_\FF^{-(2s+t+\frac{1}{2})})&\mbox{if $\chi|_{\calO_\FF^{\times}}\cdot\eta=1$ and $\chi|_{\calO_\FF^{\times}}\cdot\eta^{-1}\neq1$,}\\(1-q_\FF^{-(2s-t+\frac{1}{2})})&\mbox{if $\chi|_{\calO_\FF^{\times}}\cdot\eta\neq1$ and $\chi|_{\calO_\FF^{\times}}\cdot\eta^{-1}=1$,}\\1&\mbox{if $\chi|_{\calO_\FF^{\times}}\cdot\eta\neq1$ and $\chi|_{\calO_\FF^{\times}}\cdot\eta^{-1}\neq1$.}\end{cases} $$
With this renormalization we obtain the following result:

\begin{theorem}\label{thm:HolomorphicRenormalization}
	For all $\chi\in\widehat{\calO_\EE^{\times}}$ and $\eta\in\widehat{\calO_\FF^{\times}}$ the map $(s,t)\mapsto\widetilde{K}_{s,t}^{\chi,\eta}$ extends to a holomorphic function from $\CC^2$ to $\calD'(\EE)$, i.e. for every $\varphi\in C_c^\infty(\EE)$ the function $\CC^2\to\CC,\,(s,t)\mapsto\langle\widetilde{K}_{s,t}^{\chi,\eta},\varphi\rangle$ is holomorphic. Moreover, $\widetilde{K}_{s,t}^{\chi,\eta}\in\calD'(\EE)_{s,t}^{\chi,\eta}$ for all $s,t\in\CC$.
\end{theorem}

\begin{proof}	
	Propositions~\ref{prop:FirstTerm}, \ref{prop:SecondTerm} and \ref{prop:ThirdTerm} applied to \eqref{eq:MeroContThreeTerms} show that $\langle\widetilde{K}_{s,t}^{\chi,\eta},\varphi\rangle$ is holomorphic in $(s,t)\in\CC^2$ for every $\varphi\in C_c^\infty(\EE)$. It remains to show that $\widetilde{K}_{s,t}^{\chi,\eta}\in\calD'(\EE)_{s,t}^{\chi,\eta}$ for all $s,t\in\CC$. First note that by the considerations in Section~\ref{sec:ClassSBOKernelsEtimes}, the restriction of $K_{s,t}^{\chi,\eta}$ to $\EE^\times$ satisfies the conditions of Proposition~\ref{prop:CharSBOKernels}. For $(s,t)\in\CC^2$ with $\Re(2s\pm t+\frac{1}{2})>0$ the distribution $K_{s,t}^{\chi,\eta}$ is locally integrable by Lemma~\ref{lem:KernelL1loc}, and since $\FF$ is of measure zero in $\EE$, the conditions of Proposition~\ref{prop:CharSBOKernels} also hold on $\EE$ in the distribution sense. Consequently, the renormalized kernel $\widetilde{K}_{s,t}^{\chi,\eta}$ satisfies the conditions of Proposition~\ref{prop:CharSBOKernels} whenever $\Re(2s\pm t+\frac{1}{2})>0$, but since it is holomorphic in $(s,t)\in\CC^2$ and the conditions of Proposition~\ref{prop:CharSBOKernels} depend holomorphically on $(s,t)$, it satisfies the conditions of Proposition~\ref{prop:CharSBOKernels} for all $(s,t)\in\CC^2$, i.e. $\widetilde{K}_{s,t}^{\chi,\eta}\in\calD'(\EE)_{s,t}^{\chi,\eta}$.
\end{proof}

Using the residue calculations from Propositions~\ref{prop:SecondTerm} and \ref{prop:ThirdTerm} we now find the values of $\widetilde{K}_{s,t}^{\chi,\eta}$ at $(\chi,s,\eta,t)\in \SminusT\cup \SplusT$. These special values turn out to be scalar multiples of the Dirac distribution $\delta\in\calD'(\EE)$ at $x=0$ and a holomorphic family of distributions that is supported on $\FF\subseteq\EE$ which we now define. For $\xi\in\widehat{\calO_\FF^\times}$ recall from Appendix~\ref{app:HomogeneousDistributions} the nowhere vanishing holomorphic family of distributions $\widetilde{\xi}_\alpha\in\calD'(\FF)$, $\alpha\in\CC$. For $\eta\in\widehat{\calO_\FF^\times}$ we define a family of distributions $L_t^\eta$ on $\EE$ by
$$ L_t^\eta(x_1+x_2\alpha) = \widetilde{(\eta^{-2})}_{-2t-1}(x_1)\cdot\delta(x_2) \qquad (t\in\CC). $$
Note that the restriction of $L_t^\eta$ to $\EE^\times$ is (up to a scalar) the distribution that occurs in Proposition~\ref{prop:CharacterizationSBOKernelsPGL2} in the case $(\chi,s,\eta,t)\in\SplusT\setminus\SminusT$. By Theorem~\ref{thm:HomogeneousDistributions}, $L_t^\eta$ depends holomorphically on $t\in\CC$ and
$$ L_t^\eta(x_1+x_2\alpha) = \begin{cases}\eta_{t+\frac{1}{2}}(x_1^2)^{-1}\cdot\delta(x_2)&\mbox{if $\eta^2\neq1$},\\(1-q_\FF^{2t})\eta_{t+\frac{1}{2}}(x_1^2)^{-1}\cdot\delta(x_2)&\mbox{if $\eta^2=1$ and $2t\not\in\frac{2\pi i}{\ln q_\FF}\ZZ$,}\\(1-q_\FF^{-1})\delta(x_1)\cdot\delta(x_2)&\mbox{if $\eta^2=1$ and $2t\in\frac{2\pi i}{\ln q_\FF}\ZZ$.}\end{cases} $$

\begin{theorem}[Residue Identities]\label{thm:ResidueIdentities}
	\begin{enumerate}
		\item\label{thm:ResidueIdentities1} For $(\chi,s,\eta,t)\in\SplusT$ we have
		$$ \widetilde{K}_{s,t}^{\chi,\eta} = (1-q_\FF^{-1})\cdot L_t^\eta. $$
		\item\label{thm:ResidueIdentities2} For $(\chi,s,\eta,t)\in\SminusT$ we have
		\begin{multline*}
			\qquad\qquad\widetilde{K}_{s,t}^{\chi,\eta} = c_\FF(1-q_\FF^{-1})\chi_0(-\alpha)^{-2}|\alpha^2|_\FF^{-t+\frac{1}{2}}\\
			\times\Gamma((\chi^2|_{\calO^\times_\FF})_{2t})\Gamma((\chi^{-2})_{-t+\frac{1}{2}})L(\tfrac{1}{2},\chi_s|_{\FF^\times}\cdot\eta_t)^{-1}\cdot\delta,
		\end{multline*}
		where the constant $c_\FF>0$ and the gamma factors are as in Appendix~\ref{app:Integral}.
	\end{enumerate}
\end{theorem}

We remark that for $(\chi,s,\eta,t)\in\SplusT\cap\SminusT$ the expressions in Theorem~\ref{thm:ResidueIdentities}~\eqref{thm:ResidueIdentities1} and \eqref{thm:ResidueIdentities2} agree.

\begin{proof}
These results follow by putting together Propositions~\ref{prop:FirstTerm}, \ref{prop:SecondTerm} and \ref{prop:ThirdTerm}. It only remains to match the formulas in Propositions~\ref{prop:SecondTerm} and \ref{prop:ThirdTerm} in the case $(\chi,s,\eta,t)\in\SplusT$. More precisely, for $(\chi,s,\eta,t)\in\SplusT$ we have to show that
$$ (1-q_\FF^{-1})\int_V\eta_{t+\frac{1}{2}}(x_1^2)^{-1}\varphi(x_1)\,dx_1 = \varphi(0)\times\begin{cases}\frac{(1-q_\FF^{-1})^2q_\FF^{2tn}}{1-q_\FF^{2t}}&\mbox{if $\chi|_{\calO^\times_\FF}\cdot\eta^{-1}=1$,}\\0&\mbox{if $\chi|_{\calO^\times_\FF}\cdot\eta^{-1}\neq1$.}\end{cases} $$ 
Since $\varphi(x_1)=\varphi(0)$ for $x_1\in V$, this amounts to computing the integral
$$ \int_V\eta_{t+\frac{1}{2}}(x_1^2)^{-1}\,dx_1 $$
with $V=\varpi_\FF^n\calO_\FF$ which is similar to the computations in the previous subsections.
\end{proof}

Now we can determine the zeros of the function $(s,t)\mapsto\widetilde{K}_{s,t}^{\chi,\eta}$ and the support of $\widetilde{K}_{s,t}^{\chi,\eta}$ for all parameters, where $\widetilde{K}_{s,t}^{\chi,\eta}$ is seen as a distribution in $\calD'(\EE)_{s,t}^{\chi,\eta}$. To simplify the statement, we put
$$ L = \{(\chi,s,\eta,t):(s,t)\in L^{\chi,\eta}\}, $$
where $L^{\chi,\eta}$ is a subset of $\CC^2/(\frac{2\pi i}{\ln q_\EE}\ZZ\times\frac{2\pi i}{\ln q_\FF}\ZZ)$. (Recall that $\widetilde{K}_{s,t}^{\chi,\eta}$ only depends on $(s,t)$ modulo $\frac{2\pi i}{\ln q_\EE}\ZZ\times\frac{2\pi i}{\ln q_\FF}\ZZ$.) If $\chi|_{\calO_\FF^\times}\cdot\eta\neq1$ or $\chi|_{\calO_\FF^\times}\cdot\eta^{-1}\neq1$ we define
$$ L^{\chi,\eta}=\emptyset $$
and for $\chi|_{\calO_\FF^\times}\cdot\eta=\chi|_{\calO_\FF^\times}\cdot\eta^{-1}=1$ we put
\begin{equation}
\label{equ::L}
 L^{\chi,\eta} = \begin{cases}\{(0,\frac{1}{2}),(\frac{\pi i}{\ln q_\EE},\frac{1}{2}+\frac{\pi i}{\ln q_\FF})\}&\mbox{if $\chi^2\neq1$ and $\EE/\FF$ is unramified,}\\
\begin{array}{ll}\hspace{-.18cm}\{(0,\frac{1}{2}),(\frac{\pi i}{\ln q_\EE},\frac{1}{2}),\\\hspace{.3cm}(\frac{\pi i}{2\ln q_\EE},\frac{1}{2}+\frac{\pi i}{\ln q_\FF}),(\frac{3\pi i}{2\ln q_\EE},\frac{1}{2}+\frac{\pi i}{\ln q_\FF})\}\end{array}&\mbox{if $\chi^2\neq1$ and $\EE/\FF$ is ramified,}\\\{(-\frac{1}{2},-\frac{1}{2}),(-\frac{1}{2}+\frac{\pi i}{\ln q_\EE},-\frac{1}{2}+\frac{\pi i}{\ln q_\FF})\}&\mbox{if $\chi^2=1$ and $\EE/\FF$ is unramified}\\\begin{array}{ll}\hspace{-.18cm}\{(-\frac{1}{2},-\frac{1}{2}),(-\frac{1}{2}+\frac{\pi i}{\ln q_\EE},-\frac{1}{2}),\\\hspace{.3cm}(\frac{\pi i}{2\ln q_\EE},\frac{1}{2}+\frac{\pi i}{\ln q_\FF}),(\frac{3\pi i}{2\ln q_\EE},\frac{1}{2}+\frac{\pi i}{\ln q_\FF})\}\end{array}&\mbox{if $\chi^2=1$ and $\EE/\FF$ is ramified.}\end{cases} 
\end{equation}
Note that $L\subseteq\SminusT$ and $L\cap\SplusT=\emptyset$.

\begin{theorem}\label{thm:ZerosAndSupportOfKernel}
$\widetilde{K}_{s,t}^{\chi,\eta}=0$ if and only if $(\chi,s,\eta,t)\in L$. Moreover, we have
$$ \supp\widetilde{K}_{s,t}^{\chi,\eta} = \begin{cases}\emptyset&\mbox{for $(\chi,s,\eta,t)\in L$,}\\\{0\}&\mbox{for $(\chi,s,\eta,t)\in\SminusT\setminus L$,}\\\FF&\mbox{for $(\chi,s,\eta,t)\in\SplusT\setminus\SminusT$,}\\\EE&\mbox{else.}\end{cases} $$
\end{theorem}

\begin{proof}
	First note that since $K_{s,t}^{\chi,\eta}|_{\EE\setminus\FF}$ is a locally integrable and nowhere vanishing function, we have $\supp\widetilde{K}_{s,t}^{\chi,\eta}=\EE$ unless we have renormalized by $(1-q_\FF^{-(2s+t+\frac{1}{2})})$ or $(1-q_\FF^{-(2s-t+\frac{1}{2})})$ and one of these factors vanishes, i.e. $(\chi,s,\eta,t)\in\SminusT\cup \SplusT$. In these two cases, Theorem~\ref{thm:ResidueIdentities} expresses $\widetilde{K}_{s,t}^{\chi,\eta}$ as a constant multiple of the Dirac distribution $\delta$ or the nowhere vanishing family $L_t^\eta$ whose support is either $\FF$ or $\{0\}$. It remains to check whether the constant vanishes. In Theorem~\ref{thm:ResidueIdentities}~\eqref{thm:ResidueIdentities1} the constant $(1-q_\FF^{-1})$ is obviously non-zero, so $\widetilde{K}_{s,t}^{\chi,\eta}\neq0$ whenever $(\chi,s,\eta,t)\in\SplusT$. For $(\chi,s,\eta,t)\in\SminusT\setminus \SplusT$ Theorem~\ref{thm:ResidueIdentities}~\eqref{thm:ResidueIdentities2} shows that $\widetilde{K}_{s,t}^{\chi,\eta}=0$ if and only if
	$$ \Gamma((\chi^2|_{\calO^\times_\FF})_{2t})\Gamma((\chi^{-2})_{-t+\frac{1}{2}})\times\begin{cases}1&\mbox{if $\chi|_{\calO^\times_\FF}\cdot\eta\neq1$},\\(1-q_\FF^{-(2s+t+\frac{1}{2})})&\mbox{if $\chi|_{\calO^\times_\FF}\cdot\eta=1$ and $2s+t+\frac{1}{2}\notin\frac{2\pi i}{\ln q_\FF}\ZZ$,}\end{cases} $$
	vanishes. To understand poles and zeros of the gamma factors, we distinguish three cases and apply Lemma~\ref{lem:PropertiesGammaFactors}:
\begin{enumerate}[(1)]
	\item $\chi^2|_{\calO^\times_\FF}\neq1$. In this case, both gamma factors are holomorphic in $t$ and nowhere vanishing.
	\item $\chi^2|_{\calO^\times_\FF}=1$ and $\chi^2\neq1$. Here, the second gamma factor is holomorphic and nowhere vanishing, and the first gamma factor has simple poles at $t\in\frac{\pi i}{\ln q_\FF}\ZZ$ and simple zeros at $t\in\frac{1}{2}+\frac{\pi i}{\ln q_\FF}\ZZ$. Since $2s-t+\frac{1}{2}\in\frac{2\pi i}{\ln q_\FF}\ZZ$ and $2s+t+\frac{1}{2}\notin\frac{2\pi i}{\ln q_\FF}\ZZ$, the poles at $t\in\frac{\pi i}{\ln q_\FF}\ZZ$ do not appear.
	\item $\chi^2=1$. In addition to the zeros of the first gamma factor, the second gamma factor has simple poles at $t\in\frac{1}{2}+\frac{2\pi i}{\ln q_\EE}\ZZ$ and simple zeros at $t\in-\frac{1}{2}+\frac{2\pi i}{\ln q_\EE}\ZZ$. Note that for $q_\EE=q_\FF^2$ (i.e. $\EE/\FF$ is unramified) the poles cancel the zeros of the first gamma factor, while for $q_\EE=q_\FF$ (i.e. $\EE/\FF$ is ramified) they only cancel half of them.
\end{enumerate}
	Since $\chi^2|_{\calO^\times_\FF}=1$ is equivalent to $\chi|_{\calO^\times_\FF}\cdot\eta=1$, this gives the following description of the set $L^{\chi,\eta}$ of all $(s,t)\in\CC^2$ such that $\widetilde{K}_{s,t}^{\chi,\eta}=0$: $L^{\chi,\eta}=\emptyset$ if $\chi|_{\calO_\FF^\times}\cdot\eta\neq1$ or $\chi_{\calO_\FF^\times}\cdot\eta^{-1}\neq1$, and in the case $\chi|_{\calO_\FF^\times}\cdot\eta=\chi_{\calO_\FF^\times}\cdot\eta^{-1}=1$ we have
	$$ L^{\chi,\eta} = \begin{cases}\{(s,t):2s-t+\frac{1}{2}\in\frac{2\pi i}{\ln q_\FF}\ZZ,t\in\frac{1}{2}+\frac{\pi i}{\ln q_\FF}\ZZ\}&\mbox{if $\chi^2\neq1$,}\\\{(s,t):2s-t+\frac{1}{2}\in\frac{2\pi i}{\ln q_\FF}\ZZ,t\in-\frac{1}{2}+\frac{\pi i}{\ln q_\FF}\ZZ\}&\mbox{if $\chi^2=1$ and $q_\EE=q_\FF^2$,}\\\begin{array}{ll}\hspace{-.18cm}\{(s,t):2s-t+\frac{1}{2}\in\frac{2\pi i}{\ln q_\FF}\ZZ,t\in\frac{1}{2}+\frac{\pi i}{\ln q_\FF}(2\ZZ+1)\}\\\cup\,\{(s,t):2s-t+\frac{1}{2}\in\frac{2\pi i}{\ln q_\FF}\ZZ,t\in-\frac{1}{2}+\frac{2\pi i}{\ln q_\FF}\ZZ\}\end{array}&\mbox{if $\chi^2=1$ and $q_\EE=q_\FF$.}\end{cases} $$
	This is the same set as defined above.
\end{proof}

\section{Classification of symmetry breaking kernels for $\PGL(2)$}
\label{sec::class_sym}

In this section we determine the space $\calD'(\EE)_{s,t}^{\chi,\eta}$ of symmetry breaking kernels explicitly for all $(\chi,s,\eta,t)$, following the methods of \cite[Chapter 11]{KS15}.

Recall from Section~\ref{sec:ClassSBOKernelsEtimes} the space $\calD'(\EE^\times)_{s,t}^{\chi,\eta}$ of all distributions on $\EE^\times$ satisfying the same conditions as the distributions in $\calD'(\EE)_{s,t}^{\chi,\eta}$, but only on $\EE^\times$. By Proposition~\ref{prop:CharacterizationSBOKernelsPGL2} the space $\calD'(\EE^\times)_{s,t}^{\chi,\eta}$ is one-dimensional for all $(\chi,s,\eta,t)$. To understand the difference between $\calD'(\EE^\times)_{s,t}^{\chi,\eta}$ and $\calD'(\EE)_{s,t}^{\chi,\eta}$ we consider the restriction map $\calD'(\EE)_{s,t}^{\chi,\eta}\to\calD'(\EE^\times)_{s,t}^{\chi,\eta}$. Its kernel is equal to the space $\calD_{\{0\}}'(\EE)_{s,t}^{\chi,\eta}$ of distributions in $\calD'(\EE)_{s,t}^{\chi,\eta}$ whose support is contained in $\{0\}$. This yields an exact sequence
\begin{equation}
	0 \to \calD_{\{0\}}'(\EE)_{s,t}^{\chi,\eta} \to \calD'(\EE)_{s,t}^{\chi,\eta} \to \calD'(\EE^\times)_{s,t}^{\chi,\eta}\label{eq:ExactSequence}
\end{equation}
and hence an upper bound for the dimension of $\calD'(\EE)_{s,t}^{\chi,\eta}$:
\begin{equation}
	\dim\calD'(\EE)_{s,t}^{\chi,\eta} \leq \dim\calD_{\{0\}}'(\EE)_{s,t}^{\chi,\eta} + \dim\calD'(\EE^\times)_{s,t}^{\chi,\eta} = \dim\calD_{\{0\}}'(\EE)_{s,t}^{\chi,\eta} + 1.\label{eq:MultiplicityInequality}
\end{equation}
The following lemma makes the upper bound more explicit by determining $\calD_{\{0\}}'(\EE)_{s,t}^{\chi,\eta}$, and it also provides a lower bound based on the existence of the non-trivial holomorphic family $\widetilde{K}_{s,t}^{\chi,\eta}$:

\begin{lemma}\label{lem:DimensionEstimates}
	\begin{enumerate}
		\item\label{lem:DimensionEstimates1} $\displaystyle\calD_{\{0\}}'(\EE)_{s,t}^{\chi,\eta} = \begin{cases}\CC\delta&\mbox{if $(\chi,s,\eta,t)\in\SminusT$,}\\\{0\}&\mbox{else.}\end{cases}$
		\item\label{lem:DimensionEstimates2} $\dim\calD'(\EE)_{s,t}^{\chi,\eta}\geq1$ for all $s,t\in\CC$, $\chi\in\widehat{\calO^\times_\EE}$, $\eta\in\widehat{\calO^\times_\FF}$.
	\end{enumerate}
\end{lemma}

\begin{proof}
Since the space of distributions with support in $\{0\}$ is spanned by $\delta$, it suffices for \eqref{lem:DimensionEstimates1} to show that $\delta$ satisfies the conditions in Proposition~\ref{prop:CharSBOKernels} if and only if $(\chi,s,\eta,t)\in\SminusT$. For this we use the reformulation in Remark~\ref{rem:CharSBOKernelsDistributionLanguage}. First note that $\delta$ always satisfies \eqref{eq:AltCharSBOKernels2} since for any $\varphi\in C_c^\infty(\EE\setminus \{b^{-1}\})$:
$$ \langle \delta,|1+b\cdot|_\EE^2\varphi(\tfrac{(\cdot)}{1+b(\cdot)})\rangle = \varphi(0) = \langle\delta,\chi_{s-\frac{1}{2}}((1-b(\cdot))^2)^{-1}\varphi\rangle. $$
Moreover, \eqref{eq:AltCharSBOKernels2} is satisfied by $\delta$ if and only if $\chi_{s+\frac{1}{2}}|_{\calO_\FF^{\times}}\cdot\eta_{t+\frac{1}{2}}^{-1}=1$, i.e. $\chi|_{\calO_\FF^{\times}}\cdot\eta^{-1}=1$ and $2(s+\frac{1}{2})-(t+\frac{1}{2})\in\frac{2\pi i}{\ln q_\FF}\ZZ$. This is equivalent to $(\chi,s,\eta,t)\in\SminusT$.

\eqref{lem:DimensionEstimates2} is clear for all parameters such that $\widetilde{K}_{s,t}^{\chi,\eta}\neq0$. But if $\widetilde{K}_{s,t}^{\chi,\eta}=0$ for some $(s,t)=(s_0,t_0)$, we can consider the restriction of the holomorphic function $(s,t)\mapsto\widetilde{K}_{s,t}^{\chi,\eta}$ to a line through $(s_0,t_0)$, say $s\mapsto\widetilde{K}_{s,t_0}^{\chi,\eta}$. This function vanishes at $s=s_0$, but is not identically zero by Theorem~\ref{thm:ZerosAndSupportOfKernel}. It follows that for some $k\geq0$, the function $s\mapsto(s-s_0)^{-k}\widetilde{K}_{s,t_0}^{\chi,\eta}$ is holomorphic and non-zero at $s=s_0$. Its value at $s=s_0$ is contained in $\calD'(\EE)_{s_0,t_0}^{\chi,\eta}$ since the conditions in Proposition~\ref{prop:CharSBOKernels} depend continuously on $s,t\in\CC$.
\end{proof}

Together with Lemma~\ref{lem:DimensionEstimates}, the inequality \eqref{eq:MultiplicityInequality} implies that $\dim\calD'(\EE)_{s,t}^{\chi,\eta}=1$ for $(\chi,s,\eta,t)\notin\SminusT$. By Theorem~\ref{thm:ZerosAndSupportOfKernel} we have $\widetilde{K}_{s,t}^{\chi,\eta}\neq0$ in this case, so $\calD'(\EE)_{s,t}^{\chi,\eta}=\CC\widetilde{K}_{s,t}^{\chi,\eta}$ for $(\chi,s,\eta,t)\notin\SminusT$. For $(\chi,s,\eta,t)\in\SminusT$ we study the restriction map $\calD'(\EE)_{s,t}^{\chi,\eta}\to\calD'(\EE^\times)_{s,t}^{\chi,\eta}$.

\begin{lemma}
\label{lem::restriction_map_zero}
	The restriction map $ \calD'(\EE)_{s,t}^{\chi,\eta} \to \calD'(\EE^\times)_{s,t}^{\chi,\eta} $ is 
        identically zero if $(\chi,s,\eta,t)\in\SminusT\setminus L$.
 \end{lemma}

\begin{proof}
	The proof follows the ideas of \cite[Chapter 11.4]{KS15}.

	For $(\chi,s,\eta,t)\in\SminusT$ the normalizing factor $L(\frac{1}{2},\chi_s|_{\FF^\times}\cdot\eta_t^{-1})$ of $\widetilde{K}_{s,t}^{\chi,\eta}$ has a pole of order one. From Theorem~\ref{thm:HolomorphicRenormalization} it follows that $\widetilde{K}_{s,t}^{\chi,\eta}$ is holomorphic in $s$ and $t$, so the meromorphic function
	$$ z\mapsto u_z = L(\tfrac{1}{2},\chi_{s+z}|_{\FF^\times}\cdot\eta_{t-2z}^{-1})\cdot\widetilde{K}_{s+z,t-2z}^{\chi,\eta} = L(\tfrac{1}{2},\chi_{s+z}|_{\FF^\times}\cdot\eta_{t-2z})^{-1}\cdot K_{s+z,t-2z}^{\chi,\eta} $$
	has at most a pole of order one at $z=0$ and hence it can be expanded as
	$$ u_z = u_{-1}z^{-1}+u_0+u_1z+\cdots $$
	with $u_j\in\calD'(\EE)$. Since $(\chi,s,\eta,t)\not\in L$, Theorem~\ref{thm:ZerosAndSupportOfKernel} implies $\supp u_{-1}=\{0\}$. Hence, $u_z|_{\EE^\times}$ is holomorphic at $z=0$ and $u_0|_{\EE^\times}=(u_z|_{\EE^\times})|_{z=0}$ is the following distribution:
	\begin{align*}
		u_0|_{\EE^\times}(x_1+x_2\alpha) &= L(\tfrac{1}{2},\chi_s|_{\FF^\times}\cdot\eta_t)^{-1}\cdot K_{s,t}^{\chi,\eta}|_{\EE^\times}(x_1+x_2\alpha)\\
		&= \chi_{s-\frac{1}{2}}(x^2)(\chi_{s-\frac{1}{2}}\eta_{t+\frac{1}{2}})(x_1^2-x_2^2\alpha^2)^{-1}\cdot L(\tfrac{1}{2},\chi_s|_{\FF^\times}\cdot\eta_t)^{-1}(\chi_{s-\frac{1}{2}}\eta_{t+\frac{1}{2}})(x_2).
	\end{align*}
	Note that the first factor $\chi_{s-\frac{1}{2}}(x^2)(\chi_{s-\frac{1}{2}}\eta_{t+\frac{1}{2}})(x_1^2-x_2^2\alpha^2)^{-1}$ is smooth and nowhere vanishing on $\EE^\times$ and the second factor is non-zero by Theorem~\ref{thm:HomogeneousDistributions}. It follows that $u_0|_{\EE^\times}$	is a non-zero element of $\calD'(\EE^\times)_{s,t}^{\chi,\eta}$.
	
	For $a\in\FF^\times$ and $z\in\CC$ we define the operator $D_z^a$ on $\calD'(\EE)$ by 
	$$ D_z^au(x) = u(ax)-|a|_\FF^{4z-2}u(x). $$
	Then Proposition~\ref{prop:CharSBOKernels} and Theorem~\ref{thm:HolomorphicRenormalization} imply $D_z^au_z=0$ for all $a\in\FF^\times$ and all $z\in\CC$ for which $u_z$ is defined. Since $|a|_\FF^z=1+\ln|a|_\FF\cdot z+\calO(z^2)$ we can write
	$$ D_z^a = D_0^a+zD_1^a+\calO(z^2) $$
	with
	$$ D_0^au(x) = u(ax)-|a|_\FF^{-2}u(x) \qquad \mbox{and} \qquad D_1^au(x) = -4(\ln|a|_\FF)|a|_\FF^{-2}u(x). $$
	Note that $D_1^a$ is a multiple of the identity, so it commutes with $D_0^a$. Comparing coefficients of $z^{-1}$ and $z^0$ in $D_z^au_z=0$ yields
	\begin{equation}
		D_0^au_{-1} = 0 \;\; \mbox{and} \;\; D_0^au_0 + D_1^au_{-1} = 0.\label{eq:ClassProofIdentitiesForLaurentCoeff}
	\end{equation}
	
	If now $v\in\calD'(\EE)_{s,t}^{\chi,\eta}$, then by Proposition~\ref{prop:CharacterizationSBOKernelsPGL2} and the observation above $v|_{\EE^\times}=c\cdot u_0|_{\EE^\times}$ for some $c\in\CC$. This implies $(v-cu_0)|_{\EE^\times}=0$, so $\supp(v-cu_0)\subseteq\{0\}$ and hence $v-cu_0$ is a multiple of $\delta$. But $\delta(ax)=|a|_\FF^{-2}\delta(x)$ for all $a\in\FF^\times$, whence
	$$ D_0^a(v-cu_0) = 0. $$
	On the other hand $v\in\calD'(\EE)_{s,t}^{\chi,\eta}$ implies $D_0^av=0$ by Proposition~\ref{prop:CharSBOKernels}, so by \eqref{eq:ClassProofIdentitiesForLaurentCoeff}:
	$$ 0 = c\cdot D_0^au_0 = -c\cdot D_1^au_{-1} = 4c\cdot(\ln|a|_\FF)|a|_\FF^{-2}u_{-1} \qquad \mbox{for all }a\in\FF^\times. $$
	But since $u_{-1}\neq0$ this gives $c=0$, so $v|_{\EE^\times}=0$.
\end{proof}

By the previous statement combined with the exact sequence \eqref{eq:ExactSequence} and Lemma~\ref{lem:DimensionEstimates}, we also find $\dim\calD'(\EE)_{s,t}^{\chi,\eta}=1$ for $(\chi,s,\eta,t)\in\SminusT\setminus L$, and hence, by Theorem~\ref{thm:ZerosAndSupportOfKernel}, $\calD'(\EE)_{s,t}^{\chi,\eta}=\CC\widetilde{K}_{s,t}^{\chi,\eta}$.

It remains to study the case $(\chi,s,\eta,t)\in L$, and we show that in this case we actually have $\dim\calD'(\EE)_{s,t}^{\chi,\eta}=2$. One symmetry breaking kernel in $\calD'(\EE)_{s,t}^{\chi,\eta}$ is $\delta$ (see Lemma~\ref{lem:DimensionEstimates}). We obtain a linearly independent kernel by regularizing $(s,t)\mapsto K_{s,t}^{\chi,\eta}$ along a hyperplane in $\CC^2$.

\begin{proposition}
\label{prop::kernel_tilde_tilde}
	Let $(\chi,s,\eta,t)\in L$. Then the function
	$$ \CC\to\calD'(\EE), \quad z\mapsto (1-q_\FF^{-2z})^{-1}\widetilde{K}_{s+z,t}^{\chi,\eta} = (1-q_\FF^{-(2s+2z+t+\frac{1}{2})})K_{s+z,t}^{\chi,\eta} $$
	is weakly holomorphic at $z=0$ and its value at $z=0$ defines a distribution $\doublewidetilde{K}_{s,t}^{\chi,\eta}$ in $\calD'(\EE)_{s,t}^{\chi,\eta}$ with support equal to $\EE$.	
\end{proposition}

\begin{proof}
	By Theorem~\ref{thm:HolomorphicRenormalization} the function
	$$ f:\CC\to\calD'(\EE)_{s,t}^{\chi,\eta}, \quad f(z)=\widetilde{K}_{s+z,t}^{\chi,\eta} $$
	is weakly holomorphic, and by Theorem~\ref{thm:ZerosAndSupportOfKernel} it satisfies $f(0)=0$. Therefore,
	$$ g(z)=(1-q_\FF^{-2z})^{-1}f(z) $$
	is holomorphic at $z=0$. Since $f(z)\in\calD'(\EE)_{s+z,t}^{\chi,\eta}$ for all $z\in\CC$ and the conditions in Proposition~\ref{prop:CharSBOKernels} depend continuously on $s$, it follows that $g(z)$ is contained in $\calD'(\EE)_{s+z,t}^{\chi,\eta}$ for all $z\in\CC$, in particular $g(0)\in\calD'(\EE)_{s,t}^{\chi,\eta}$. To show that the distribution $g(0)$ has support equal to $\EE$, first note that $(\chi,s,\eta,t)\in\SminusT\setminus\SplusT$, so $1-q_\FF^{-(2s-t+\frac{1}{2})}=0$ while $1-q_\FF^{-(2s+t+\frac{1}{2})}\neq0$. It follows that
	$$ g(z) = (1-q_\FF^{-(2s+2z+t+\frac{1}{2})})K_{s+z,t}^{\chi,\eta} $$
	with $1-q_\FF^{-(2s+2z+t+\frac{1}{2})}\neq0$ for $z=0$. Therefore, the restriction of $g(0)$ to $\EE\setminus\FF$ is a non-trivial scalar multiple of the nowhere vanishing function $K_{s,t}^{\chi,\eta}|_{\EE\setminus\FF}$. In particular, $\EE\setminus\FF\subseteq\supp g(0)$ and hence $\supp g(0)=\EE$.
\end{proof}

Clearly $\delta$ and $\doublewidetilde{K}_{s,t}^{\chi,\eta}$ are linearly independent since they have different supports. This completes the proof of the following classification result:

\begin{corollary}\label{cor::dim_of D_E_chi_eta}
$$ \calD'(\EE)_{s,t}^{\chi,\eta} = \begin{cases}\CC\widetilde{K}_{s,t}^{\chi,\eta}&\mbox{for $(\chi,s,\eta,t)\notin L$,}\\\CC\delta\oplus\CC\doublewidetilde{K}_{s,t}^{\chi,\eta}&\mbox{for $(\chi,s,\eta,t)\in L$.}\end{cases} $$
\end{corollary}

\section{Mapping properties of symmetry breaking operators}
\label{sec::mapp_prop}

In this section we study some mapping properties of symmetry breaking operators between principal series of $\PGL(2)$. For this, we denote by $\widetilde{A}_{s,t}^{\chi,\eta}:\pi_{\chi,s}\to\tau_{\eta,t}$ the intertwining operator with distribution kernel $\widetilde{K}_{s,t}^{\chi,\eta}$, i.e.
\begin{equation}
	\widetilde{A}_{s,t}^{\chi,\eta}f(\overline{n}_y) = \int_\EE\widetilde{K}_{s,t}^{\chi,\eta}(x)f(\overline{n}_{x+y})\,dx \qquad (f\in\Ind_B^G(\chi_s),y\in\FF),\label{eq:DefinitionAtilde}
\end{equation}
where the integral is to be understood in the sense of distributions. Note that $\widetilde{A}_{s,t}^{\chi,\eta}=0$ for $(\chi,s,\eta,t)\in L$. In the same way, we denote by $\doublewidetilde{A}_{s,t}^{\chi,\eta}$ the symmetry breaking operator with kernel $\doublewidetilde{K}_{s,t}^{\chi,\eta}$ for $(\chi,s,\eta,t)\in L$, and by $C_{s,t}^{\chi,\eta}$ the operator with kernel $\delta(x)$ in the case $(\chi,s,\eta,t)\in\SminusT$, i.e. the restriction map
\begin{equation}
	C_{s,t}^{\chi,\eta}f(\overline{n}_y) = f(\overline{n}_y) \qquad (f\in\Ind_B^G(\chi_s),y\in\FF).\label{eq:DefinitionC}
\end{equation}

In the first part we determine the action of symmetry breaking operators on vectors which transform by a character under the action of a maximal compact subgroup, and in the second part we find the images of all symmetry breaking operators in the case where $\chi|_{\calO_\FF^\times}\cdot\eta=\chi|_{\calO_\FF^\times}\cdot\eta^{-1}=1$.

\subsection{Evaluation on one-dimensional types}

The restriction of the representation $\pi_{\chi,s}$ of $G=\PGL(2,\EE)$ to the maximal compact subgroup $K=\PGL(2,\calO_\EE)$ contains a one-dimensional subrepresentation of $K$ if and only if $\chi^2=1$, i.e. $\chi$ is real-valued (see e.g. \cite[Lemma 2.2, Theorem 3.3, \S3.2]{Sil70} for the case of $p$-adic fields). This one-dimensional subrepresentation is unique and given by the character $\chi\circ\det:K\to\CC^\times$. (Note that since $\chi^2=1$, the composition $\chi\circ\det$ is well-defined on $\PGL(2,\calO_\EE)$.) It contains a unique vector $\varphi_{\chi,s}$ with $\varphi_{\chi,s}(1)=1$ which is given by
$$ \varphi_{\chi,s}(kb) = \chi_{s+\frac{1}{2}}(b)^{-1}\chi(\det k)^{-1} \qquad (k\in K,b\in B). $$

\begin{lemma}
For $x\in\EE$ we have
$$ \varphi_{\chi,s}(\overline{n}_x) = \max(1,|x|_\EE)^{-(2s+1)}. $$
\end{lemma}

\begin{proof}
For $|x|_\EE\leq1$, the matrix $\overline{n}_x$ is contained in $K$, so we can choose $b=1$ in the decomposition into $KB$ to find
$$ \varphi_{\chi,s}(\overline{n}_x) = 1. $$
For $|x|_\EE>1$ we write
$$ \overline{n}_x = \begin{pmatrix}1&0\\x&1\end{pmatrix} = \begin{pmatrix}\varpi^{-v(x)}_{\EE}&-\varpi^{v(x)}_{\EE}/x\\\varpi^{-v(x)}_{\EE}x&0\end{pmatrix}\begin{pmatrix}\varpi^{v(x)}_{\EE}&\varpi^{v(x)}_{\EE}/x\\0&\varpi^{-v(x)}_{\EE}\end{pmatrix}\in KB, $$
so
\begin{equation*}
	\varphi_{\chi,s}(\overline{n}_x)=|\varpi^{2v(x)}_{\EE}|_{\EE}^{-(s+\frac{1}{2})}=|x|_\EE^{-(2s+1)}.\qedhere
\end{equation*}
\end{proof}

Since $\widetilde{A}_{s,t}^{\chi,\eta}$ is $G'$-intertwining and the vector $\varphi_{\chi,s}$ transforms by the character $\chi\circ\det$ under the action of $K$, its image $\widetilde{A}_{s,t}^{\chi,\eta}\varphi_{\chi,s}$ transforms by the character $\chi|_{\calO_\FF^\times}\circ\det$ under the action of $K'=\PGL(2,\calO_\FF)$. But $\tau_{\eta,t}|_{K'}$ contains a one-dimensional subrepresentation given by the character $\chi|_{\calO_\FF^\times}\circ\det$ if and only if $\eta=\chi|_{\calO_\FF^\times}$, so we find
$$ \widetilde{A}_{s,t}^{\chi,\eta}\varphi_{\chi,s} = 0 \qquad \mbox{if }\chi|_{\calO_\FF^\times}\neq\eta. $$
Note that under the assumption $\chi^2=1$, we have the following equivalences:
$$ \chi|_{\calO_\FF^\times}=\eta \quad \Leftrightarrow \quad \chi|_{\calO_\FF^\times}\cdot\eta=1 \quad \Leftrightarrow \quad \chi|_{\calO_\FF^\times}\cdot\eta^{-1}=1. $$

In what follows, assume that these conditions are satisfied, then $\widetilde{A}_{s,t}^{\chi,\eta}\varphi_{\chi,s}$ is a constant multiple of $\psi_{\eta,t}$, the unique vector in $\tau_{\eta,t}$ which transforms by the character $\eta\circ\det$ of $K'$ and is normalized by $\psi_{\eta,t}(1)=1$. To determine the constant, we evaluate \eqref{eq:DefinitionAtilde} at $y=0$ and use $\psi_ {\eta,t}(1)=1$, so
$$ \widetilde{A}_{s,t}^{\chi,\eta}\varphi_{\chi,s} = \widetilde{A}_{s,t}^{\chi,\eta}\varphi_{\chi,s}(1)\cdot\psi_{\eta,t}, $$
where
$$ \widetilde{A}_{s,t}^{\chi,\eta}\varphi_{\chi,s}(1) = (1-q_\FF^{-(2s-t+\frac{1}{2})})(1-q_\FF^{-(2s+t+\frac{1}{2})})\int_\EE|x|_\EE^{-t-\frac{1}{2}}|x_2|_\FF^{2s+t-\frac{1}{2}}\max(1,|x|_\EE)^{-(2s+1)}\,dx. $$
The integral is computed in the following lemma:

\begin{lemma}
For $|\alpha^2|_\FF=1$:
\begin{equation*}
	\int_\EE|x|_\EE^{-t-\frac{1}{2}}|x_2|_\FF^{2s+t-\frac{1}{2}}\max(1,|x|_\EE)^{-(2s+1)}\,dx =\\
	\frac{(1-q_\FF^{-1})(1-q_\FF^{-4s-2})}{(1-q_\FF^{-(2s-t+\frac{1}{2})})(1-q_\FF^{-(2s+t+\frac{1}{2})})},
\end{equation*}
and for $|\alpha^2|_\FF=q_\FF^{-1}$:
\begin{equation*}
	\int_\EE|x|_\EE^{-t-\frac{1}{2}}|x_2|_\FF^{2s+t-\frac{1}{2}}\max(1,|x|_\EE)^{-(2s+1)}\,dx =\\
	\frac{(1-q_\FF^{-1})(1+q_\FF^{t-\frac{1}{2}})(1-q_\FF^{-(2s+1)})}{(1-q_\FF^{-(2s-t+\frac{1}{2})})(1-q_\FF^{-(2s+t+\frac{1}{2})})}.
\end{equation*}
\end{lemma}

\begin{proof}
We write the integral over $x\in\EE$ as a series of integrals over $\varpi_\FF^m\calO_\FF^\times\times\varpi_\FF^n\calO_\FF^\times$. Using that $\operatorname{vol}(\varpi_\FF^m\calO_\FF^\times)=(1-q_\FF^{-1})q_\FF^{-m}$ we find
$$ = (1-q_\FF^{-1})^2\sum_{m\in\ZZ}\sum_{n\in\ZZ}q_\FF^{-m-n}|(\varpi_{\FF}^m)^2-(\varpi_{\FF}^n)^2\alpha^2|_\FF^{-t-\frac{1}{2}}|\varpi_\FF^n|_\FF^{2s+t-\frac{1}{2}}\max(1,|(\varpi_\FF^m)^2-(\varpi_\FF^n)^2\alpha^2|_\FF)^{-(2s+1)}. $$
Recall that $|(\varpi_{\FF}^m)^2-(\varpi_{\FF}^n)^2\alpha^2|_\FF$ equals $|(\varpi_{\FF}^m)^2|_\FF$ if $n \geq m$ and $|(\varpi_{\FF}^n)^2\alpha^2|_\FF$ if $m >n$, which is true for both cases $|\alpha^2|_\FF=1$ or $|\alpha^2|_\FF=q_\FF^{-1}$. Then by writing $\sum\limits_{m\in\ZZ}\sum\limits_{n\in\ZZ}= \sum\limits_{m\in\ZZ}(\sum\limits_{n<m}+ \sum\limits_{m \leq n}) $, we have
\begin{align*}
	={}& (1-q_\FF^{-1})^2 \sum_{m\in\ZZ}\sum_{m\leq n}q_\FF^{-m-n}|\varpi_{\FF}^{2m}|_\FF^{-t-\frac{1}{2}}|\varpi_\FF^n|_\FF^{2s+t-\frac{1}{2}}\max(1,|\varpi_\FF^{2m}|_\FF)^{-(2s+1)}\\
	&+ (1-q_\FF^{-1})^2 \sum_{m\in\ZZ}\sum_{m>n}q_\FF^{-m-n}|\varpi_{\FF}^{2n}\alpha^2|_\FF^{-t-\frac{1}{2}}|\varpi_\FF^n|_\FF^{2s+t-\frac{1}{2}}\max(1,|\varpi_\FF^{2n}\alpha^2|_\FF)^{-(2s+1)}\\
	={}& (1-q_\FF^{-1})^2 \sum_{m\in\ZZ}\sum_{m\leq n}q_\FF^{2mt}q_\FF^{-n(2s+t+\frac{1}{2})}\max(1,q_\FF^{-2m})^{-(2s+1)}\\
	&+ (1-q_\FF^{-1})^2|\alpha^2|_\FF^{-t-\frac{1}{2}}  \sum_{m\in\ZZ}\sum_{m>n}q_\FF^{-m}q_\FF^{-n(2s-t-\frac{1}{2})}\max(1,q_\FF^{-2n}|\alpha^2|_\FF )^{-(2s+1)}.
\end{align*}
In the first sum we write $n=m+k$ and in the second sum we write $m=n+k+1$:
\begin{multline*}
	(1-q_\FF^{-1})^2\sum_{m\in\ZZ}q_\FF^{-m(2s-t+\frac{1}{2})}\max(1,q_\FF^{-2m})^{-(2s+1)}\sum_{k=0}^\infty q_\FF^{-k(2s+t+\frac{1}{2})}\\
	+ q_\FF^{-1}(1-q_\FF^{-1})^2|\alpha^2|_\FF^{-t-\frac{1}{2}}\sum_{n\in\ZZ}q_\FF^{-n(2s-t+\frac{1}{2})}\max(1,q_\FF^{-2n}|\alpha^2|_\FF)^{-(2s+1)}\sum_{k=0}^\infty q_\FF^{-k}.
\end{multline*}
Finally, the sums over $k$ are geometric series, and splitting the sums according to $\max(1,q_\FF^{-2n}|\alpha^2|_\FF)$  and $\max(1,q_\FF^{-2m})$ we get:
\begin{multline*}
	\frac{1-q_\FF^{-1}}{1-q_\FF^{-(2s-t+\frac{1}{2})}} \left(\frac{1-q_\FF^{-1}}{1-q_\FF^{-(2s+t+\frac{1}{2})}} + |\alpha^2|_\FF^{-t-\frac{1}{2}}q_\FF^{-1}  \right)\\
	+ q_\FF^{-(2s+t+\frac{3}{2})}\frac{1-q_\FF^{-1}}{1-q_\FF^{-(2s+t+\frac{3}{2})}} \left(\frac{1-q_\FF^{-1}}{1-q_\FF^{-(2s+t+\frac{1}{2})}} + |\alpha^2|_\FF^{-2s -t-\frac{3}{2}}q_\FF^{-1}  \right).
\end{multline*}
Considering the cases $|\alpha^2|_\FF=1$ and $|\alpha^2|_\FF=q_\FF^{-1}$ separately and simplifying the resulting expression shows the claimed formulas.
\end{proof}

\begin{corollary}
\label{cor::sph_vectors}
Assume $\chi^2=1$, then $\widetilde{A}_{s,t}^{\chi,\eta}\varphi_{\chi,s}=0$ unless $\eta=\chi|_{\calO_\FF^\times}$ in which case we have
$$ \widetilde{A}_{s,t}^{\chi,\eta}\varphi_{\chi,s} = (1-q_\FF^{-1})\psi_{\eta,t}\times\begin{cases}(1-q_\EE^{-(2s+1)})&\mbox{for $\EE/\FF$ unramified,}\\(1-q_\EE^{-(2s+1)})(1+q_\FF^{t-\frac{1}{2}})&\mbox{for $\EE/\FF$ ramified.}\end{cases} $$
\end{corollary}

Note that if $\EE/\FF$ is unramified we have $\calO_\FF^{\times}\subseteq\calO_\EE^{\times2}$, so the condition $\chi^2=1$ implies $\chi|_{\calO_\FF^\times}=1$.

\begin{remark}\label{rem:ActionOnSphVectorForCandDoubleTildeA}
Clearly, $C_{s,t}^{\chi,\eta}\varphi_{\chi,s}=\psi_{\eta,t}$ for $(\chi,s,\eta,t)\in\SminusT$ and $\chi^2=1$. In the case $(\chi,s,\eta,t)\in L$ one can also determine the action of the operator $\doublewidetilde{A}_{s,t}^{\chi,\eta}$ by taking residues in Corollary~\ref{cor::sph_vectors}.
\end{remark}

\subsection{Images of symmetry breaking operators}

In this section we determine the image of $\widetilde{A}_{s,t}^{\chi,\eta} \in\Hom_{G'}(\pi_{\chi,s}|_{G'},\tau_{\eta,t})$ for all $s,t\in\CC$ in the case where $\chi|_{\calO_\FF^\times}\cdot\eta=\chi|_{\calO_\FF^\times}\cdot\eta^{-1}=1$. Recall from Section \ref{subsec::holom_renor} that in this case
$$ \widetilde{K}_{s,t}^{\chi,\eta}(x) = (1-q_\FF^{2s+t+\frac{1}{2}})(1-q_\FF^{2s-t+\frac{1}{2}})\cdot\chi_0(x^2)|x|_\EE^{-t-\frac{1}{2}}|x_2|_{\FF}^{2s+t-\frac{1}{2}}. $$

We recall from Theorem~\ref{thm:CompositionSeries} that the representation $\tau_{\eta,t}$ is irreducible unless $t\in\pm\frac{1}{2}+\frac{i\pi}{\ln q_\FF}\ZZ$. For $t\in-\frac{1}{2}+\frac{i\pi}{\ln q_\FF}\ZZ$ it has a one-dimensional subrepresentation spanned by the function $\psi_{\eta,t}$ defined in the previous section. Note that $\psi_{\eta,t}(\overline{n}_y)=1$ for all $y\in\FF$. As explained in Section~\ref{sec:PSforPGL2}, the irreducible subrepresentation $\tau_{\eta,t}^0$ of $\tau_{\eta,t}$ for $t\in\frac{1}{2}+\frac{i\pi}{\ln q_\FF}\ZZ$ is the kernel of the intertwining operator
$$ T_{\eta,t}:\tau_{\eta,t}\to\tau_{\eta,-t}, \quad T_{\eta,t}f(\overline{n}_y) = \int_\FF f(\overline{n}_z)\,dz, $$
which maps onto the functions which are constant on $\overline{N}'$, i.e. onto $\CC\psi_{\eta,-t}$.

\begin{theorem}
\label{thm::long_image}
	Assume that $\chi|_{\calO_\FF^\times}\cdot\eta=\chi|_{\calO_\FF^\times}\cdot\eta^{-1}=1$, and let $s,t \in \CC$. The image of $\widetilde{A}_{s,t}^{\chi,\eta}$ is given as follows:
	\begin{enumerate}
		\item For $\chi^2\neq1$ and $\EE/\FF$ unramified:
		$$ \Im(\widetilde{A}_{s,t}^{\chi,\eta}) = \begin{cases}\{0\}&\mbox{if }(s,t)\in\{(0,\frac{1}{2}),(\frac{\pi i}{\ln q_\EE},\frac{1}{2}+\frac{\pi i}{\ln q_\FF})\}=L^{\chi,\eta},\\\CC\psi_{\eta,t}&\mbox{if }(s,t)\in \{(0,-\frac{1}{2}),(\frac{\pi i}{\ln q_\EE},-\frac{1}{2}+\frac{\pi i}{\ln q_\FF})\},\\\tau_{\eta,t}^0&\mbox{if }(s,t)\in\left[\CC\times(\frac{1}{2}+\frac{\pi i}{\ln q_\FF}\ZZ)\right]\setminus L^{\chi,\eta},\\\tau_{\eta,t}&\mbox{else}.\end{cases} $$
		\item For $\chi^2\neq1$ and $\EE/\FF$ ramified:
		$$ \Im(\widetilde{A}_{s,t}^{\chi,\eta}) = \begin{cases}\{0\}&\begin{array}{l}\mbox{if }(s,t)\in\{(0,\frac{1}{2}),(\frac{\pi i}{\ln q_\EE},\frac{1}{2}),\\\hspace{2.5cm}(\frac{\pi i}{2\ln q_\EE},\frac{1}{2}+\frac{\pi i}{\ln q_\FF}),(\frac{3\pi i}{2\ln q_\EE},\frac{1}{2}+\frac{\pi i}{\ln q_\FF})\}=L^{\chi,\eta},\end{array}\\\CC\psi_{\eta,t}&\begin{array}{l}\mbox{if }(s,t)\in\{ (0,-\frac{1}{2}),(\frac{\pi i}{\ln q_\EE},-\frac{1}{2}),\\\hspace{2.5cm}(\frac{\pi i}{2\ln q_\EE},-\frac{1}{2}+\frac{\pi i}{\ln q_\FF}),(\frac{3\pi i}{2\ln q_\EE},-\frac{1}{2}+\frac{\pi i}{\ln q_\FF})\},\end{array}\\\tau_{\eta,t}^0&\mbox{if }(s,t)\in\left[\CC\times(\frac{1}{2}+\frac{\pi i}{\ln q_\FF}\ZZ)\right]\setminus L^{\chi,\eta},\\\tau_{\eta,t}&\mbox{else}.\end{cases} $$		
		\item For $\chi^2=1$ and $\EE/\FF$ unramified:
		$$ \Im(\widetilde{A}_{s,t}^{\chi,\eta}) = \begin{cases}\{0\}&\mbox{if }(s,t)\in\{(-\frac{1}{2},-\frac{1}{2}),(-\frac{1}{2}+\frac{\pi i}{\ln q_\EE},-\frac{1}{2}+\frac{\pi i}{\ln q_\FF})\}=L^{\chi,\eta},\\\CC\psi_{\eta,t}&\mbox{if }(s,t)\in\left[\CC\times(-\frac{1}{2}+\frac{\pi i}{\ln q_\FF}\ZZ)\right]\setminus L^{\chi,\eta},\\\tau_{\eta,t}^0&\mbox{if }(s,t)\in\{(-\frac{1}{2},\frac{1}{2}),(-\frac{1}{2}+\frac{\pi i}{\ln q_\EE},\frac{1}{2}+\frac{\pi i}{\ln q_\FF})\},\\\tau_{\eta,t}&\mbox{else}.\end{cases} $$
		\item For $\chi^2=1$ and $\EE/\FF$ ramified:
		$$ \Im(\widetilde{A}_{s,t}^{\chi,\eta}) = \begin{cases}\{0\}&\begin{array}{l}\mbox{if }(s,t)\in\{(-\frac{1}{2},-\frac{1}{2}),(-\frac{1}{2}+\frac{\pi i}{\ln q_\EE},-\frac{1}{2}),\\\hspace{2.5cm}(\frac{\pi i}{2\ln q_\EE},\frac{1}{2}+\frac{\pi i}{\ln q_\FF}),(\frac{3\pi i}{2\ln q_\EE},\frac{1}{2}+\frac{\pi i}{\ln q_\FF})\}=L^{\chi,\eta},\end{array}\\\CC\psi_{\eta,t}&\begin{array}{l}\mbox{if }(s,t)\in\left[(\CC\setminus(-\frac{1}{2}+\frac{\pi i}{\ln q_\EE}\ZZ))\times\{-\frac{1}{2}\}\right]\\\hspace{2.5cm}\cup\,\{(\frac{\pi i}{2\ln q_\EE},-\frac{1}{2}+\frac{\pi i}{\ln q_\FF}),(\frac{3\pi i}{2\ln q_\EE},-\frac{1}{2}+\frac{\pi i}{\ln q_\FF})\},\end{array}\\\tau_{\eta,t}^0&\begin{array}{l}\mbox{if }(s,t)\in\left[(\CC\times\{\frac{1}{2}+\frac{\pi i}{\ln q_\FF}\})\setminus\{(\frac{\pi i}{2\ln q_\EE},\frac{1}{2}+\frac{\pi i}{\ln q_\FF}),(\frac{3\pi i}{2\ln q_\EE},\frac{1}{2}+\frac{\pi i}{\ln q_\FF})\}\right]\\\hspace{2.5cm}\cup\,\{(-\frac{1}{2},\frac{1}{2}),(-\frac{1}{2}+\frac{\pi i}{\ln q_\EE},\frac{1}{2})\},\end{array}\\\tau_{\eta,t}&\begin{array}{l}\mbox{else.}\end{array}\end{cases} $$
	\end{enumerate}
\end{theorem}

The proof follows from Theorem~\ref{thm:ZerosAndSupportOfKernel} together with the following two lemmas treating the cases $t\in\frac{1}{2}+\frac{i\pi}{\ln q_\FF}\ZZ$ and $t\in-\frac{1}{2}+\frac{i\pi}{\ln q_\FF}\ZZ$ where $\tau_{\eta,t}$ is reducible. 

 We start with the case $t\in\frac{1}{2}+\frac{i\pi}{\ln q_\FF}\ZZ$. The image of $\widetilde{A}_{s,t}^{\chi,\eta}$ is contained in the subrepresentation $\tau_{\eta,t}^0$ if and only if $T_{\eta,t}\circ\widetilde{A}_{s,t}^{\chi,\eta}=0$. This composition can be expressed in terms of $\widetilde{A}_{s,-t}^{\chi,\eta}$:

\begin{lemma}\label{lem:CompositionWithStdIntertwiner}
	Assume that $\chi|_{\calO_\FF^\times}\cdot\eta=\chi|_{\calO_\FF^\times}\cdot\eta^{-1}=1$ and $t\in\frac{1}{2}+\frac{\pi i}{\ln q_\FF}\ZZ$. Then
	$$ T_{\eta,t}\circ\widetilde{A}_{s,t}^{\chi,\eta} = c_\FF\Gamma(|\cdot|_\FF^{2t})\Gamma((\chi^2)_{-t+\frac{1}{2}}) \cdot \widetilde{A}_{s,-t}^{\chi,\eta} $$
	as an identity of holomorphic functions in $s\in\CC$. In particular, $\Im(\widetilde{A}_{s,t}^{\chi,\eta})\subseteq\tau_{\eta,t}^0$ if and only if
	$$ \begin{cases}\mbox{always}&\mbox{if $\chi^2\neq1$,}\\(s,t)\in\{(-\frac{1}{2},\frac{1}{2}),(-\frac{1}{2}+\frac{i\pi}{\ln q_\EE},\frac{1}{2}+\frac{i\pi}{\ln q_\FF})\}&\mbox{if $\chi^2=1$ and $\EE/\FF$ is unramified,}\\t=\frac{1}{2}+\frac{\pi i}{\ln q_\FF}\mbox{ or }(s,t)\in\{(-\frac{1}{2},\frac{1}{2}),(-\frac{1}{2}+\frac{\pi i}{\ln q_\EE},\frac{1}{2})\}&\mbox{if $\chi^2=1$ and $\EE/\FF$ is ramified.}\end{cases} $$
\end{lemma}

\begin{proof}
	The composition of $\widetilde{A}_{s,t}^{\chi,\eta}$ with $T_{\eta,t}$ is given by
	$$ T_{\eta,t}\circ\widetilde{A}_{s,t}^{\chi,\eta}f(\overline{n}_z) = \int_\FF\int_\EE\widetilde{K}_{s,t}^{\chi,\eta}(x)f(\overline{n}_{x+y})\,dx\,dy = \int_\EE\left(\int_\FF\widetilde{K}_{s,t}^{\chi,\eta}(x+y)\,dy\right)f(\overline{n}_x)\,dx. $$
	On the other hand, $T_{\eta,t}\circ\widetilde{A}_{s,t}^{\chi,\eta}\in\Hom_{G'}(\pi_{\chi,s}|_H,\tau_{\eta,-t})$, so by Corollary~\ref{cor::dim_of D_E_chi_eta} $T_{\eta,t}\circ\widetilde{A}_{s,t}^{\chi,\eta}=c(\chi,s,\eta,t)\widetilde{A}_{s,-t}^{\chi,\eta}$ for some constant $c(\chi,s,\eta,t)$. (This argument works for generic $s\in\CC$ such that $(s,t)\not\in L^{\chi,\eta}$, but since $c(\chi,s,\eta,t)$ is meromorphic in $s$, the formula $T_{\eta,t}\circ\widetilde{A}_{s,t}^{\chi,\eta}=c(\chi,s,\eta,t)\widetilde{A}_{s,-t}^{\chi,\eta}$ is true for all $s$.) In terms of integral kernels, this means
	$$ \int_\FF\widetilde{K}_{s,t}^{\chi,\eta}(x+y)\,dy = c(\chi,s,\eta,t)\widetilde{K}_{s,-t}^{\chi,\eta}(x). $$
	Putting $x=\alpha$ this reads (dropping the normalization factors):
	$$ \int_\FF\chi_0(y+\alpha)^2|y+\alpha|_\EE^{-t-\frac{1}{2}}\,dy = c(\chi,s,\eta,t)\cdot\chi_0(\alpha)^2|\alpha|_\EE^{t-\frac{1}{2}}. $$
	The remaining integral can be evaluated with Proposition~\ref{prop:GammaIntegral}. This shows the claimed formula. For the last statement, we combine Theorem~\ref{thm:ZerosAndSupportOfKernel} and Lemma~\ref{lem:PropertiesGammaFactors}.
\end{proof}

Now let us consider the case $t\in-\frac{1}{2}+\frac{i\pi}{\ln q_\FF}\ZZ$. The image of $\widetilde{A}_{s,t}^{\chi,\eta}$ is contained in the one-dimensional subrepresentation spanned by the function $\psi_{\eta,t}$ if and only if the kernel $\widetilde{K}_{s,t}^{\chi,\eta}(x)$ is independent of $x_1$, where $x=x_1+x_2\alpha$.

\begin{lemma}
	Assume that $\chi|_{\calO_\FF^\times}\cdot\eta=\chi|_{\calO_\FF^\times}\cdot\eta^{-1}=1$ and $t\in-\frac{1}{2}+\frac{\pi i}{\ln q_\FF}\ZZ$. Then the kernel $\widetilde{K}_{s,t}^{\chi,\eta}(x)$ is independent of $x_1$ if and only if $(s,t)\in L^{\chi,\eta}$ (in which case $\widetilde{K}_{s,t}^{\chi,\eta}=0$), or $\chi^2=1$ and $t\in-\frac{1}{2}+\frac{2\pi i}{\ln q_\EE}\ZZ$, or
	$$ (s,t)\in\begin{cases}\{(0,-\frac{1}{2}),(\frac{\pi i}{\ln q_\EE},-\frac{1}{2}+\frac{\pi i}{\ln q_\FF})\}&\mbox{if $\EE/\FF$ is unramified,}\\\{(0,-\frac{1}{2}),(\frac{\pi i}{\ln q_\EE},-\frac{1}{2}),(\frac{\pi i}{2\ln q_\EE},-\frac{1}{2}+\frac{\pi i}{\ln q_\FF}),(\frac{3\pi i}{2\ln q_\EE},-\frac{1}{2}+\frac{\pi i}{\ln q_\FF})\}&\mbox{if $\EE/\FF$ is ramified.}\end{cases} $$
\end{lemma}

\begin{proof}
	Let us first assume that $2s\pm t+\frac{1}{2}\not\in\frac{2\pi i}{\ln q_\FF}\ZZ$, then $\supp\widetilde{K}_{s,t}^{\chi,\eta}=\EE$ by Theorem~\ref{thm:ZerosAndSupportOfKernel}. On $\EE\setminus\FF$ the kernel is (up to the non-zero normalization factor) given by
	$$ K_{s,t}^{\chi,\eta}(x_1+x_2\alpha) = \chi_0(x^2)|x|_\EE^{-t-\frac{1}{2}}|x_2|_\FF^{2s+t-\frac{1}{2}} $$
	which is independent of $x_1$ if and only if it is equal to the same expression for $x_1=0$:
	$$ (\chi^2)_{-t-\frac{1}{2}}(x) = \chi_0(x^2)|x|_\EE^{-t-\frac{1}{2}} = \chi_0(x_2^2\alpha^2)|x_2\alpha|_\EE^{-t-\frac{1}{2}} \qquad \mbox{for all }x=x_1+x_2\alpha, x_2\neq0. $$
	But since $\chi^2|_{\calO_\FF^\times}=1$ and $|x_2\alpha|_\EE^{-t-\frac{1}{2}}=|x_2^2\alpha^2|_\FF^{-t-\frac{1}{2}}=|x_2|_\FF^{-2t-1}|\alpha^2|_\FF^{-t-\frac{1}{2}}$, the right hand side equals $\chi_0(\alpha^2)|\alpha^2|_\FF^{-t-\frac{1}{2}}\in\{\pm1\}$ for all $x_2\in\FF^\times$. This implies that the above condition is equivalent to
	$$ (\chi^2)_{-t-\frac{1}{2}}(x) = \chi_0(\alpha^2)|\alpha^2|_\FF^{-t-\frac{1}{2}} $$
	for all $x\in\EE\setminus\FF$, and therefore, by density, also for all $x\in\EE^\times$. This implies that the character $(\chi^2)_{-t-\frac{1}{2}}$ of $\EE^\times$ is constant, in particular the right hand side has to be $1$. We conclude that $\chi^2=1$ and $-t-\frac{1}{2}\in\frac{2\pi i}{\ln q_\EE}\ZZ$.\\
	Now, if $2s+t+\frac{1}{2}\in\frac{2\pi i}{\ln q_\FF}\ZZ$, then $\widetilde{K}_{s,t}^{\chi,\eta}$ is a non-zero multiple of $L_t^\eta$ by Theorem~\ref{thm:ResidueIdentities}. Since $t\in-\frac{1}{2}+\frac{\pi i}{\ln q_\FF}\ZZ$ we observe that $L_t^\eta(x)=(1-q_\FF^{-1})\delta(x_2)$ is independent of $x_1$.\\
	Finally, if $2s-t+\frac{1}{2}\in\frac{2\pi i}{\ln q_\FF}\ZZ$, then $\widetilde{K}_{s,t}^{\chi,\eta}(x)$ is a multiple of $\delta(x)$ which is not independent of $x_1$. It follows that, in this case, $\widetilde{K}_{s,t}^{\chi,\eta}$ is independent of $x_1$ if and only if it is zero, i.e. $(s,t)\in L^{\chi,\eta}$.
\end{proof}

For $(\chi,s,\eta,t)\in L$, we also determine the image of the two operators $C_{s,t}^{\chi,\eta}$ and $\doublewidetilde{A}_{s,t}^{\chi,\eta}$. We start with $C_{s,t}^{\chi,\eta}$ which is for all $(\chi,s,\eta,t)\in\SminusT$ given by restriction from $\overline{N}$ to $\overline{N}'$ (see \eqref{eq:DefinitionC}).

\begin{proposition}
\label{prop::im_C}
	For all $(\chi,s,\eta,t)\in\SminusT$ we have
	$$ \Im(C_{s,t}^{\chi,\eta}) = \tau_{\eta,t}. $$
\end{proposition}

\begin{proof}
	The operator $C_{s,t}^{\chi,\eta}$ is clearly non-zero since there exist compactly supported locally constant functions on $\EE$ whose restriction to $\FF$ is non-zero. Therefore, its image is $\tau_{\eta,t}$ in all cases where this representation is irreducible. If $\tau_{\eta,t}$ is reducible, it has a unique proper subrepresentation which is either given by $\CC\psi_{\eta,t}$ with $\psi_{\eta,t}(\overline{n}_y)=1$ for all $y\in\FF$ or by the kernel $\tau_{\eta,t}^0$ of the operator $T_{\eta,t}:\tau_{\eta,t}\to\tau_{\eta,-t}$ given by integration over $\FF$. Since there exist compactly supported locally constant functions on $\EE$ whose restriction to $\FF$ is non-constant, the image of $C_{s,t}^{\chi,\eta}$ can never be contained in $\CC\psi_{\eta,t}$. Similarly, there exist compactly supported locally constant functions on $\EE$ whose integral over $\FF$ is non-zero, so the image of $C_{s,t}^{\chi,\eta}$ can neither be $\tau_{\eta,t}^0$.
\end{proof}

Now let us study the images of the operators $\doublewidetilde{A}_{s,t}^{\chi,\eta}$ for $(\chi,s,\eta,t)\in L$. Note that $(\chi,s,\eta,t)\in L$ implies $\chi|_{\calO_\FF^\times}\cdot\eta=\chi|_{\calO_\FF^\times}\cdot\eta^{-1}=1$ and either $t\in\frac{1}{2}+\frac{\pi i}{\ln q_\FF}\ZZ$ or $t\in-\frac{1}{2}+\frac{\pi i}{\ln q_\FF}\ZZ$.

\begin{proposition}
\label{prop::im_A}
	Let $(\chi,s,\eta,t)\in L$.
	\begin{enumerate}
		\item For $t\in\frac{1}{2}+\frac{\pi i}{\ln q_\FF}\ZZ$ we always have $\Im(\doublewidetilde{A}_{s,t}^{\chi,\eta})=\tau_{\eta,t}^0$.
		\item For $t\in-\frac{1}{2}+\frac{\pi i}{\ln q_\FF}\ZZ$ we always have $\Im(\doublewidetilde{A}_{s,t}^{\chi,\eta})=\CC\psi_{\eta,t}$.
	\end{enumerate}
\end{proposition}

\begin{proof}
	\begin{enumerate}
		\item From Lemma~\ref{lem:CompositionWithStdIntertwiner} it follows that
		\begin{align*}
			T_{\eta,t}\circ\doublewidetilde{A}_{s,t}^{\chi,\eta} &= \left.(1-q_\FF^{-2z})^{-1}T_{\eta,t}\circ\widetilde{A}_{s+z,t}^{\chi,\eta}\right|_{z=0}\\
			&= \left.(1-q_\FF^{-2z})^{-1}c_\FF\Gamma(|\cdot|_\FF^{2t})\Gamma((\chi^2)_{-t+\frac{1}{2}}) \cdot \widetilde{A}_{s+z,-t}^{\chi,\eta}\right|_{z=0}.
		\end{align*}
		It can be checked case-by-case using Lemma~\ref{lem:PropertiesGammaFactors} that $\Gamma(|\cdot|_\FF^{2t})\Gamma((\chi^2)_{-t+\frac{1}{2}})=0$ in all cases, so the final expression is equal to zero for all $z$, hence $T_{\eta,t}\circ\doublewidetilde{A}_{s,t}^{\chi,\eta}=0$.
		\item For $t\in-\frac{1}{2}+\frac{\pi i}{\ln q_\FF}\ZZ$ and $(s,t)\in L^{\chi,\eta}$ we always have $\chi^2=1$, $2s-t+\frac{1}{2}=0$ and $t\in-\frac{1}{2}+\frac{2\pi i}{\ln q_\EE}\ZZ$ by Theorem~\ref{thm:ZerosAndSupportOfKernel}. Hence, the kernel $K_{s,t}^{\chi,\eta}$ has the form $K_{s+z,t}^{\chi,\eta}(x)=|x_2|^{2z-2}$. This distribution is clearly independent of $x_1$ for all $z\in\CC$, so the same is true for its value at $z=0$. This shows the claim.\qedhere
	\end{enumerate}
\end{proof}

\section{Application: Discrete components in the restriction of unitary representations}
\label{sec::app}

We recall the unitary structure on the complementary series and the Steinberg representations of $G=\PGL(2,\EE)$. The construction makes use of the invariant non-degenerate bilinear form
$$ \Ind_B^G(\chi_s)\times\Ind_B^G(\overline{\chi}_{-s})\to\CC, \quad (f_1,f_2)\mapsto\int_{G/B}(f_1\cdot f_2)(x)\,d_{G/B}(xB), $$
where $d_{G/B}(xB)$ denotes the invariant integral on $C(G/B,\Omega_{G/B})$ (see Theorem~\ref{thm::invaraint integral}). Using \eqref{eq:IntegralG/PvsNbar}, this form can also be written as an integral over $\overline{N}$:
$$ \int_{G/B}(f_1\cdot f_2)(x)\,d_{G/B}(xB) = \int_{\overline{N}} f_1(\overline{n})f_2(\overline{n})\,d\overline{n}. $$
This motivates the realization of $\pi_{\chi,s}$ on a space of locally constant functions on $\overline{N}\simeq\EE$ by restriction from $G/B$ to the open dense subset $\overline{N}B/B\simeq\EE$: for $f\in\Ind_B^G(\chi_s)$ let $f_\EE(x)=f(\overline{n}_x)$ ($x\in\EE$) and put
$$ I(\chi,s)^\infty = \{f_\EE:f\in\Ind_B^G(\chi_s)\}. $$
Then the invariant bilinear form from above translates to
$$ I(\chi,s)^\infty\times I(\overline{\chi},-s)^\infty\to\CC, \quad (f_1,f_2)\mapsto\int_\EE f_1(x)f_2(x)\,dx. $$
Since this form is non-degenerate, we can identify $I(\chi,s)^\infty$ with a subspace of the dual space
$$ I(\chi,s)^{-\infty} := (I(\overline{\chi},-s)^\infty)^\vee. $$
Note that since $C_c^\infty(\EE)\subseteq I(\chi,s)^\infty$, we can view $I(\chi,s)^{-\infty}$ as a subspace of $\calD'(\EE)$.

Recall from Section~\ref{sec:PSforPGL2} the standard intertwining operators $T_{\chi,s}:\pi_{\chi,s}\to\pi_{\overline{\chi},-s}$. Using \eqref{eq:IntertwinerConvolutionFormula}, we obtain their counterparts on $I(\chi,s)^\infty$:
$$ T_{\chi,s}:I(\chi,s)^\infty\to I(\overline{\chi},-s)^\infty, \quad T_{\chi,s}f(x') = \int_\EE\chi_{s-\frac{1}{2}}(x-x')^2f(x)\,dx. $$
Combining the invariant bilinear form from above with the standard intertwining operator $\widetilde{T}_{\chi,s}$, we obtain an invariant sesquilinear form on $I(\chi,s)^\infty$ whenever $\chi^2=1$ and $s\in\RR$:
$$ \langle f_1,f_2\rangle_{\chi,s} = \int_\EE T_{\chi,s}f_1(x')\overline{f_2(x')}\,dx' = \int_\EE\int_\EE|x-x'|_\EE^{2s-1}f_1(x)\overline{f_2(x')}\,dx\,dx' \qquad (f_1,f_2\in I(\chi,s)^\infty). $$
To determine for which $s\in\RR$ this form is positive definite, we apply the Fourier transform on $\EE$ (see Section~\ref{app:Integral} for its definition and basic properties).

Since the intertwining operator $T_{\chi,s}$ is given by convolution with the character $|\cdot|_\EE^{2s-1}$, and since the Fourier transform of this character equals $\Gamma(|\cdot|_\EE^{2s})|\cdot|_\EE^{-2s}$, we can rewrite the invariant sesquilinear form using the Plancherel formula for the Fourier transform:
$$ \langle f_1,f_2\rangle_{\chi,s} = c_\EE\Gamma(|\cdot|_\EE^{2s})\int_\EE\widehat{f}_1(x)\overline{\widehat{f}_2(x)}|x|_\EE^{-2s}\,dx. $$
For $s\in(0,\frac{1}{2})$, the gamma function is regular (see Lemma~\ref{lem:PropertiesGammaFactors}) and the function $|x|_\EE^{-2s}$ is locally integrable and positive, so the form is positive definite. The completion of $I(\chi,s)^\infty$ with respect to this inner product is the Hilbert space
$$ I(\chi,s) = \{f\in\calD'(\EE):\widehat{f}\in L^2(\EE,|x|_\EE^{-2s}\,dx)\}. $$
The representations $\pi_{\chi,s}$ of $G=\PGL(2,\EE)$ on $I(\chi,s)$ are unitary and irreducible and they form the complementary series of $G$. Note that
$$ I(\chi,s)^\infty \subseteq I(\chi,s) \subseteq I(\chi,s)^{-\infty}. $$

For $s=\frac{1}{2}$, the kernel $I(\chi, {\frac{1}{2}})^\infty_0=\ker T_{\chi,{\frac{1}{2}}}$ of the intertwining operator $T_{\chi,{\frac{1}{2}}}$ is a proper subrepresentation of codimension one, called the Steinberg representation. Note that
$$ I(\chi,\tfrac{1}{2})_0^\infty = \left\{f\in I(\chi,\tfrac{1}{2})^\infty:\int_\EE f(x)\,dx=0\right\}. $$
The invariant form $\langle\cdot,\cdot\rangle_{\chi,\frac{1}{2}}$ vanishes on $I(\chi,\frac{1}{2})^\infty_0$, but we can renormalize it to obtain an inner product
$$ \langle f_1,f_2\rangle'_{\chi,\frac{1}{2}} := \left.\Gamma(|\cdot|_\EE^{2s})^{-1}\langle f_1,f_2\rangle_{\chi,s}\right|_{s=\frac{1}{2}} \qquad (f_1,f_2\in I(\chi,\tfrac{1}{2})^\infty_0). $$
Denote by $I(\chi,\frac{1}{2})_0$ the completion of $I(\chi,\frac{1}{2})_0^\infty$. By the previous computation we have
$$ \langle f_1,f_2\rangle'_{\chi,\frac{1}{2}} = c_\EE\int_\EE\widehat{f}_1(x)\overline{\widehat{f}_2(x)}|x|_\EE^{-1}\,dx. $$
Note that the integral exists since $f\in I(\chi,\frac{1}{2})_0^\infty$ implies $\widehat{f}(0)=0$, so $\widehat{f}$ vanishes in a neighborhood of $0$. It follows that
$$ I(\chi,\tfrac{1}{2})_0 = \{f\in\calD'(\EE):\widehat{f}\in L^2(\EE,|x|_\EE^{-1}\,dx)\}. $$

Denote by $J(\eta,t)^{\pm\infty}$ and $J(\eta,t)$ the corresponding representation spaces for $G'$. For $\eta=\chi|_{\calO_\FF^\times}$ and $t=2s+\frac{1}{2}$ the restriction operator
$$ C_{s,t}^{\chi,\eta}:I(\chi,s)^\infty\to J(\eta,t)^\infty, \quad C_{s,t}^{\chi,\eta}f=f|_\FF, $$ is $G'$-intertwining. We consider its transpose
$$ (C_{s,t}^{\chi,\eta})^\vee:J(\overline{\eta},-t)^{-\infty}\to I(\overline{\chi},-s)^{-\infty}, \quad (C_{s,t}^{\chi,\eta})^\vee f(x_1+x_2\alpha) = f(x_1)\delta(x_2). $$

\begin{theorem}\label{thm:DiscreteComponents}
	Assume that $\chi^2=1$ and put $\eta=\chi|_{\calO_\FF^\times}$.
	\begin{enumerate}
		\item\label{thm:DiscreteComponents1} For $s\in(\frac{1}{4},\frac{1}{2})$ and $t=2s-\frac{1}{2}$, the intertwining operator
		$$ (C_{-s,-t}^{\chi,\eta})^\vee:J(\eta,t)^{-\infty}\to I(\chi,s)^{-\infty} $$
		restricts to an isometric embedding (up to a scalar) $J(\eta,t)\hookrightarrow I(\chi,s)$. In particular, the restriction of the irreducible unitarizable representation $\pi_{\chi,s}$ of $G$ to $G'$ contains $\tau_{\eta,t}$ as a direct summand.
		\item\label{thm:DiscreteComponents2} For $s=t=\frac{1}{2}$, the intertwining operator
		$$ (C_{-s,-t}^{\chi,\eta})^\vee:J(\eta,t)^{-\infty}\to I(\chi,s)^{-\infty} $$
		restricts to an isometric embedding (up to a scalar) $J(\eta,t)_0\hookrightarrow I(\chi,s)_0$. In particular, the restriction of the irreducible unitarizable representation $\pi_{\chi,s}^0$ of $G$ to $G'$ contains $\tau_{\eta,t}^0$ as a direct summand.
	\end{enumerate}
\end{theorem}

\begin{proof}
	We first note that
	$$ \widehat{(C_{-s,-t}^{\chi,\eta})^\vee f}(x_1+x_2\alpha) = \widehat{f}(x_1) \qquad \mbox{for all }f\in J(\eta,t)^{-\infty},x_1,x_2\in\FF. $$
	It follows that
	\begin{align*}
		\|(C_{-s,-t}^{\chi,\eta})^\vee f\|_{\chi,s}^2 &= c_\EE\Gamma(|\cdot|_\EE^{2s})\int_\EE|\widehat{f}(x_1)|^2|x|_\EE^{-2s}\,dx\\
		&= c_\EE\Gamma(|\cdot|_\EE^{2s})\int_\FF|\widehat{f}(x_1)|^2\left(\int_\FF|x_1+x_2\alpha|_\EE^{-2s}\,dx_2\right)\,dx_1.
	\end{align*}
	Substituting $x_2=x_1y$, the inner integral equals
	$$ \int_\FF|x_1+x_2\alpha|_\EE^{-2s}\,dx_2 = |x_1|^{1-4s}_\FF\int_\FF|1+y\alpha|_\EE^{-2s}\,dy = c_\FF\Gamma(|\cdot|_\FF^{4s-1})\Gamma(|\cdot|_\EE^{1-2s})\cdot|x_1|^{1-4s}_\FF $$
	by Proposition~\ref{prop:GammaIntegral}. Note that the integral converges if and only if $s>\frac{1}{4}$. Since $1-4s=-2t$, it follows that
	$$ \|(C_{-s,-t}^{\chi^\vee,\eta^\vee})^\vee f\|_{\chi,s}^2 = c_\EE c_\FF\Gamma(|\cdot|_\FF^{4s-1})\Gamma(|\cdot|_\EE^{1-2s})\Gamma(|\cdot|_\EE^{2s})\int_\FF|\widehat{f}(x_1)|^2|x_1|_\FF^{-2t}\,dx_1 $$
	which is a non-zero multiple of $\|f\|_{\eta,t}^2$ for $s\in(\frac{1}{4},\frac{1}{2})$. This shows \eqref{thm:DiscreteComponents1}. To prove \eqref{thm:DiscreteComponents2}, we renormalize by $\Gamma(|\cdot|_\EE^{2s})$ at $s=\frac{1}{2}$ and find
	$$ \langle(C_{-s,-t}^{\chi^\vee,\eta^\vee})^\vee f,(C_{-s,-t}^{\chi^\vee,\eta^\vee})^\vee f\rangle_{\chi,s}^0 = c_\EE\Gamma(|\cdot|_\FF^{4s-1})\Gamma(|\cdot|_\EE^{1-2s})\langle f,f\rangle_{\eta,t}^0. $$
	By Lemma~\ref{lem:PropertiesGammaFactors}, $\Gamma(|\cdot|_\FF^{4s-1})$ has a single zero at $s=\frac{1}{2}$ while $\Gamma(|\cdot|_\EE^{1-2s})$ has a single pole, so the product is regular and non-zero at $s=\frac{1}{2}$. This shows \eqref{thm:DiscreteComponents2}.
\end{proof}

\appendix

\section{Homogeneous distributions}\label{app:HomogeneousDistributions}

Let $\FF$ be a non-archimedean local field and $q_\FF$ the cardinality of its residue field. We denote by $C_c^\infty(\FF)$ the space of compactly supported locally constant functions on $\FF$ with values in $\CC$. A distribution on $\FF$ is by definition a functional $u: C_c^\infty(\FF) \to \CC$. The space of all distributions on $\FF$ is denoted by $\calD'(\FF)$. Every locally integrable function $u$ on $\FF$ can be identified with the distribution given by
$$ \langle u,f\rangle = \int_\FF u(x)f(x)\,dx \qquad (f\in C_c^\infty(\FF)). $$

For a multiplicative character $\xi$ of $\FF^{\times}$ we say that a distribution $u\in\calD'(\FF)$ is \emph{homogeneous of degree $\xi$} if for any test function $f \in C_c^\infty(\FF)$ and $a \in \FF^{\times}$ we have (see \cite[\S2.3]{GG63}):
$$\langle u,f(a^{-1}\cdot) \rangle= \xi(a)  |a|_\FF \langle u,f\rangle.$$
This means, $u$ satisfies the following equality in the sense of distributions:
$$ u(ax) = \xi(a)u(x) \qquad (a\in\FF^\times,x\in\FF). $$ 

Every multiplicative character $\xi$ of $ \FF^\times$ can uniquely be written as
$$ \xi(a) =: \chi_s(a) = |a|_\FF^s\chi(\varpi_\FF^{-\nu_\FF(a)}a) $$
for some $s\in\CC/\frac{2\pi i}{\ln q_\FF}\ZZ$ and $\chi\in\widehat{\calO_\FF^\times}$. One can prove that $\chi_s$ is locally integrable on $\FF$ if and only if $\Re(s)>-1$ and in this case it defines a homogeneous distribution of degree $\chi_s$ (see \cite[\S2.3]{GG63}). To extend this family of distributions analytically to all $s\in\CC$, we renormalize $\chi_s$ by
$$ \widetilde{\chi}_s(a): = L(1,\chi_s)^{-1}\chi_s(a), $$
where
$$ L(t,\chi_s) = \begin{cases}1&\mbox{for $\chi\neq1$,}\\(1-q_\FF^{-s-t})^{-1}&\mbox{for $\chi=1$.}\end{cases} $$
Then we have the following classification of homogeneous distributions in terms of $\widetilde{\chi}_s$ (see e.g. \cite[\S2.3]{GG63}):

\begin{theorem}\label{thm:HomogeneousDistributions}
	For every multiplicative character $\xi$ of $\FF^\times$ the space of homogeneous distribution of degree $\xi$ is one-dimensional. More precisely, for $\xi=\chi_{s-1}$ it is spanned by $\widetilde{\chi}_{s-1}$. For $s\not\in\frac{2\pi i}{\ln q_\FF}\ZZ$ we have $\supp\widetilde{\chi}_{s-1}=\FF$ and for $s\in\frac{2\pi i}{\ln q_\FF}\ZZ$ we have $\widetilde{\chi}_{s-1}=(1-q_\FF^{-1})\cdot\delta$, where $\delta$ denotes the Dirac distribution at $x=0$.
\end{theorem}

\section{Evaluation of an integral}\label{app:Integral}

For a quadratic extension $\EE/\FF$ of non-archimedean local fields, $\EE=\FF(\alpha)$ with $\alpha^2\in \FF$, and a multiplicative character $\chi$ of $\EE^\times$ we evaluate the integrals
$$ \int_\FF\chi(1+y\alpha)\,dy \qquad  \mbox{and} \qquad \int_\FF\chi(y+\alpha)\,dy, $$
where $dy$ denotes a Haar measure on the additive group $\FF$.

The computation involves the Fourier transform on both $\FF$ and $\EE$. For this, we use the Haar measure $dx=dx_1\,dx_2$ ($x=x_1+x_2\alpha$, $x_1,x_2\in\FF$) on $\EE$ which is the product of two copies of the same fixed Haar measure $dy$ on $\FF$. Fix a non-trivial additive character $\psi$ of $\FF$ and extend it to $\EE$ by
$$ \psi(x_1+x_2\alpha) := \psi(x_1) \qquad (x_1,x_2\in\FF). $$
The character $\psi$ can be used to define the Fourier transform $\calF f=\widehat{f}$ of functions $f\in C_c^\infty(\EE)$ and $f\in C_c^\infty(\FF)$:
$$ \widehat{f}(x) := \int_\EE\psi(xy)f(y)\,dy \quad (f\in C_c^\infty(\EE)) \qquad  \qquad \widehat{f}(x_1): = \int_\FF\psi(x_1y_1)f(y_1)\,dy_1 \quad (f\in C_c^\infty(\FF)). $$
The Fourier transform $\calF$ gives bijections $C_c^\infty(\EE)\to C_c^\infty(\EE)$ and $C_c^\infty(\FF)\to C_c^\infty(\FF)$, respectively (see \cite[\S2.4]{GG63}). Its inverse is given by the inversion formula: there exist constants $c_\EE,c_\FF>0$ such that
\begin{align*}
	f(x) &= c_\EE\int_\EE\psi(-xy)\widehat{f}(y)\,dy \quad &&(f\in C_c^\infty(\EE)) \qquad \mbox{and}\\
	f(x_1) &= c_\FF\int_\FF\psi(-x_1y_1)\widehat{f}(y_1)\,dy_1 &&(f\in C_c^\infty(\FF)).
\end{align*}
With these definitions, the Fourier transform satisfies the convolution formula
$$ \widehat{f*g}=\widehat{f}\cdot\widehat{g}, \qquad \mbox{where }(f*g)(x)=\int_\EE f(y)g(x-y)\,dy $$
and the Plancherel formula
$$ \int_\EE f(x)\overline{g(x)}\,dx = c_\EE\int_\EE \widehat{f}(y)\overline{\widehat{g}(y)}\,dy, $$
and similar for $\FF$.

Using the symplectic form
$$\omega(x_1+x_2\alpha,y_1+y_2\alpha):=x_1y_2-x_2y_1 \qquad (x_1,x_2,y_1,y_2\in\FF) $$ 
on the two-dimensional vector space $\EE$ over $\FF$, we further define the \emph{symplectic Fourier transform} $\Fsymp f$ of $f\in C_c^\infty(\EE)$ by
$$ \Fsymp f(x): = \int_\EE\psi(\omega(x,y))f(y)\,dy. $$
The symplectic Fourier transform appears naturally when studying the convolution of functions on the Heisenberg group in terms of the Fourier transform in the central variable (see e.g. \cite[\S XII.3.3]{Stein}). Since $\Fsymp f(x_1+x_2\alpha)=\calF f(-x_2+x_1\alpha^{-1})$, the symplectic Fourier transform also maps $C_c^\infty(\EE)$ bijectively onto itself. Note that $\omega(ax,y)=\omega(x,\overline{a}y)$ for all $a,x,y\in \EE$, where $\overline{a_1+a_2\alpha}=a_1-a_2\alpha$.

All Fourier transforms can be extended uniquely to distributions by duality. More precisely, for $u,f\in C_c^\infty(\EE)$ we have
$$ \int_\EE \widehat{u}(x)f(x)\,dx = \int_\EE u(y)\widehat{f}(y)\,dy, $$
so we define for $u\in\calD'(\EE)$:
$$ \langle\widehat{u},f\rangle := \langle u,\widehat{f}\rangle \qquad (f\in C_c^\infty(\EE)). $$
Similar formulas extend the Fourier transform on $\FF$ and the symplectic Fourier transform on $\EE$ to distributions. Then for $x\in \FF$ we have
$$ \int_\FF\psi(xy)\,dy = c_\FF^{-1}\cdot\delta(x) $$
in the sense of distributions.

By \cite[\S2.5]{GG63} and Theorem \ref{thm:HomogeneousDistributions}, the Fourier transform of a multiplicative character $\chi(x)|x|_{\EE}^{-1}$ of $\EE^\times$, which defines a homogeneous distribution of degree $\chi(x)|x|_{\EE}^{-1}$ if $\chi\neq1$, is a homogeneous distribution of degree $\chi(x)^{-1}$ to within a factor, called the gamma factor $\Gamma(\chi)$. The gamma factor $\Gamma(\chi)$ takes the following integral form:
$$ \Gamma(\chi) = \lim_{n\to\infty}\int_{q_{\EE}^{-n}\leq|x|_{\EE}\leq q_{\EE}^n}\psi(x)\chi(x)|x|_{\EE}^{-1}\,dx, $$
where $q_\EE$ denotes the order of the residue field of $\calO_\EE$. Then
$$ \widehat{\chi}(x) = \Gamma(\chi|\cdot|_{\EE})\cdot|x|_{\EE}^{-1}\chi(x)^{-1} $$
as distributions. From \cite[\S2.5]{GG63} we further recall:

\begin{lemma}\label{lem:PropertiesGammaFactors}
	For a multiplicative character $\chi$ of $\calO_\EE^{\times}$ and $s\in\CC$ let $\chi_s$ be the multiplicative character of $\EE^\times$ given by $\chi_s(x):=\chi(\varpi_\EE^{-\nu_\EE(x)}x)|x|_\EE^s$ ($x\in\EE^\times$).
	\begin{enumerate}[(i)]
		\item If $\chi\neq1$, then $\Gamma(\chi_s)$ is holomorphic in $s\in\CC$ and nowhere vanishing.
		\item If $\chi=1$, then $\Gamma(\chi_s)$ is meromorphic in $s\in\CC$ with simple poles at $s \in \frac{2\pi i}{\ln q_\EE}\ZZ$ and simple zeros at $s \in 1+ \frac{2\pi i}{\ln q_\EE}\ZZ$ and no further singularities or zeros.
	\end{enumerate}
\end{lemma}

The relevance of the symplectic Fourier transform is its relation to the integrals we would like to evaluate:

\begin{lemma}
\label{lem::Fsym_formula}
For $\varphi\in C_c^\infty(\EE)$, the following identity holds:
$$ \int_\FF\varphi(1+y\alpha)\,dy = c_\FF\int_\FF\psi(r)\Fsymp\varphi(r\alpha)\,dr. $$
\end{lemma}

\begin{proof}
Indeed, we can compute as follows (in the distribution sense):
\begin{align*}
	\int_\FF\psi(r)\Fsymp\varphi(r\alpha)\,dr &= \int_\FF\int_\EE\psi(r)\psi(-ry_1)\varphi(y)\,dy\,dr = \int_\EE\varphi(y)\int_\FF\psi(r(1-y_1))\,dr\,dy\\
	&= c_\FF^{-1}\int_\EE\varphi(y)\delta(1-y_1)\,dy = c_\FF^{-1}\int_\FF\varphi(1+y_2\alpha)\,dy_2.\qedhere
\end{align*}
\end{proof}

To be able to use the formula from Lemma \ref{lem::Fsym_formula}, we need to compute $\Fsymp\chi$ for a multipicative character $\chi$ of $\EE^\times$.

\begin{lemma}\label{lem:Fsym_character}
The symplectic Fourier transform of a multiplicative character $\chi$ of $\EE^\times$ is given by
$$ \Fsymp\chi(x) = |\alpha|_\EE\chi(\alpha)\Gamma(\chi|\cdot|_\EE)\cdot|\overline{x}|_\EE^{-1}\chi(\overline{x})^{-1}. $$
\end{lemma}

\begin{proof}
For $x\in \EE$ and $a\in \EE^\times$ we have
\begin{multline*}
	\Fsymp\chi(ax) = \int_\EE\psi(\omega(ax,y))\chi(y)\,dy = \int_\EE\psi(\omega(x,\overline{a}y))\chi(y)\,dy\\
	= |\overline{a}|_\EE^{-1}\int_\EE\psi(\omega(x,y))\chi(\overline{a}^{-1}y)\,dy = |\overline{a}|_\EE^{-1}\chi(\overline{a})^{-1}\Fsymp\chi(x).
\end{multline*}
Thus, $\Fsymp\chi$ is again a homogeneous distribution of degree $|\overline{x}|_\EE^{-1}\chi(\overline{x})^{-1}$.  It follows that
$$ \Fsymp\chi(x) = c_\chi|\overline{x}|_\EE^{-1}\chi(\overline{x})^{-1} $$
for some constant $c_\chi\in\CC$. To determine $c_\chi$, put $x=-\alpha$, then
\begin{equation*}
	c_\chi|\alpha|_\EE^{-1}\chi(\alpha)^{-1} = \Fsymp\chi(-\alpha) = \int_\EE\psi(y_1)\chi(y)\,dy = \int_\EE\psi(y)\chi(y)\,dy = \Gamma(\chi|\cdot|_\EE).\qedhere
\end{equation*}
\end{proof}

\begin{proposition}\label{prop:GammaIntegral}
For a multiplicative character $\chi$ of $\EE^\times$ we have
\begin{align*}
	\int_\FF\chi(1+y\alpha)\,dy &= c_\FF\chi(-1)\Gamma(\chi^{-1}|_{\FF^\times}|\cdot|_\FF^{-1})\Gamma(\chi|\cdot|_\EE),\\
	\int_\FF\chi(y+\alpha)\,dy &= c_\FF\chi(-\alpha)|\alpha^2|_\FF\Gamma(\chi^{-1}|_{\FF^\times}|\cdot|_\FF^{-1})\Gamma(\chi|\cdot|_\EE).
\end{align*}
\end{proposition}

\begin{proof}
First note that by the coordinate change $y=z\alpha^2$ we have
$$ \int_\FF\chi(y+\alpha)\,dy = \chi(\alpha)|\alpha^2|_\FF\int_\FF\chi(1+z\alpha)\,dz, $$
so it suffices to compute the first integral.
By the previous lemmata \ref{lem::Fsym_formula} and \ref{lem:Fsym_character}:
\begin{align*}
	\int_\FF\chi(1+y\alpha)\,dy &= c_\FF\int_\FF\psi(r)\Fsymp\chi(r\alpha)\,dr\\
	&= c_\FF|\alpha|_\EE\chi(\alpha)\Gamma(\chi|\cdot|_\EE)\int_\FF\psi(r)|-r\alpha|_\EE^{-1}\chi(-r\alpha)^{-1}\,dr\\
	&= c_\FF\chi(-1)\Gamma(\chi|\cdot|_\EE)\int_\FF\psi(r)|r|_\FF^{-2}\chi(r)^{-1}\,dr.
\end{align*}
The remaining integral equals $\Gamma(\chi^{-1}|_{\FF^\times}|\cdot|_\FF^{-1})$.
\end{proof}


\providecommand{\bysame}{\leavevmode\hbox to3em{\hrulefill}\thinspace}
\providecommand{\MR}{\relax\ifhmode\unskip\space\fi MR }
\providecommand{\MRhref}[2]{%
  \href{http://www.ams.org/mathscinet-getitem?mr=#1}{#2}
}
\providecommand{\href}[2]{#2}

\end{document}